\documentclass[aos,preprint]{imsart}

\RequirePackage[OT1]{fontenc}
\RequirePackage{amsthm,amsmath}
\RequirePackage{bbm}
\RequirePackage[numbers]{natbib}
\RequirePackage[colorlinks,citecolor=blue,urlcolor=blue]{hyperref}

\arxiv{arXiv:0000.0000}

\startlocaldefs
\numberwithin{equation}{section}
\theoremstyle{plain}

\endlocaldefs

\RequirePackage{enumerate,enumitem}
\usepackage{siunitx}
\usepackage{natbib}
\usepackage{multirow}
\usepackage{graphicx}
\usepackage{morefloats}
\usepackage{wrapfig}
\usepackage{savesym}
\usepackage{subfigure}
\usepackage{hyperref}
\usepackage{enumerate}
\usepackage{geometry}
\usepackage{amsmath,amssymb,amsthm}
\numberwithin{equation}{section}

\newtheorem{theorem}{Theorem}[section]
\newtheorem{lemma}[theorem]{Lemma}

\newtheorem{proposition}[theorem]{Proposition}

\newtheorem{assumption}[theorem]{Assumption}

\newenvironment{remark}[1][Remark]{\begin{trivlist}
\item[\hskip \labelsep {\bfseries #1}]}{\end{trivlist}}

\newcommand{\ignore}[1]{}

\begin{document}

\newcommand*{\tran}{^{\mkern-1.5mu\mathsf{T}}}
\newcommand{\reff}[1]{(\ref{#1})}
\newcommand{\sF}{\mathcal{F}}
\newcommand{\sG}{\mathcal{G}}
\newcommand{\sK}{\mathcal{K}}
\newcommand{\sL}{\mathcal{L}}
\newcommand{\sI}{\mathcal{I}}
\newcommand{\sT}{\mathcal{T}}
\newcommand{\sN}{\mathcal{N}}
\newcommand{\sH}{\mathcal{H}}
\newcommand{\IE}{\mathbb{E}}
\newcommand{\IZ}{\mathbb{Z}}
\newcommand{\IN}{\mathbb{N}}
\newcommand{\IR}{\mathbb{R}}
\newcommand{\Ii}{\mathbf{1}}
\newcommand{\IP}{\mathbb{P}}

\def\argmin{\mathop{\rm argmin}}
\def\argmax{\mathop{\rm argmax}}

\begin{frontmatter}

\title{Simultaneous inference for time-varying models}
\runtitle{Simultaneous inference for time-varying models}

\begin{aug}

  \author{\fnms{Sayar} \snm{Karmakar$^{1}$}\ead[label=e1]{sayarkarmakar@ufl.edu}},
  \author{\fnms{Stefan}  \snm{Richter$^{2}$}\ead[label=e2]{stefan.richter@iwr.uni-heidelberg.de}}
  \and
  \author{\fnms{Wei Biao}  \snm{Wu$^{3}$}\ead[label=e3]{wbwu@galton.uchicago.edu}}

  \runauthor{S. Karmakar et al.}

  \affiliation{University of Florida, Heidelberg University and University of Chicago}

\end{aug}

\begin{abstract}

A general class of non-stationary time series is considered in this paper. We estimate the time-varying coefficients by using local linear M-estimation. For these estimators, weak Bahadur representations are obtained and are used to construct simultaneous confidence bands. For practical implementation, we propose a bootstrap based method to circumvent the slow logarithmic convergence of the theoretical simultaneous bands. Our results substantially generalize and unify the treatments for several time-varying regression and auto-regression models. The performance for tvARCH and tvGARCH models is studied in simulations and a few real-life applications of our study are presented through the analysis of some popular financial datasets.
\end{abstract}

\begin{keyword}
\kwd{Time-varying regression}
\kwd{Time-series models} 
\kwd{Generalized linear models}
\kwd{Simultaneous confidence band} 
\kwd{Gaussian approximation}
\kwd{Bootstrap}
\end{keyword}

\end{frontmatter}

\section{Introduction}\label{sec:introduction}

Time-varying dynamical systems have been studied extensively in the literature of statistics, economics and related fields. For stochastic processes observed over a long time horizon, stationarity is often an over-simplified assumption that ignores systematic deviations of parameters from constancy. For example, in the context of financial datasets, empirical evidence shows that external factors such as war, terrorist attacks, economic crisis, some political event etc. introduce such parameter inconstancy. As \citet{bai97} points out, `failure to take into account parameter changes, given their presence, may lead to incorrect policy implications and predictions'. Thus functional estimation of unknown parameter curves using time-varying models has become an important research topic recently. In this paper, we propose a general setting for simultaneous inference of local linear M-estimators in semi-parametric time-varying models. Our formulation is general enough to allow unifying time-varying models from the usual linear regression, generalized regression and several auto-regression type models together. Before discussing our new contributions in this paper, we provide a brief overview of some previous works in these areas. 

In the regression context, time-varying models are discussed over the past two decades to describe non-constant relationships between the response and the predictors; see, for instance, \citet{fan99}, \citet{fan2000}, \citet{hoover98}, \citet{huang04}, \citet{lin01}, \citet{ramsay05}, \citet{zhang02} among others.  Consider the following two regression models
$$\text{Model I: } y_i = x_i\tran \theta_i +e_i, \quad \text{Model II: } y_i = x_i\tran\theta_0 + e_i, \quad\quad i = 1, \ldots ,n,$$

\noindent where $x_i \in \IR^d$ ($i = 1,\ldots,n$) are the covariates, $\tran$ is the transpose, $\theta_0$ and $\theta_i = \theta(i/n)$ are the regression coefficients. Here, $\theta_0\in\IR^d$ is a constant parameter and $\theta : [0, 1] \to  \mathbb{R}^d$ is a smooth function. Estimation of $\theta(\cdot)$ has been considered by \citet{hoover98}, \citet{cai07}) and \citet{zhouwu10} among others. Hypothesis testing is widely used to choose between model I and model II, see, for instance, \citet{regression2012}, \citet{regression2015}, \citet{chow60}, \citet{brown75}, \citet{nabeya88},  \citet{leybourne89}, \citet{nyblom89}, \citet{ploberger89}, \citet{andrews93} and \citet{lin99}. \citet{zhouwu10} discussed obtaining  simultaneous confidence bands (SCB) in model I, i.e. with additive errors. However their treatment is heavily based on the closed-form solution and it does not extend to processes defined by a more general recursion. Little has been known for time-varying models in this direction previously. 

The results from time-varying linear regression can be naturally extended to time-varying AR, MA or ARMA processes. However, such an extension is not obvious for conditional heteroscedastic (CH) models. These are difficult to estimate but also often more useful in analyzing and predicting financial datasets. Since \citet{engle82} introduced the classical ARCH model and \citet{bollerslev} extended it to a more general GARCH model, these have remained primary tools for analyzing and forecasting certain trends for stock market datasets. As the market is vulnerable to frequent changes, non-uniformity across time is a natural phenomenon. The necessity of extending these classical models to a set-up where the parameters can change across time has been pointed out in several references; for example \citet{stuaricua2005}, \citet{engle05} and \citet{fry08}. Towards time-varying parameter models in the CH setting, numerous works discussed the CUSUM-type procedure, for instance, \citet{kim00} for testing for changes in the parameters of a  GARCH(1,1) time series. \citet{kulperger05} studied the high moment partial sum process based on residuals and applied it to residual CUSUM tests in GARCH models. Interested readers can find some more change–point detection results in the context of CH models in \citet{chu95}, \citet{chen97}, \citet{linbook99}, \citet{kokoszka00} or \citet{andreou06}.

Historically in the analysis of financial datasets, the common practice to account for the time-varying nature of the parameter curves was to transfer a stationary tool/method in some ad hoc way. For example, in \citet{mikosch2004}, the authors analyzed S\&P500 data from 1953-1990 and suggested that time-varying parameters are more suitable due to such a long time-horizon. They re-estimated the parameters for every block of 100 sample points and to account for the abrupt fluctuation of the coefficients, they generated re-estimates of parameters for samples of size $100,200, \ldots.$ This treatment suffers from different degree of reliability of the estimators at different parts of the time horizon. There are examples outside the analysis of economic datasets, where similar approach of splitting the time-horizon has been adapted to fit CH type models. For example, in \citet{giaco12}, the authors analyzed Italian mortality rates from 1960-2003 using an AR(1)-ARCH(1) model and observed abrupt behavior of yearwise coefficients. Our framework can simultaneously capture these models and provide significant improvements over such heuristic treatments.

A time-varying framework and a pointwise curve estimation using M-estimators for locally stationary ARCH models was provided by \citet{dahlhaussubbarao2006}. Since then, while several pointwise approaches were discussed in the tvARMA and tvARCH case (cf. \citet{dahlhaus2009}, \citet{dahlhaussubbarao2006}, \citet{fry08}), pointwise theoretical results for estimation in tvGARCH processes were discussed in \citet{tvgarch2013} and \citet{rohan13} for GARCH(1,1) and GARCH($p$,$q$) models. Even though the conditional heteroscedastic model remained widely popular in analyzing many different types of econometric data, the topic of simultaneous inference in this field remains relatively untouched. Consider the simple tvARCH(1) model
$$X_i= \sigma_i\zeta_i,\quad  \zeta_i \sim N(0,1),\quad  \sigma_i^2= \alpha_0(i/n)+\alpha_1(i/n)X_{i-1}^2.$$
Typically, it is considered that for large number of realizations, the corresponding parameters $\alpha_0, \alpha_1$ vary smoothly over time and can be modeled as smooth functions $\alpha_0,\alpha_1:[0,1] \to \IR$. Pointwise confidence bands for these functions do not help to infer about their overall pattern (like testing for constancy or some specific parametric form). While one remedy could be to subjectively assume a certain class of functions for $\alpha_0,\alpha_1$ such as linear or polynomial and perform a hypothesis test, this can be problematic for many real life datasets. See for example the intercept function for the USGBP analysis in Section \ref{sec:simureal}. We rather take an objective approach where we do \emph{not} assume any parametric form as such and wish to establish valid simultaneous inference. In this paper, we therefore derive \emph{simultaneous} confidence bands which cover $\alpha_0, \alpha_1$ over the \emph{whole time interval} $t \in (0,1)$ with a given confidence. After construction, one can perform many hypothesis tests such as time-constancy, linearity etc. in one go.
To the best of our knowledge, no theoretical results for simultaneous confidence intervals for nonstationary time series were derived before this work.

We next summarize our contributions in this paper.
We use Bahadur representations, a Gaussian approximation theorem from \citet{zhouwu09} and extreme value theory for Gaussian processes to obtain simultaneous confidence bands for contrasts of  parameter curves in very general time-varying models. These intervals provide a generalization from testing parameter constancy to testing any particular parametric form such as linear, quadratic, exponential etc. To deal with bias expansions, we use a theory for locally stationary processes which was recently formalized in \citet{dahlhaus2017}. 

Moving on to some practical applicability of our results, we show how our result applies to time-varying ARCH and GARCH models. 
For tv(G)ARCH models, we improve the existing conditions in \cite{subbarao2008} (we only need that the innovation process has $4+a$ moments for some $a > 0$ compared to 8 moments needed therein) for constructing  confidence intervals and provide \emph{simultaneous} instead of pointwise confidence intervals. We provide an empirical justification of how the coverage can be significantly improved by a wild bootstrap technique and use Gaussian approximation theory to theoretically establish it. Finally we also provide some data analysis and volatility forecasting. First we show for numerous real-life datasets that the time-varying fit does better than the time-constant ones in short-range forecasts. This  underlines the importance to decide whether a constant or a time-varying model should be used. One interesting find from our analysis is that  simultaneous inference can lead to models where a subset is time-varying and these semi-time varying model can sometimes achieve both statistical confidence and better forecasting ability.




The rest of the article is organized as follows. In Section \ref{sec:model}, we state two specific classes of time series models and the related assumptions. For the sake of better focus and readability, we decided to narrow down the scope of the paper to these specific models. However our theoretical results of M-estimation and the SCBs allow to treat much more general models. The more general assumptions are given in the Appendix (cf. Assumption \ref{ass1} therein). In Section \ref{sec:main results} we provide our main results, namely a Bahadur representation of the estimators of the parameter functions and a SCB result for the related contrasts. Section \ref{sec:implementation} is dedicated to practical issues which arise when using the SCBs, like estimation of the dispersion matrix of the estimator, bandwidth selection and a wild Bootstrap procedure to overcome the slow logarithmic convergence from the theoretical SCB.
Some summarized simulation studies and real data applications  can be found in Section \ref{sec:simureal}. The proofs of the main results are deferred to Appendix, while the proof of several more elementary lemmata and a more general assumption set for tvGARCH processes can be found in the Supplementary material.

\section{Model assumptions and estimators}\label{sec:model}

\subsection{The model}\label{ssc:themdoel}
Suppose that $\zeta_i$, $i\in\IZ$ is a sequence of i.i.d. random variables. We consider the following two time series models. In both cases, $\Theta$ denotes a parameter space specified below in Section \ref{sec:assum_new}.
\begin{itemize}
    \item Case 1: Recursively defined time series. Suppose that for $i,\ldots,n$,
\begin{equation}
    Y_i = \mu(Y_{i-1},...,Y_{i-p}, \theta(i/n)) + \sigma(Y_{i-1},...,Y_{i-p}, \theta(i/n)) \zeta_i,\label{eq:tvrec_model}
\end{equation}
where $\theta = (\alpha_1,\ldots,\alpha_k,\beta_0,\ldots,\beta_l)\tran:[0,1]\to \Theta \subset \IR^{k+l+1}$ and
\begin{eqnarray*}
    \mu(x,\theta) := \sum_{i=1}^{k}\alpha_i m_i(x),\quad\quad \sigma(x,\theta) := \big(\sum_{i=0}^{l}\beta_i \nu_i(x)\big)^{1/2},
\end{eqnarray*}
with some functions $m_i:\IR^p \to \IR$, $\nu_i:\IR^p \to \IR_{\ge 0}$. Put $X_i^c = (Y_{i-1},\ldots,Y_{1 \vee (i-p)},0,\ldots)\tran$.

This model covers, for instance, tvARMA and tvARCH models.
    \item Case 2: tvGARCH. For $i = 1,\ldots,n$, consider the recursion
\begin{eqnarray*}
    Y_i &=& \sigma_i^2 \zeta_i^2,\\
    \sigma_i^2 &=& \alpha_0(i/n) + \sum_{j=1}^{m}\alpha_j(i/n) Y_{i-j} + \sum_{j=1}^{l}\beta_{j}(i/n)\sigma_{i-j}^2,
\end{eqnarray*}
where $\theta = (\alpha_0,\ldots,\alpha_m,\beta_1,\ldots,\beta_l) : [0,1] \to \Theta \subset \IR^{m+l+1}$. Put $X_i^c := (Y_{i-1},...,Y_{1},0,0,...)$.

\end{itemize}

\noindent Case 1 does not directly cover the tvGARCH model, we therefore operate with it separately as Case 2 throughout the paper. In either case, our goal is to estimate $\theta(\cdot)$ from the observations $Z_i^c = (Y_i, X_i^c)$, $i = 1,\ldots,n$.

\subsection{The estimator}\label{section_estimator} In this paper, we focus on local M-estimation: Let $K(\cdot)\in \sK$, where 
$\sK$ is the family of non-negative symmetric kernels with support $[-1,1]$ which are continuously differentiable on $[-1,1]$ such that $\int_{-1}^{1} |K'(u)|^2 d u > 0$. We consider as objective function $\ell(z,\theta)$ the negative conditional Gaussian likelihood. This reads
\begin{itemize}
    \item in Case 1:
    \[
    \ell(y,x,\theta) = \frac{1}{2}\Big[ \Big(\frac{y-\mu(x,\theta)}{\sigma(x,\theta)}\Big)^2 + \log \sigma(x,\theta)^2\Big],
\]
    \item in Case 2:
    \[
    \ell(y,x,\theta)=\frac{1}{2}\Big[ \frac{y}{\sigma(x,\theta)^2}+\log(\sigma(x,\theta)^2)\Big],
\]
where here, $\sigma(x,\theta)^2$ is recursively defined via $\sigma(x,\theta)^2 = \alpha_0 + \sum_{j=1}^{m}\alpha_j x_j + \sum_{j=1}^{l}\beta_j \sigma(x_{j\rightarrow},\theta)^2$ and $x_{j\rightarrow}:= (x_{j+1},x_{j+2},\ldots)$.
\end{itemize}
For some bandwidth $b_n > 0$, define the local linear likelihood function
\begin{equation}
    L_{n,b_n}^c(t,\theta,\theta') := (n b_n)^{-1}\sum_{i=1}^{n}K_{b_n}(t-i/n) \ell(Z_i^c, \theta + \theta'\cdot  (i/n-t)),\label{eq:likelihood}
\end{equation}
where $K_{b_n}(\cdot) := K(\cdot/b_n)$. Let $\Theta' := [-R,R]^k$ with some $R > 0$. A local linear estimator of $\theta(t)$, $\theta'(t)$ is given by
\begin{equation}
    (\hat \theta_{b_n}(t),\widehat{\theta'}_{b_n}(t)) = \argmin_{(\theta,\theta') \in \Theta\times\Theta'} L_{n,b_n}^c(t,\theta,\theta'), \quad\quad t \in [0,1].\label{eq:estimator}
\end{equation}

\begin{remark}
As defined above, we consider a quite  specific form of the objective function $\ell$. In the Appendix, we allow $\ell$ to be much more general. Basically, it has to be twice continuously differentiable and 'compatible' with the time series model. A referee asked if also the  differentiability assumption on $\ell$ might be relaxed. A relaxation might be possible by using sharper and more recent Gaussian approximation results from \cite{karmakarsinica} and empirical process results for dependent data. However, this would significantly increase the complexity on the assumptions on $\ell$, since its smoothness is used for several completely different key steps in the proofs, such as Bahadur representations, a bias expansion and the quantification of the underlying dependence.
\end{remark}

\subsection{Assumptions}\label{sec:assum_new}

For our main results, we need the following assumptions on our time series models.

\begin{assumption}[Case 1]\label{ass:case1} Assume that
\begin{enumerate}
    \item $\zeta_i$ are i.i.d. with $\IE \zeta_i = 0$, $\IE \zeta_i^2 = 1$ and for some $a > 0$, $\IE |\zeta_i|^{(2+a)M} < \infty$. Here, $M=3$. In the special case $\sigma(x,\theta)^2 \equiv \beta_0$, one can choose $M = 2$.
    \item For all $t \in [0,1]$, the sets
    \[
        \{m_1(\tilde X_0(t)), \ldots, m_k(\tilde X_0(t))\}, \quad\quad \{\nu_0(\tilde X_0(t)), \ldots, \nu_l(\tilde X_0(t))\}
    \]
    are (separately) linearly independent in $L^2(\IP)$.
    \item There exist $(\kappa_{ij}) \in \IR_{\ge 0}^{k\times p}$, $(\rho_{ij}) \in \IR_{\ge 0}^{(l+1)\times p}$ such that for all $i$:
    \begin{equation}
        \sup_{x\not=x'}\frac{|m_i(x) - m_i(x')|}{|x-x'|_{\kappa_{i\cdot},1}} \le 1, \quad\quad  \sup_{x\not=x'}\frac{|\sqrt{\nu_i(x)} - \sqrt{\nu_i(x')}|}{|x-x'|_{\rho_{i\cdot},1}} \le 1.\label{example:tvrec_eq1}
    \end{equation}
    Let $\nu_{min} > 0$ be some constant such that for all $x\in\IR$, $\nu_0(x) \ge \nu_{min}$. With some $\beta_{min} > 0$, choose $\tilde \Theta \subset \IR^{k} \times \IR_{\ge \beta_{min}}^{l+1}$ such that for all $\theta \in \tilde \Theta$,
    \begin{equation}
        \sum_{j=1}^{p}\Big(\sum_{i=1}^{k}|\alpha_i| \kappa_{ij} + \|\zeta_0\|_{2M}\cdot \sum_{i=0}^{l}\sqrt{\beta_i}\rho_{ij}\Big) < 1.\label{example:parameterconditions}
    \end{equation}
    \item $\Theta \subset \tilde \Theta$ is compact and for all $t \in [0,1]$, $\theta(t)$ lies in the interior of $\Theta$. Each component of $\theta(\cdot)$ is in $C^3([0,1])$.
\end{enumerate}
\end{assumption}

\noindent For the tvAR($k$) model (cf. \cite{richterdahlhaus2017}, Example 4.1), one may choose $p = k$, $m_1(x) = x_1$, ..., $m_k(x) = x_k$, $l = 0$, $\nu_0(x) = 1$, leading to the rather strong condition $\sum_{i=1}^{k}|\alpha_i| < 1$ in \reff{example:parameterconditions}. However, as it can be seen in the proof of Proposition \ref{example:tvrec} in the appendix, the condition \reff{example:parameterconditions} is only needed to guarantee the existence of the process and corresponding moments. By using techniques which are more specific to the model, one can obtain much less strict assumptions such as $\Theta$ being a compact subset of
\[
    \{\theta = (\alpha_1,...,\alpha_k,\beta_0) \in \IR^{k} \times (0,\infty): \alpha(z) = 1 + \sum_{i=1}^{k}\alpha_i z^i \text{ has only zeros outside the unit circle}\},
\]
cf. \cite{richterdahlhaus2017}, Example 4.1. In the tvARCH case, the above Assumption \ref{ass:case1} asks for $\IE |\zeta_1|^{6+a} < \infty$ with some $a > 0$. 

In the following, we consider Case 2, the tvGARCH model. In this specific model, the moment conditions can be relaxed to $\IE|\zeta_1|^{4+a} < \infty$. The tvGARCH model was for instance studied in the stationary case in \citet{garch2004}. More recently, pointwise asymptotic results were obtained in \citet{tvgarch2013}. For a matrix $A$, we define $\|A\|_q := (\|A_{ij}\|_q)_{ij}$ as a component-wise application of $\|\cdot\|_q$. For matrices $A,B$, let $A \otimes B$ denote the Kronecker product and 
\begin{equation}\label{eq:kronecker} 
A^{\otimes k}=A \otimes \ldots \otimes A    
\end{equation}
\noindent denote the $k$-fold Kronecker product. Let $\rho(A)$ denote the spectral norm of $A$.

\begin{assumption}[Case 2]\label{ass:case2}
 Let $ f(\theta) = (\alpha_1,\ldots,\alpha_m,\beta_1,\ldots,\beta_l)\tran$ and let $e_j = (0,\ldots,0,1,0,\ldots,0)\tran$ be the unit column vector with $j$th element being 1, $1 \le j \le l+m$. Define $M_i(\theta) = (f(\theta)\zeta_i^2, e_1,\ldots,e_{m-1},f(\theta),e_{m+1},\ldots,e_{m+l-1})\tran$. Let $\alpha_{min} > 0$ and $\tilde \Theta \subset \IR_{\ge \alpha_{min}}\times \IR_{>0}^{m+l}$ such that for all $\theta,\theta' \in \tilde \Theta$,
\begin{equation}
    \rho(\IE[ M_0(\theta)\otimes M_0(\theta')]) < 1.\label{ass:case2_parametercondition}
\end{equation}
Suppose that
\begin{enumerate}
    \item[(i)] $\Theta \subset \tilde \Theta$ is compact and for all $t \in [0,1]$, $\theta(t)$ lies in the interior of $\Theta$. Each component of $\theta(\cdot)$ is in $C^3[0,1]$,
    \item[(ii)] $\zeta_i$ are i.i.d. with $\IE \zeta_i = 0$, $\IE \zeta_i^2 = 1$ and $\IE |\zeta_i|^{4+a} < \infty$ with some $a > 0$.
\end{enumerate}
\end{assumption}

\noindent In the important GARCH(1,1) case, a straightforward calculation shows that the condition \reff{ass:case2_parametercondition} can be translated to 
\begin{equation}
    \rho(\IE[M_0(\theta)^{\otimes 2}]) < 1\quad \Longleftrightarrow\quad \beta_1^2 + 2 \alpha_1 \beta_1 + \alpha_1^2 \|\zeta_0\|_4^4 < 1 \label{eq:discussionmomentgarch}
\end{equation}
If $\zeta_0 \sim N(0,1)$, it holds that $\|\zeta_0\|_4^2 = \sqrt{3} \approx 1.73$. \citet{bollerslev} proved that stationary GARCH(1,1)  processes have 4th moments under the exact same condition (\ref{eq:discussionmomentgarch}). In Section \ref{sec:implementation}, Remark \ref{rem:garch} therein, we further talk about the applicability of \reff{eq:discussionmomentgarch}.

We \emph{conjecture} that also for general GARCH($l,m$) models, \reff{ass:case2_parametercondition} is equivalent to the condition
\[
    \text{ for all }\theta \in \tilde \Theta: \quad\quad \rho(\IE[M_0(\theta)^{\otimes 2}]) < 1,
\]
which would then exactly meet the condition from  \cite{bollerslev}. Note that estimation and the true curve $\theta(\cdot)$ lie in $\Theta$ which is has to be a compact subset of $\tilde \Theta$. Therefore, we automatically ask that \emph{all} parameters of the GARCH process are nonzero. Again, this condition could in principle be relaxed which would add a significant amount of technicalities.


\section{Main results}\label{sec:main results} We discuss the theoretical confidence band result in this section. 
We directly start with a weak Bahadur representation which plays a key role for introducing simultaneity.  For $l \ge 0$, define
\[
    \mu_{K,l} := \int K(x) x^l dx, \quad\quad\sigma_{K,l}^2 := \int K(x)^2 x^l dx.
\]
We now have to define some quantities $V(t), I(t), \Lambda(t)$ which are needed to provide the theoretical results. They correspond to the so-called (miss-specified) Fisher information matrices which occur naturally as variance of the M-estimators. These quantities need not to be known in practice because they are estimated. They depend on the so-called \emph{stationary approximation} $\tilde Y_i(t)$ of the considered time-varying process $Y_i$. In case 1 and case 2, this is given as follows: For $t \in [0,1]$,
\begin{itemize}
    \item $\tilde Y_i(t)$ is the solution of
    \[
        \tilde Y_i(t) = \mu(\tilde Y_{i-1}(t),...,\tilde Y_{i-p}(t),\theta(t)) + \sigma(\tilde Y_{i-1}(t),...,\tilde Y_{i-p}(t),\theta(t)),\quad i\in\IZ,
    \]
    \item $\tilde Y_i(t)$ is the solution of
    \begin{eqnarray*}
        \tilde Y_i(t) &=& \tilde\sigma_i(t)^2 \zeta_i^2,\\
    \tilde\sigma_i(t)^2 &=& \alpha_0(t) + \sum_{j=1}^{m}\alpha_j(t) \tilde Y_{i-j}(t) + \sum_{j=1}^{l}\beta_{j}(t)\tilde\sigma_{i-j}(t)^2, \quad i\in\IZ.
    \end{eqnarray*}
\end{itemize}

For $t\in [0,1]$, let $\tilde Z_j(t) := (\tilde Y_j(t),\tilde Y_{j-1}(t),...)$ denote the infinite vector containing the stationary approximations. We now define
\begin{eqnarray}
        V(t) &=& \IE \nabla_{\theta}^2 \ell(\tilde Z_0(t), \theta(t)),\label{eq:vdef_paper}\\
        I(t) &=& \IE[\nabla_{\theta} \ell(\tilde Z_0(t), \theta(t))\cdot \nabla_{\theta} \ell(\tilde Z_0(t), \theta(t))\tran],\label{eq:idef_paper}\\
        \Lambda(t) &=& \sum_{j\in\IZ}\IE[\nabla_{\theta} \ell(\tilde Z_0(t), \theta(t))\cdot \nabla_{\theta} \ell(\tilde Z_j(t), \theta(t))\tran].\label{eq:lambdadef_paper}
\end{eqnarray}

In our theoretical models, these quantities can be related to each other. The following lemma (a direct implication of Propositions \ref{example:tvrec} and  \ref{example:garch} in the appendix) summarizes these forms.
\begin{lemma}\label{lemma:simpler_forms}
    \begin{itemize}
        \item Case 1: It holds that $V(t) = \Lambda(t)$.\\
        If additionally (i) $\IE \zeta_0^3 = 0$, or (ii) $\mu(x,\theta) \equiv 0$ or (iii) $\sigma(x,\theta) \equiv \beta_0$ and $\IE m(\tilde X_0(t)) = 0$, then $$I(t) = \big(\begin{smallmatrix}I_k & 0\\
0 & (\IE \zeta_0^4 - 1) I_{l+1}/2\end{smallmatrix}\big)\cdot V(t),$$ where $I_d$ denotes the $d$-dimensional identity matrix.
        \item Case 2: It holds that $\Lambda(t) = I(t) = ((\IE \zeta_0^4 - 1)/2)V(t)$.
    \end{itemize}
\end{lemma}

\subsection{A weak Bahadur representation for \texorpdfstring{$\hat \theta_{b_n}$}{thetaest}}\label{ssc:bahadur}


In the following, we obtain a weak Bahadur representation of $\hat \theta_{b_n}$
which will be used to construct simultaneous confidence bands.  The first part of Theorem \ref{theorem:bahadur_paper} shows that $\hat \theta_{b_n}(t) - \theta(t)$ can be approximated by the expression $V(t)^{-1}\nabla_{\theta}L_{n,b_n}^c(t,\theta(t),\theta'(t))$ as expected due to a standard Taylor argument. The second part of Theorem \ref{theorem:bahadur_paper} deals with approximating this term by a weighted sum of $t$-free terms, namely
\[
    (nb_n)^{-1}\sum_{i=1}^{n}K_{b_n}(i/n-t)h_i, \quad\quad h_i := \nabla_{\theta}\ell(\tilde Z_i(i/n),\theta(i/n)),
\]
which is necessary to apply some earlier results from \citet{zhouwu10}. 
Let $\sT_n := [b_n, 1-b_n]$. For some vector or matrix $x$, let $|x|:= |x|_2$ denote its Euclidean or Frobenius norm, respectively.

\begin{theorem}[Weak Bahadur representation of $\hat \theta_{b_n}$]\label{theorem:bahadur_paper}
Let $\beta_n = (n b_n)^{-1/2}b_n^{-1/2}\log(n)^{1/2}$ and put
    \[
        \tau_n^{(1)} = (\beta_n + b_n)( (nb_n)^{-1/2}\log(n) + b_n^{2}).
    \]
    Let Assumption \ref{ass:case1} or \ref{ass:case2} hold. Then it holds that
        \begin{eqnarray}
        && \sup_{t \in \sT_n}\Big| V(t)\cdot\big\{\hat \theta_{b_n}(t) - \theta(t)\big\} -\nabla_{\theta} L_{n,b_n}^c(t,\theta(t),\theta'(t))\Big| = O_{\IP}(\tau_n^{(1)}),\label{eq:bahadur_paper}\\
        &&\sup_{t\in \sT_n}\big|\nabla_{\theta} L_{n,b_n}^c(t,\theta(t),\theta'(t)) - b_n^2\frac{\mu_{K,2}}{2}V(t) \theta''(t)\label{eq:bahadur_localstat_paper}\\
        &&\quad\quad\quad\quad\quad\quad - (n b_n)^{-1}\sum_{i=1}^{n}K_{b_n}(i/n-t)h_i \big| = O_{\IP}(\beta_n b_n^2 + b_n^3 + (nb_n)^{-1}).\nonumber
    \end{eqnarray}
\end{theorem}

\subsection{Simultaneous confidence bands for \texorpdfstring{$\hat \theta_{b_n}$}{thetaest}}

Based on the weak Bahadur result, we use results from \citet{MR2827528} to obtain a Gaussian analogue of
\[
\frac{1}{n b_n}\sum_{i=1}^{n}K_{b_n}(t-i/n) C\tran V(t)^{-1} \nabla_{\theta}\ell(\tilde Z_i(i/n), \theta(i/n)) =:\frac{1}{n b_n}\sum_{i=1}^{n}K_{b_n}(t-i/n)\tilde h_i(i/n)
\]
for some $C \in \IR^{s \times k}$. For a positive semidefinite matrix $A$ with eigendecomposition $A = QDQ\tran$, where $Q$ is orthonormal and $D$ is a diagonal matrix, define $A^{1/2} = Q D^{1/2}Q\tran$, where $D^{1/2}$ is the elementwise root of $D$. Then the following asymptotic statement for simultaneous confidence bands for $\theta(\cdot)$ holds.

\begin{theorem}[Simultaneous confidence bands for $\theta(\cdot)$]\label{th:scb_paper}
Let $C$ be a fixed $k \times s$ matrix with rank $s \le k$. Define $\hat \theta_{b_n,C}(t) := C\tran \hat \theta_{b_n}(t)$ 
and $\theta_C(t) := C\tran \theta(t)$, $A_C(t) := V(t)^{-1} C$, $\Sigma_C^2(t) := A_C\tran(t) \Lambda(t) A_C(t)$.

Let Assumption \ref{ass:case1} or \ref{ass:case2} be fulfilled. Assume that, for some  $\alpha_{exp}<\frac{1}{2}$,
\[
    \log(n)^4 \big( b_n n^{\alpha_{exp}}\big)^{-1} \to 0, \quad\quad n b_n^7 \log(n) \to 0.
\]
Then with $\hat K(x) = K(x) x$,
\begin{eqnarray}
    &&\lim_{n\to\infty}\IP\Big( \frac{\sqrt{n b_n}}{\sigma_{K,0}} \sup_{t \in \sT_n}\Big| \Sigma_C^{-1}(t)\Big\{ \hat \theta_{b_n,C}(t) - \theta_C(t) - b_n^2\frac{\mu_{K,2}}{2} \theta_C''(t)\Big\}\Big| \nonumber\\
    &&\quad\quad\quad\quad\quad\quad\quad\quad\quad\quad- B_K(m^{*}) \le \frac{u}{\sqrt{2 \log(m^{*})}}\Big) = \exp(-2 \exp(-u)), \label{eq:scb_paper}
\end{eqnarray}
where in both cases $\sT_n = [b_n,1-b_n]$, $m^{*} = 1/b_n$ and
\begin{equation}\label{eq:Bk}
    B_K(m^{*}) = \sqrt{2\log(m^{*})} + \frac{\log(C_K) + (s/2-1/2)\log(\log(m^{*})) - \log(2)}{\sqrt{2\log(m^{*})}},
\end{equation}
with
\[
    C_K = \frac{\Big\{\int_{-1}^{1}|K'(u)|^2 d u / \sigma_{K,0}^2 \pi\Big\}^{1/2}}{\Gamma(s/2)}.
\]
\end{theorem}

\begin{remark} The conditions on $b_n$ are fulfilled for bandwidths $b_n = n^{-\alpha}$, where $\alpha \in (0,1)$ satisfies
\[
    \frac{1}{7} < \alpha < \alpha_{exp}.
\]
\noindent The bandwidths $b_n = c n^{-1/5}$ are covered in both cases. 
\end{remark}

\noindent Note that for practical use of the SCB in \reff{eq:scb_paper}, one needs to estimate the bias term, choose a proper bandwidth $b_n$ and estimate $\Sigma_C(t)$. Furthermore, the theoretical SCB only has slow logarithmic convergence, thus one requires huge $n$ to achieve the desired coverage probability. To tackle these type of problems, we discuss practical issues in the next Section \ref{sec:implementation}.

\section{Implementational issues}\label{sec:implementation}

In this section, we discuss some issues which arise by implementing the procedure from Theorem \ref{th:scb_paper}. We focus on estimation of $\hat \theta_{b_n}$ and optimization of the corresponding SCBs.

\subsection{Bias correction} \label{ssc:bias correction} There are several possible ways to eliminate the bias term in Theorem \ref{th:scb_paper}. A natural way is to estimate $\theta''(t)$ by using a local quadratic estimation routine with some bandwidth $b_n' \ge b_n$. However the estimation of $\theta''(t)$ may be unstable due to the convergence condition $nb_n^5 \to \infty$ which may be hard to realize together with $n b_n^7 \log(n) \to 0$ from Theorem \ref{th:scb_paper} in practice. Here instead we propose a bias correction via the a jack-knife method inspired from \cite{hardle1986}. We define 
\begin{eqnarray}\label{eq:bias corrected theta hat}
    \tilde \theta_{b_n}(t) := 2\hat \theta_{b_n/\sqrt{2}}(t) - \hat \theta_{b_n}(t).
\end{eqnarray}
Since the weak Bahadur representation from Theorem \ref{theorem:bahadur_paper} holds both for $\hat \theta_{b_n/\sqrt{2}}$ and $\hat \theta_{b_n}(t)$, we obtain
\[
    \sup_{t\in \sT_n}\big|V(t)\cdot \{\tilde \theta_{b_n}(t) - \theta(t)\} - (nb_n)^{-1}\sum_{i=1}^{n}\tilde K_{b_n}(i/n-t)h_i\big| = O_{\IP}(\tau_n^{(2)} + \beta_n b_n^2 + b_n^3 + (nb_n)^{-1}),
\]
where $\tilde K(x) := 2\sqrt{2}K(\sqrt{2}x) - K(x)$. Note that the bias term of order $b_n^2$ is eliminated by construction. This shows that Theorem \ref{th:scb_paper} still holds true for $\tilde \theta_{b_n}(\cdot)$ with kernel $K$ replaced by the fourth-order kernel $\tilde K$ and with no bias term of order $b_n^2$.

\subsection{Estimation of the covariance matrix  \texorpdfstring{$\Sigma_C(t)$}{SigmaC(t)}}\label{ssc:covariance estimation}
In this subsection, we discuss the estimation  of $\Sigma_C^2(t)$ since this term is generally unknown but arises in the SCB in Theorem \ref{th:scb_paper}. By Lemma \ref{lemma:simpler_forms}, one has in both cases that $\Lambda(t) = I(t)$, which shows that
\begin{equation}
        \Sigma_C^2(t) = C \tran (V(t)^{-1})\tran\Lambda(t) V(t)^{-1} C = C \tran (V(t)^{-1})\tran I(t) V(t)^{-1} C.\label{eq:sigma_estimation1}
\end{equation}
As pointed out by a referee, $\Sigma_C^2(t)$ is of the  well-known 'sandwich'-form (cf. \cite{bollerslevwooldridge1992}). Even if the distribution of the innovations $\zeta_i$ is misspecified by the likelihood, one can typically simplify the representation  \reff{eq:sigma_estimation1}. If the distribution of $\zeta_0$ is correctly specified, one has $V(t) = I(t)$ and thus
\begin{equation}
    \Sigma_C^2(t) = C\tran I(t)^{-1} C.\label{eq:sigma_estimation2}
\end{equation}
If the distribution of $\zeta_0$ is misspecified by the likelihood, Lemma \ref{lemma:simpler_forms} shows that $V(t) = c_0\cdot I(t)$ with some constant matrix $c_0$ which only depends on the fourth moment $\IE[\zeta_0^4]$ of $\zeta_0$. Then one has
\begin{equation}
    \Sigma_C^2(t) = c_0\cdot C\tran I(t)^{-1} C\label{eq:sigma_estimation3}
\end{equation}
This also means that the representation \reff{eq:sigma_estimation2} is stable under misspecification of the innovation distribution as long as one corrects the expression with the factor $c_0$. To do so, one needs a possibility to estimate $\IE[\zeta_0^4]$ from the data. A possibility how to do this for linear processes was discussed in \cite{bootstrap4cumulant}.

In summary, the representation \reff{eq:sigma_estimation1} holds always true, the simpler representations \reff{eq:sigma_estimation2} and \reff{eq:sigma_estimation3} can be used under additional assumptions or if stable estimators of $c_0$ are available.

To cover all possible situations above, we discuss both estimation of $V(t)$ and $I(t)$. We propose the (boundary-corrected) estimators
\begin{eqnarray}
    \hat V_{b_n}(t) &:=& (nb_n \hat \mu_{K,0,b_n}(t))^{-1}\sum_{i=1}^{n}K_{b_n}(i/n-t) \nabla_{\theta}^2\ell(Z_i^c, \hat \theta_{b_n}(t) + (i/n-t)\widehat \theta_{b_n}'(t)),\label{proposal_est_V}\\
    \hat I_{b_n}(t) &:=& (nb_n\hat \mu_{K,0,b_n}(t))^{-1}\sum_{i=1}^{n}K_{b_n}(i/n-t) \nabla_{\theta}\ell(Z_i^c, \hat \theta_{b_n}(t) + (i/n-t)\widehat \theta_{b_n}'(t))\label{proposal_est_I}\\
    &&\quad\quad\quad\quad\quad\quad\quad\quad\quad\quad\times \nabla_{\theta}\ell(Z_i^c, \hat \theta_{b_n}(t) + (i/n-t)\widehat \theta_{b_n}'(t))\tran,\nonumber
\end{eqnarray}
where $\hat \mu_{K,0,b_n}(t) := \int_{-t/b_n}^{(1-t)/b_n} K(x) dx$. The convergence of these estimators is given in the next Proposition. Note that the following Proposition also holds if $\widehat \theta_{b_n}'$ in \reff{proposal_est_V} and \reff{proposal_est_I} is replaced by $0$.
\begin{proposition}\label{prop:v_i_estimation}
    Let Assumption \ref{ass:case1} or \ref{ass:case2} hold. Let $(\beta_n + b_n)\log(n)^2 \to 0$. Then
    \begin{enumerate}
        \item[(i)] $\sup_{t\in (0,1)}|\hat V_{b_n}(t) - V(t)| = O_{\IP}((\log n)^{-1}).$
        \item[(ii)] If $r > 4$, then $\sup_{t\in (0,1)}|\hat I_{b_n}(t) - I(t)| = O_{\IP}((\log n)^{-1}).$
    \end{enumerate}
\end{proposition}

\noindent This shows uniform consistency of $\hat V_{b_n}(\cdot)$, $\hat I_{b_n}(\cdot)$ if $(\beta_n + b_n)\log(n)^2 \to 0$. Note that in (ii), we need more moments to discuss $\nabla_{\theta}\ell \cdot \nabla_{\theta}\ell\tran \in \sH(2M_y,2M_x,\chi,\bar{\bar C})$ ($\bar{\bar C} > 0$). In many special cases, this may be relaxed.

In either case \reff{eq:sigma_estimation1} or \reff{eq:sigma_estimation2}, we define $\hat \Sigma_{C}(t)$ by replacing $V(t),I(t)$ by the corresponding estimators $\hat V_{b_n}(t)$, $\hat I_{b_n}(t)$.

\subsection{Bandwidth selection}\label{ssc:bandwidth} Based on the asymptotic squared error decomposition
\[
    \big|\hat \theta_{b_n,C}(t) - \theta_{C}(t)\big| \approx \Big|\frac{b_n^2}{2}\mu_{K,2}\theta''_C(t)\Big|^2 + \frac{\sigma_{K,0}^2}{nb_n}\mathrm{tr}(\Sigma_C(t)),
\]
which can be read off the weak Bahadur representation \reff{eq:scb_paper}, the squared error global optimal bandwidth choice reads
\begin{equation}
    \hat b_n = n^{-1/5}\cdot \Big(\frac{\sigma_{K,0}^2 \int_0^{1}\text{tr}(C \tran V(t)^{-1}I(t) V(t)^{-1} C) dt}{\mu_{K,2}^2 \int_0^{1} |\theta_C''(t)|^2 dt}\Big)^{1/5}.\label{optimalbandwidthchoice}
\end{equation}

\noindent In practice, $\hat b_n$ is not available due to the unknown quantities on the right hand side, in particular $\theta''(t)$. We therefore adapt a model-based cross validation method from \citet{richterdahlhaus2017}, which was shown to work even if the underlying parameter curve is only H{\" o}lder continuous and $\nabla_{\theta}\ell(\tilde Z_i(t),\theta(t))$ is uncorrelated. Here, we reformulate this selection procedure for the local linear setting. For $j = 1,\ldots,n$, define the leave-one-out local linear likelihood
\begin{equation}
    L_{n,b_n,-j}^c(t,\theta,\theta') := (nb_n)^{-1}\sum_{i=1, i\not= j}^{n}K_{b_n}(i/n-t) \ell(Z_i^c, \theta + (i/n-t)\theta')\label{eq:cv1}
\end{equation}
and the corresponding leave-one-out estimator \[
    (\hat \theta_{b_n,-j}(t), \hat \theta_{b_n,-j}'(t)) = \argmin_{\theta \in \Theta, \theta' \in \Theta'} L_{n,b_n,-j}^c(t,\theta,\theta').
\]
The bandwidth $\hat b_n^{CV}$ is chosen via minimizing
\begin{equation}
    CV(b_n) := n^{-1}\sum_{i=1}^{n}\ell(Z_i^c, \hat \theta_{b_n,-i}(i/n)) w(i/n),\label{eq:cv2}
\end{equation}
where $w(\cdot)$ is some weight function to exclude boundary effects. A possible choice is $w(\cdot) := \Ii_{[\gamma_0,1-\gamma_0]}$ with some fixed $\gamma_0 > 0$. Note that it is important to use  the modified local linear approach due to the different bias terms. In  \citet{richterdahlhaus2017}, it was shown that the local constant version of this procedure selects asymptotically optimal bandwidths and works even if a model misspecification is present, i.e. if the function $\ell$ leads to estimators $\hat \theta_{b_n}$ which are not consistent. This motivates that a similar behavior should hold for the local constant version.


\subsection{Bootstrap method}\label{ssc:bootstrap}
The SCB for $\theta_C(t)$ obtained in Theorem \ref{th:scb_paper} provides a slow logarithmic rate of convergence to the Gumbel distribution. Thus, even for  moderately large values of sample size $n$, it is practically infeasible to use such a theoretical SCB as the coverage will possibly be lower than the specified nominal level. First we show an empirical coverage comparison of how far the theoretical confidence intervals lag behind in achieving their nominal coverage. We use the same simulation setting (cf. Section \ref{ssc:simulations}) for the tvGARCH case: 

 $$X_i = \sigma_i \zeta_i, \sigma_i^2 = \alpha_0(i/n) + \alpha_1(i/n) X_{i-1}^2 + \beta_1(i/n) \sigma_{i-1}^2,$$ where $\alpha_0(t) = 1.0 + 0.2 \sin(2\pi t)$, $\alpha_1(t) = 0.45 + 0.1 \sin(\pi t)$ and $\beta_1(t) = 0.1 + 0.1\sin(\pi t)$, $\zeta_i$ is i.i.d. standard normal distributed. For estimation, we choose $K(x) = \frac{3}{4}(1 - x^2)\Ii_{[-1,1]}(x)$ to be the Epanechnikov kernel, $n = 2000, 5000 $ for several different $b_n$. From Table \ref{tab:data_garchnonboot} one can see that the simultaneous coverage is never even positive for the SCB specified in Theorem \ref{th:scb_paper}. The individual coverages are very low for small bandwidth and with higher bandwidth they over-compensate. The performance for $n= 5000$ observations is slightly better, hinting at the logarithmic rate of convergence in Theorem \ref{th:scb_paper}.

\begin{table}[!ht]
\caption{Coverage probabilities of the SCB in (b) for $n = 2000, 5000$, $b_n=0.25, 0.3, 0.35$}
\centering
\label{tab:data_garchnonboot}
\begin{tabular}{| l | l || c | c | c | c || c | c | c | c |}
\hline
& &  \multicolumn{4}{c||}{$\alpha = 90\%$} &  \multicolumn{4}{c|}{$\alpha = 95\%$}\\
 $n$ & $b_n$ & $\alpha_0$ & $\alpha_1$ & $\beta_1$ & $(\alpha_0,\alpha_1,\beta_1)\tran$ & $\alpha_0$ & $\alpha_1$ & $\beta_1$ & $(\alpha_0,\alpha_1,\beta_1)\tran$ \\
\hline\hline
2000 & $0.3$ Gumbel & 0.778 & 0.405 & 0.638 & 0 & 0.629 & 0.215 & 0.42  & 0 \\
\hline
& $0.3$ Bootstrap & 0.944 & 0.822 & 0.924 & 0.869 & 0.967 & 0.89 & 0.961  & 0.92 \\
\hline
 & $0.35$ Gumbel & 0.954 & 0.891 & 0.94 & 0 & 0.871 & 0.714  & 0.818 & 0 \\
\hline
 & $0.35$ Bootstrap & 0.941 & 0.843 & 0.923 & 0.864 & 0.967 & 0.908 & 0.945  & 0.913 \\
\hline
\hline
5000 & $0.25$ Gumbel & 0.6 & 0.248 & 0.481 & 0 & 0.481 & 0.126 & 0.298 & 0 \\
\hline
 & $0.25$ Bootstrap & 0.955 & 0.855 & 0.936 & 0.886 & 0.974 & 0.903 & 0.959 & 0.929 \\
\hline
 & $0.30$ Gumbel & 0.756 & 0.461 & 0.642 & 0 & 0.643 & 0.261 & 0.432 & 0 \\
\hline
 & $0.30$ Bootstrap & 0.941 & 0.889 & 0.939 & 0.904 & 0.974 & 0.931 & 0.967 & 0.949 \\
\hline
 & $0.35$ Gumbel & 0.96 & 0.917 & 0.957 & 0 & 0.878 & 0.701 & 0.84 & 0 \\
\hline
 & $0.35$ Bootstrap & 0.949 & 0.903 & 0.941 & 0.903 & 0.972 & 0.95 & 0.975 & 0.946 \\
\hline
\end{tabular}
\end{table}

We circumvent this convergence issue in this subsection by proposing a wild bootstrap algorithm. Recall the jackknife-based bias corrected estimator $\tilde \theta_{b_n}$ from (\ref{eq:bias corrected theta hat}). Let $\tilde{\theta}_C(t)=C^T \tilde{\theta}_{b_n}(t)$. We have the following proposition as the key idea behind the bootstrap method. 

\begin{proposition}\label{prop:prop1}
Suppose that Assumption \ref{ass:case1} or Assumption \ref{ass:case2} holds. Furthermore, assume that $b_n =O(n^{-\kappa})$ with $1/7 < \kappa < \frac{1}{2}$. Then on a richer probability space, there are i.i.d. $V_1,V_2,\ldots,\sim N(0, Id_s)$ such that
\begin{eqnarray}\label{eq:prop1}
    \sup_{t \in \sT_n} |\hat \theta_{b_n,C}(t)-\theta_C(t)- \Sigma_C(t) Q_{b_n}^{(0)}(t)|=O_{\IP}\big(\frac{n^{-\nu}}{\sqrt{n b_n} \log(n)^{1/2}}\big),
\end{eqnarray}

\noindent where $\nu = \min\{\frac{1}{4}-\kappa/2, 7\kappa/2-1/2,\kappa/2\} > 0$ and $$Q_{b_n}^{(0)}(t)=\frac{1}{nb_n}\sum_{i=1}^n V_i K_{b_n}(i/n-t).$$
\end{proposition}

\noindent The proof of Proposition \ref{prop:prop1} is immediate from the approximation rates \reff{eq:localscb1}, \reff{eq:localscb2}, \reff{eq:proofbahadur1} and \reff{eq:applicatbahadur1} in the appendix which, ignoring the $\log(n)$ terms, are of the form $c_n \cdot (nb_n)^{-1/2}\log(n)^{-1/2}$ with 
\[
    c_n \in \{\big(b_n n^{(2\gamma+\varsigma \gamma-\varsigma)/(\varsigma+4\gamma+2\gamma \varsigma)}\big)^{-1/2}, b_n^{1/2}, b_n, (n b_n^{7})^{1/2}, (n b_n^2)^{-1/2}\},
\]
where $\gamma > 1$ is arbitrarily large and $\varsigma > 0$.

\noindent One can interpret \reff{eq:prop1} in the sense that $\Sigma_C(t)Q_{b_n}^{(0)}(t)$ approximates the stochastic variation in $\hat \theta_{b_n,C}(t)-\theta_C(t)$ uniformly over $t \in \sT_n$. Thus it can be used as a margin for the noise to construct confidence bands, provided one can consistently estimate $\Sigma_C(t)$.

\subsubsection{Boundary considerations} The results shown above only hold for $t \in \sT_n$. For inference of some time series models like ARCH or GARCH, large bandwidths are needed to get sufficiently smooth and stable estimators even for a large number of observations. It seems hard to generalize the SCB result Theorem \ref{th:scb_paper} to the whole interval $t \in (0,1)$. However it is possible to generalize the bootstrap procedure which may be more important in practice:

\begin{proposition}\label{prop:prop2}
Suppose that the conditions on $\kappa,\nu$ of Proposition \ref{prop:prop1} hold. Then on a richer probability space, there exist i.i.d. $V_1,V_2,\ldots,\sim N(0, Id_s)$ such that
\begin{eqnarray*}
    &&\sup_{t \in (0,1)} |N_{b_n}^{(0)}(t)\cdot \big\{\hat{\theta}_{b_n,C}(t)-\theta_C(t)\big\} + b_n^2 N_{b_n}^{(1)}(t)\theta_{C}''(t) -  \Sigma_C(t) W_{b_n}(t)|=O_{\IP}\big(\frac{n^{-\nu}}{\sqrt{n b_n} \log(n)^{1/2}}\big),
\end{eqnarray*}

\noindent where
\begin{equation}
 W_{b_n}(t)= Q_{b_n}^{(0)}(t) - \frac{\hat \mu_{K,1,b_n}(t)}{\hat \mu_{K,2,b_n}(t)} \cdot Q_{b_n}^{(1)}(t)\label{eq:boundary_corrected_bootstrap}
\end{equation}
and $N_{b_n}^{(j)}(t) := \frac{\hat \mu_{K,j,b_n}(t) \hat \mu_{K,j+2,b_n}(t) - \hat \mu_{K,j+1,b_n}(t)^2}{\hat \mu_{K,2,b_n}(t)}$, $\hat \mu_{K,j,b_n}(t) := \int_{-t/b_n}^{(1-t)/b_n}K(x) x^j dx$,
\[
    Q_{b_n}^{(j)}(t)=\frac{1}{nb_n}\sum_{i=1}^n V_i K_{b_n}(i/n-t)\big[(i/n-t)b_n^{-1}\big]^{j}, \quad\quad  (j = 0,1).
\]
\end{proposition}

Note that the additional term in \reff{eq:boundary_corrected_bootstrap} reduces to $Q_{b_n}^{(0)}(t)$ for $t \in \sT_n$.

To eliminate the bias inside $t \in \sT_n$, it is still recommended to use the jack-knife estimator $\tilde \theta_C(t)$. From Proposition \ref{prop:prop2} we obtain
\begin{eqnarray}
    &&\sup_{t\in(0,1)}\big|N_{b_n}^{(0)}(t) N_{b_n/\sqrt{2}}^{(0)}(t)\big\{\tilde \theta_C(t) - \theta(t)\big\} + b_n^2\big\{N_{b_n/\sqrt{2}}^{(1)}(t)N_{b_n}^{(0)}(t)-N_{b_n}^{(1)}(t)N_{b_n/\sqrt{2}}^{(0)}(t)\big\}\theta_C''(t)\nonumber\\
    &&\quad\quad\quad\quad\quad\quad\quad\quad - \Sigma_C(t)W_{b_n}^{(debias)}(t)\big| = O_{\IP}\big(\frac{n^{-\nu}}{\sqrt{n b_n} \log(n)^{1/2}}\big),\label{eq:boundary_corrected_bootstrap2}
\end{eqnarray}
where
\begin{equation}
    W_{b_n}^{(debias)}(t) = 2 N_{b_n}^{(0)}(t)\cdot \big[ Q_{b_n/\sqrt{2}}^{(0)}(t) - \frac{\hat \mu_{K,1,b_n/\sqrt{2}}(t)}{\hat \mu_{K,2,b_n/\sqrt{2}}(t)} Q_{b_n/\sqrt{2}}^{(1)}(t)\big] - N_{b_n/\sqrt{2}}^{(0)}(t)\cdot \big[ Q_{b_n}^{(0)}(t) - \frac{\hat \mu_{K,1,b_n}(t)}{\hat \mu_{K,2,b_n}(t)} Q_{b_n}^{(1)}(t)\big].\label{eq:bootstrap_corrected_empirical_qunatile_generator}
\end{equation}
The additional factor $N_{b_n}^{(0)}(t) N_{b_n/\sqrt{2}}^{(0)}(t)$ in \reff{eq:boundary_corrected_bootstrap2} serves as an indicator how near $t$ is to the boundary. For $t \in \sT_n$, this factor is 1 while for $t \in (0,1)\backslash \sT_n$, $N_{b_n}^{(0)}(t) N_{b_n/\sqrt{2}}^{(0)}(t)$ may be very small, inducing large diameters of the band near the boundary.
Note that the bias correction of the jack-knife estimator $\tilde \theta_C(t)$ may be useless in $t \in (0,1) \backslash \sT_n$ since $N_{b_n/\sqrt{2}}^{(1)}(t)N_{b_n}^{(0)}(t)\not= N_{b_n}^{(1)}(t)N_{b_n/\sqrt{2}}^{(0)}(t)$. However it is necessary from a theoretical point of view to use the same estimator for the whole region $(0,1)$ to get a uniform band based on the approximation \reff{eq:boundary_corrected_bootstrap2}.

In practice, the result \reff{eq:boundary_corrected_bootstrap2} can be used as follows: We can create a large number of i.i.d. copies $W_{b_n}^{(boot,debias)}(t)$ of $W_{b_n}^{(debias)}(t)$ by creating i.i.d. copies
\begin{eqnarray}\label{eq:boot}
\\    
  \nonumber Q_{b_n}^{(0),boot}(t)=\frac{1}{nb_n}\sum_{i=1}^n V_i^* K_{b_n}(i/n-t), \quad\quad Q_{b_n}^{(1),boot} \frac{1}{nb_n}\sum_{i=1}^n V_i^* K_{b_n}(i/n-t)\cdot (i/n-t)b_n^{-1}
\end{eqnarray}

\noindent where $V_1^*, V_2^*,\ldots ,$  are i.i.d. $N(0, I_{s \times s})$-distributed random variables, and computing $W_{b_n}^{(boot,debias)}(t)$ according to \reff{eq:bootstrap_corrected_empirical_qunatile_generator}. Quantiles of $W_{b_n}^{(debias)}(t)$ then can be determined by using the corresponding empirical quantile of the copies $W_{b_n}^{(boot,debias)}(t)$. Then one can use \reff{eq:boundary_corrected_bootstrap2} to construct the confidence band for $\theta_C(t)$. For convenience of the readers, we provide a summarized algorithm of the above discussion.

\textbf{Algorithm for constructing SCBs of $\theta_C(t)$:}
\begin{itemize}
\item Compute the appropriate bandwidth $b_n$ based on the cross validation method in Subsection \ref{ssc:bandwidth} and compute $\tilde{\theta}_{C}(t)$ based on the jackknife-based estimator from \ref{ssc:bias correction}.
\item For $r = 1,\ldots,N$ with some large $N$, generate $n$ i.i.d. $N(0,I_{s\times s})$ random variables $V_1^{*},\ldots,V_n^{*}$ and compute $q_r=\sup_{t \in (0,1)}|W_{b_n}^{(boot,debias)}(t)|$, where $W_{b_n}^{(boot,debias)}(t)$ is computed according to \reff{eq:bootstrap_corrected_empirical_qunatile_generator}, \reff{eq:boot}.
\item Compute $u_{1-\alpha} = q_{\lfloor (1-\alpha) N\rfloor}$, the empirical $(1-\alpha)$th quantile of $\sup_{t\in[0,1]}|W_{b_n}^{(debias)}(t)|$. 
\item Calculate $\hat{\Sigma}_C(t) = \{ C^T \hat{V}(t)^{-1} \hat{\Lambda}(t) \hat{V}(t)^{-1} C\} ^{1/2}$ with the estimators proposed in Subsection \ref{ssc:covariance estimation}. As mentioned there,  $V(t)^{-1}\Lambda(t)V(t)^{-1}$ can often be simplified. 
\item The SCB for $\theta_C(t)$ is $\tilde \theta_{C,b_n}(t)+\hat{\Sigma}_C(t) u_{1-\alpha} \mathcal{B}_s$, where $\mathcal{B}_s = \{x\in \IR^{s}: |x| \le 1\}$ is the unit ball in $\mathbb{R}^s$.
\end{itemize}

\begin{remark}\label{rem:garch}(Discussion of the tvGARCH parameter restriction) A very valid question was asked by a reviewer about the applicability of the assumption \reff{eq:discussionmomentgarch}. We would like to point out that this assumption is necessary under the fourth moment assumption of the GARCH process. Investigating the proof of Proposition \ref{example:garch} very minutely, it seems that it might be possible to relax the existence of $4+a$ moments for the GARCH process to only $2+a$ moments which could potentially improve the condition \reff{eq:discussionmomentgarch} to $$ \alpha_1(\cdot)+\beta_1(\cdot) <1.$$ However, the entire bias expansion arguments in the proof of Theorem \ref{th:scb_paper}  would change based on this relaxed moment assumption and it would require a different notion of local-stationarity that allows more approximating terms. To keep the general theme of the paper, we decided against proving a separate result for just GARCH(1,1). Moreover, from a practical point of view, when we estimate $\Sigma_C(t)$, we use $I_{b_n}(t)$ from section 4 which is only consistent under at least 4th moment existence of the GARCH process. We also found that, for some very popular stock market datasets (one such example is given in Section \ref{sec:simureal}) one can reasonably assume that the condition \reff{eq:discussionmomentgarch} is satisfied.   
\end{remark}

\section{Simulation results and applications}\label{sec:simureal}
This section consists of some summarized simulations and some real data applications related to our theoretical results. Because of the generality of our theoretical framework, it is impossible to report simulation performance even for the most prominent examples in these different classes. Therefore we restrict ourselves to conditional heteroscedasticity (CH) models for simulations and real data applications. For the time-varying simultaneous band, to the best of our knowledge, there is no or little simulation results reported. For the tvAR, tvMA and tvARMA processes we obtained quite satisfactory results, but they are omitted here to keep this discussion concise.


\subsection{Simulations}\label{ssc:simulations}

In this section, we study the finite sample coverage probabilities of our SCBs for theoretical coverage $\alpha = 0.9$ and $\alpha = 0.95$ in the following tvARCH(1) and tvGARCH(1,1) models:
\begin{itemize}
    \item[(a)] $X_i = \sqrt{\alpha_0(i/n) + \alpha_1(i/n) X_{i-1}^2}\zeta_i$, where $\alpha_0(t) = 0.8 + 0.3 \cos(\pi t)$, $\alpha_1(t) = 0.45+0.1 \cos(\pi t)$,
    \item[(b)] $X_i = \sigma_i \zeta_i$, $\sigma_i^2 = \alpha_0(i/n) + \alpha_1(i/n) X_{i-1}^2 + \beta_1(i/n) \sigma_{i-1}^2$, where $\alpha_0(t) = 2.4 + 0.02\cos(\pi t)$, $\alpha_1(t) = 0.4 + 0.1\cos(\pi t)$ and $\beta_1(t) = 0.5 - 0.1\cos(\pi t)$,
\end{itemize}
where $\zeta_i$ is i.i.d. standard normal distributed. For estimation, we choose $K(x) = \frac{3}{4}(1 - x^2)\Ii_{[-1,1]}(x)$ to be the Epanechnikov kernel, $n = 500, 1000, 2000, 5000 $ for several different $b_n$ (the optimal bandwidths \reff{optimalbandwidthchoice} are also reported for model (a) and model (b)). For each situation, $N = 2000$ replications are performed and it is checked if the obtained SCB based on \reff{eq:boundary_corrected_bootstrap2} contains the true curves in $t \in (0,1)$. In both models we have $\Lambda(t) = I(t) = V(t)$ and therefore estimate $\Sigma_C^2(t) = C\tran I(t)^{-1} C$ via replacing $I(t)$ by $\hat I_{b_n}(t)$ from \reff{proposal_est_I}. We obtained the results given in Tables \ref{tab:data_arch} and \ref{tab:data_garch}. The estimation, for smaller sample sizes $n$,  sometimes may lead to difficulties since the optimization routine (\emph{optim} in programming language R) may not converge. We decided to discard these pathological cases for simplicity. It can be seen that the empirical coverage probabilities are reasonably close to the nominal level for bandwidths close to the optimal ones and they do not differ too much for other bandwidths as well.



\begin{table}[!ht]
\caption{Coverage probabilities of the SCB in (a) for $n = 500, 1000, 2000,5000$;}
\centering
\label{tab:data_arch}
\begin{tabular}{| l | l || c | c | c || c | c | c |}
\hline
& & \multicolumn{3}{c||}{$\alpha = 90\%$} &  \multicolumn{3}{c|}{$\alpha = 95\%$}\\
 $n$ & $b_n$ & $\alpha_0$ & $\alpha_1$ & $(\alpha_0,\alpha_1)\tran$ & $\alpha_0$ & $\alpha_1$ & $(\alpha_0,\alpha_1)\tran$ \\
\hline\hline
500 &  0.45 & 0.859 & 0.839 & 0.833 & 0.948 & 0.912 & 0.900 \\
\hline
 &  0.5 & 0.869 & 0.864 & 0.842 & 0.937 & 0.914 & 0.895 \\
\hline
 &  0.55 & 0.864 & 0.849 & 0.832 & 0.930 & 0.901 & 0.898\\
\hline\hline
1000 &  0.4 & 0.873 & 0.846 & 0.845 & 0.937 & 0.906 & 0.900 \\
\hline
 &  0.45 & 0.885 & 0.875 & 0.879 & 0.941 & 0.925 & 0.927 \\
\hline
 &  0.5 & 0.887 & 0.876 & 0.864 & 0.948 & 0.926 & 0.931\\
\hline
&  0.55 & 0.871 & 0.870 & 0.866 & 0.931 & 0.925 & 0.921 \\
\hline \hline 
2000 & $ 0.3$ & 0.893 & 0.861 & 0.868 & 0.946 & 0.924 & 0.930 \\
\hline
&  0.35 & 0.886 & 0.872 & 0.866 & 0.938 & 0.928 & 0.921\\
\hline
&  0.4 & 0.891 & 0.878 & 0.874 & 0.937 & 0.926 & 0.933\\
\hline
&  0.45 & 0.874 & 0.873 & 0.883 & 0.940 & 0.937 & 0.937\\
\hline \hline
5000 &  0.25 & 0.885 & 0.883 & 0.882 & 0.941 & 0.931 & 0.936\\
\hline
&  0.3 & 0.892 & 0.883 & 0.889 & 0.949 & 0.938 & 0.941\\
\hline
&  0.35 & 0.900 & 0.891 & 0.894 & 0.948 & 0.945 & 0.938\\
\hline
&  0.4 & 0.900 & 0.899 & 0.894 & 0.953 & 0.947 & 0.937\\
\hline
&  0.45 & 0.878 & 0.880 & 0.881 & 0.934 & 0.937 & 0.930\\
\hline
\end{tabular}
\end{table}


\begin{table}[!ht]
\caption{Coverage probabilities of the SCB in (b) for $n = 500, 1000, 2000, 5000$}
\centering
\label{tab:data_garch}
\begin{tabular}{| l | l || c | c | c | c || c | c | c | c |}
\hline
& &  \multicolumn{4}{c||}{$\alpha = 90\%$} &  \multicolumn{4}{c|}{$\alpha = 95\%$}\\
 $n$ & $b_n$ & $\alpha_0$ & $\alpha_1$ & $\beta_1$ & $(\alpha_0,\alpha_1,\beta_1)\tran$ & $\alpha_0$ & $\alpha_1$ & $\beta_1$ & $(\alpha_0,\alpha_1,\beta_1)\tran$ \\
\hline
500 & $0.55$ & 0.943 & 0.778 & 0.902 & 0.793 & 0.962 & 0.851 & 0.939  & 0.854 \\
\hline
 & $0.6$ & 0.946 & 0.815 & 0.922 & 0.837 & 0.962 & 0.875 & 0.961  & 0.89 \\
\hline
 & $0.65$ & 0.947 & 0.828 & 0.92 & 0.841 & 0.966 & 0.884 & 0.956  & 0.899 \\
\hline
 & $0.7$ & 0.939 & 0.849 & 0.922 & 0.853 & 0.962 & 0.897 & 0.957  & 0.903 \\
\hline\hline
1000 & $0.5$ & 0.943 & 0.857 & 0.936 & 0.864 & 0.974 & 0.913 & 0.966  & 0.914 \\
\hline
 & $0.55$ & 0.947 & 0.869 & 0.926 & 0.891 & 0.97 & 0.929 & 0.954  & 0.937 \\
\hline
 & $0.6$ & 0.946 & 0.884 & 0.943 & 0.908 & 0.971 & 0.937 & 0.966  & 0.948 \\
\hline 
 & $0.65$ & 0.944 & 0.873 & 0.921 & 0.889 & 0.965 & 0.921 & 0.951  & 0.934 \\
\hline \hline
2000 & $0.3$ & 0.944 & 0.822 & 0.924 & 0.869 & 0.967 & 0.89 & 0.961  & 0.92 \\
\hline
 & $0.35$ & 0.941 & 0.843 & 0.923 & 0.864 & 0.967 & 0.908 & 0.945  & 0.913 \\
\hline
 & $0.40$ & 0.958 & 0.846 & 0.92 & 0.894 & 0.972 & 0.91 & 0.96  & 0.943 \\
\hline 
 & $0.45$ & 0.957 & 0.875 & 0.931 & 0.897 & 0.98 & 0.928 & 0.965  & 0.942 \\
\hline 
 & $0.50$ & 0.946 & 0.887 & 0.952 & 0.911 & 0.978 & 0.938 & 0.979  & 0.952 \\
\hline 
\hline
5000 & $0.25$ & 0.955 & 0.855 & 0.936 & 0.886 & 0.974 & 0.903 & 0.959 & 0.929 \\
\hline
 & $0.30$ & 0.941 & 0.889 & 0.939 & 0.904 & 0.974 & 0.931 & 0.967 & 0.949 \\
\hline
 & $0.35$ & 0.949 & 0.903 & 0.941 & 0.903 & 0.972 & 0.95 & 0.975 & 0.946 \\
\hline
 & $0.40$ & 0.954 & 0.882 & 0.94 & 0.92 & 0.968 & 0.94 & 0.969 & 0.959 \\
\hline
 & $0.45$ & 0.966 & 0.892 & 0.956 & 0.909 & 0.98 & 0.946 & 0.98 & 0.96 \\
\hline
 & $0.50$ & 0.948 & 0.886 & 0.932 & 0.896 & 0.976 & 0.937 & 0.976 & 0.957  \\
\hline
\end{tabular}
\end{table}

\subsection{Applications}\label{ssc:applications}

In this section, we consider a few real-data applications of our procedure. As mentioned in Section \ref{sec:introduction}, there are abundant results in the literature about time-varying regression but the results for time-varying autoregressive conditional heteroscedastic models are scarce. Thus it is important to evaluate the performance of our constructed SCBs for these type of models in both theoretical and real data scenarios. Among the popular heteroscedastic models, usually GARCH type models are most difficult to estimate due to the recursion of the variance term. 

We consider two examples from the class of conditional heteroscedastic models with two types of financial datasets: one foreign exchange and one stock market daily pricing dataset. As \citet{fry08} found out, ARCH models have good forecasting ability for currency exchange type data whereas for data coming from the stock market, GARCH models are preferred. Typically, these daily closing price datasets show unit root behavior and thus instead of using the daily price data, we model the log-return data. The log-return is defined as follows and is close to the relative return

$$Y_i=\log P_i- \log P_{i-1}=\log \left( 1+ \frac{P_i-P_{i-1}}{P_{i-1}}\right) \approx \frac{P_i-P_{i-1}}{P_{i-1}}, $$
where $P_i$ is the closing price on the $i^{th}$ day. Because of the apparent time-varying nature of volatility these log-return data typically show, conditional heteroscedastic models are used for analysis and forecasting.

\subsubsection{Real data application I: USD/GBP rates}
For the first application, we consider a tvARCH($p$) model with $p=1,2$. It has the following form
\[
    Y_{i}^2 = \sigma_i^2 \zeta_i^2, \quad\quad \sigma_i^2 = \alpha_0(i/n) + \alpha_1(i/n) Y_{i-1}^2 + \ldots + \alpha_p(i/n)Y_{i-p}^2.
\]

Many different exchange rates from 1990-1999 for USD with other currencies were analyzed in \cite{fry08} using tvARCH($p$) models with $p=0,1,2$. We collect the data for USD-GBP exchange rates from \url{www.federalreserve.gov/releases/h10/Hist/default1999.htm}. The authors suggested choosing $p=1$ for USD-GBP exchange rates and we also decided to restrict ourselves to fitting a tvARCH(1) model only. Note that in principle, our simultaneous bands can be used to decide whether the additional parameter in a tvARCH(2) model is needed or not. The dataset has sample size 2514 and we use a  cross-validated bandwidth of $b_n=0.26$. We also provide the plots for the log-returns and an ACF plot of the squared time series that shows the  evidence of conditional heteroscedasticity.

\begin{figure} [!ht]  
\centering
\includegraphics[width=15.5cm,height=13.2cm]{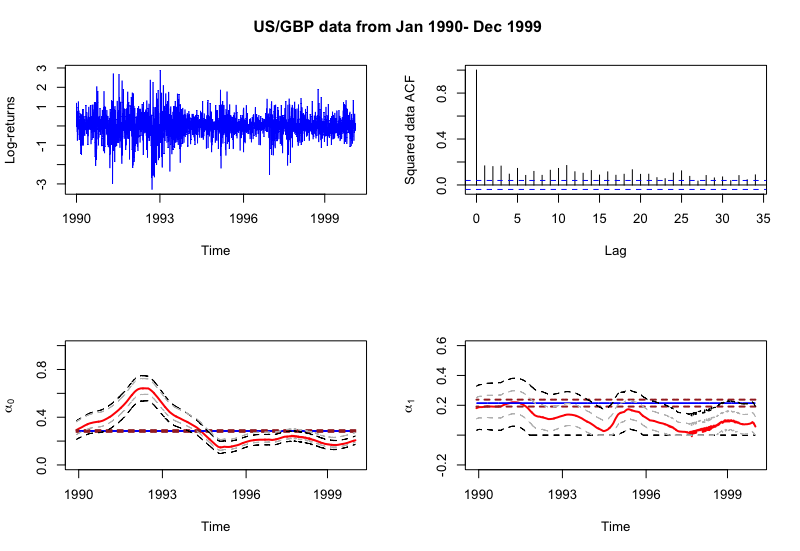}
\caption{Analysis of USD/GBP data from Jan 1990 to Dec 1999. Top left: Log-returns. Top right: ACF plot. Bottom panel: Estimates of the parameters $\alpha_0(\cdot)$, $\alpha_1(\cdot)$, respectively (red) with 95\% SCBs (dashed) and estimates of the parameters assuming constancy (blue).  We also provide pointwise bands (grey) and time-constant estimate $\pm$ standard error band (brown). Optimal bandwidth was 0.26 and $\alpha_0(\cdot)$ is time-varying}
\label{fig:usgbp} 
\end{figure}

\noindent Based on Figure \ref{fig:usgbp}, time-constancy for the  parameter curve $\alpha_0(\cdot)$ is rejected at 5\% level of significance. For $\alpha_1(\cdot)$, the estimate generally stays below the stationary fit. This can be explained by the geometry of the parameter space of the GARCH model, cf. \cite{hillebrand_garch}. Also, one can see from the plot of actual log-returns that there are large shocks from 1990 to 1993 compared to those seen in 1993-1999. This can be explained through the high (low) values shown for the estimated curve $\alpha_0(\cdot)$ for the time-period 1990-1993 (1993-1999).

\subsubsection{Real data application II: NASDAQ index data}
In the empirical analysis of log-returns for stock market data,  \citet{palm96} and others have found that lower order GARCH models account sufficiently for conditional heteroscedasticity. Moreover, GARCH(1,1) and in a very few cases GARCH(1,2) and GARCH(2,1) models are used and higher order GARCH models are typically not necessary. Another advantage of using GARCH(1,1) over ARCH($p$) models is that one does not need to worry about choosing a proper lag $p$ as GARCH(1,1) can be thought as an ARCH model with $p=\infty$. In this subsection, we implement a time-varying version of GARCH(1,1) and obtain the bootstrapped SCB. A tvGARCH(1,1) model has the following form:
\[
    Y_{i}^2 = \sigma_i^2 \zeta_i^2, \quad\quad \sigma_i^2 = \alpha_0(i/n) + \alpha_1(i/n) Y_{i-1}^2 + \beta_1(i/n) \sigma_{i-1}^2.
\]

As our second example, we choose to analyze the log returns of NASDAQ from January 2011 to December 2018. This is an important index in the US stock market. We collect this one and all other stock index datasets in later analysis from \url{www.investing.com}. Our cross-validated bandwidth is $b_n=0.405$ for this dataset of size $n=1751$. Since our simulations show excellent performance for sample sizes around $n=2000$ and the estimated parameter functions satisfy the parameter restriction $\sup_{0 \leq t \leq 1}(\hat{\beta}_1(t)^2+2 \hat{\alpha}_1(t)\hat{\beta}_1(t)+3\hat{\alpha}_1(t)^2)<1$, it is reasonable to say our simultaneous confidence bands would also be valid here. As one can see from Figure \ref{fig:nasdaq}, the time series shows significant lags in its ACF plot after squaring; indicating conditional heteroscedasticity.

\begin{figure} [!ht]  
\centering
\includegraphics[width=15.5cm,height=13.2cm]{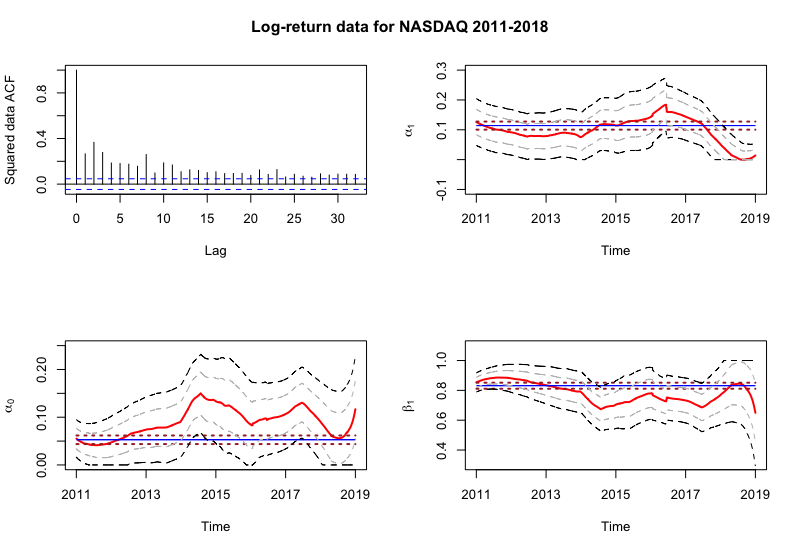}
\caption{Analysis of Nasdaq data from Jan 2011 to Dec 2018. Top left: ACF plot. Top right, bottom left, bottom right: Estimates of the parameters $\alpha_0(\cdot)$, $\alpha_1(\cdot)$ and $\beta_1(\cdot)$, respectively (red) with SCBs (dashed) and estimates of the parameters assuming constancy (blue). We also provide pointwise bands (grey) and time-constant estimate $\pm$ standard error band (brown). Optimal bandwidth was 0.405. A horizontal line does not pass through the SCB of $\alpha_1(\cdot)$ and thus it is time-varying. }
\label{fig:nasdaq} 
\end{figure}

One can see that the estimates for $\alpha_0(\cdot)$ is mostly above the corresponding time-constant fit. As mentioned in the caption $\alpha_1(\cdot)$ is time-varying since the SCB does not contain a horizontal line. The fit fluctuates around the time-constant fit. For $\beta_1(t)$, the time-varying fit is below the corresponding time-constant fit. Overall, since $\alpha_1(t)$ is deemed time-varying through this analysis the time-constant hypothesis can be rejected at 5\% level of significance.

\subsection{Forecasting volatility} It is a legitimate question whether   time-varying models in forecasting econometric time-series are more useful  compared to their time-constant analogue. Note that the main goal of this paper is not to build better forecasting models. The extension to predictive intervals from confidence intervals for conditional heteroscedastic models is not very straight-forward. Moreover, it is unclear how in-fill asymptotics discussed in this paper would extend to forecasting future trends or estimate time-varying functions with time arguments $t > 1$ in the future. Any asymptotic theory would need to consider the rescaling mechanism rigorously, keeping in mind the data observed up to a certain point. In this subsection we show empirically that time-varying models can indeed lead to better forecasts compared to time-constant analogues. 

\subsubsection{Short-range forecasts}
Following \cite{starica2003garch} and \cite{subbarao2008}, we show for a wide range of econometric datasets that time-varying models can provide better short range forecasts. We also allow multiple windows of forecasting and multiple start points to highlight why most of these datasets call for a time-varying fit. In the following Tables \ref{tab:amse500} and \ref{tab:amse1000}, we use ARCH(1) models for the forex datsets and GARCH(1,1) for the stock market indices.  Our POOS (pseudo-out-of-sample) evaluation of forecasting is chalked out as follows:

Define, for a $h-$step ahead forecasting scheme, $\bar{\sigma}^2_{t,t+h}= \frac{1}{h}\sum_{i=t+1}^{t+h}\sum\sigma_{i|t}^2$ where $\sigma_{i|t}^2$ are the $(i-t)$ step ahead forecasts of time-constant or time-varying fit at time $t$. We compare this with the  `realized' volatility $\bar{X}^2_{t,t+h}=\frac{1}{h}\sum_{i=t+1}^{t+h}X_i^2$. We then compute the aggregated measure for a start point $s$ as following: $$ AMSE = \frac{1}{n-h-s}\sum_{s+1}^{n-h} (\bar{\sigma}^2_{t,t+h} - \bar{X}^2_{t,t+h}).$$ 

For the forecasting horizon $h$ values, we choose $h \in \{25,50,75,100,150,200\}$ and start points $s \in \{500,1000\}$. Our forecasting method is the same as the one implemented in the fGARCH R-package. For the time-constant fit at time $t$ we use the data from 1 to $t$  to predict the $h$-step ahead forecast. For the time-varying fit however, it is unclear what the time-varying projection will be. Following \cite{subbarao2008}, we assume the last $m$ points to be stationary for a small $m$ and use that to obtain the future forecasts. Since $m$ is a tuning parameter, we choose the value that produces the minimum $AMSE$ over $m\in \{100,200, \ldots, 500\}$. 

\begin{table}[!ht]
\caption{Forecasting volatility: Short range forecast comparison of time-varying and time-constant model for startpoint at 500. TV and TC stands for time-varying and time-constant models}
\centering
\label{tab:amse500}
\begin{tabular}{| l || c | c || c | c || c | c || c | c || c | c || c | c ||}
\hline
&  \multicolumn{2}{c||}{ahead=25} &  \multicolumn{2}{c||}{ahead=50} & \multicolumn{2}{c||}{ahead=75} & \multicolumn{2}{c||}{ahead=100} & \multicolumn{2}{c||}{ahead=150} & \multicolumn{2}{c||}{ahead=200}   \\
 $Data$ & TC & TV & TC & TV & TC & TV & TC & TV & TC & TV & TC & TV \\
\hline\hline
USGBP & 0.101 & 0.07 & 0.085 & 0.057 & 0.078 & 0.053 & 0.075 & 0.052 & 0.07 & 0.048 & 0.067 & 0.046 \\
USCHF & 2.968 & 2.775 & 1.702 & 1.528 & 1.269 & 1.115 & 1.048 &  0.099 & 0.803 & 0.601 & 0.646 & 0.524 \\
USCAD & 0.029 & 0.026 & 0.024 & 0.024 & 0.022 & 0.023 & 0.02 & 0.021 & 0.017 & 0.018 & 0.015 & 0.017 \\
EURGBP & 0.083 & 0.074 & 0.057 & 0.048 & 0.044 & 0.036 & 0.039 & 0.032 & 0.034 & 0.029 & 0.028 & 0.025 \\
EURUSD & 0.052 & 0.033 & 0.044 & 0.028 & 0.041 & 0.027 & 0.042 & 0.028 & 0.038 & 0.032 & 0.036 & 0.038 \\
\hline \hline
SP Merval & 7.722 & 7.094 & 6.632 & 4.679 & 5.317 & 4.376 & 5.763 & 4.514 & 5.372 & 4.529 & 5.086 & 4.620 \\
\hline 
BSE & 3219 & 3038 & 2916 & 1901 & 2313 & 2040 & 2259 & 1321 & 2000 & 1106 & 1897 & 1038 \\ 
\hline 
SP500 old & 5.475 & 5.346 & 8.752 & 7.511 & 4.265 & 4.048 & 3.339 & 2.344 & 2.092 & 1.079 & 1.768 & 0.939 \\
\hline
SP500 & 0.319 & 0.042 & 0.379 & 0.372 & 0.249 & 0.324 & 0.304 & 0.306 & 0.277 & 0.280 & 0.269 & 0.274 \\
\hline
Dow Jones & 0.037 & 0.0365 & 0.316 & 0.277 & 0.223 & 0.231 & 0.215 & 0.200 & 0.185 & 0.183 & 0.169 & 0.176 \\
\hline
FTSE & 0.470 & 0.461 & 0.456 & 0.283 & 0.322 & 0.279 & 0.346 & 0.206 & 0.336 & 0.240 & 0.319 & 0.278 \\
\hline 
NASDAQ & 0.424 &0.399 & 0.498 & 0.237 & 0.283 & 0.233 & 0.404 & 0.175 & 0.388 & 0.200 & 0.364 & 0.218 \\
\hline 
Smallcap & 0.448 & 0.292 & 1.035 & 0.189 & 0.398 & 0.157 & 0.981 & 0.143 & 0.979 & 0.144 & 0.970 & 0.143 \\ 
\hline
NYSE & 0.318 & 0.245 & 0.470 & 0.144 & 0.277 & 0.148 & 0.418 & 0.107 & 0.413 & 0.121 & 0.404 & 0.136 \\
\hline
DAX & 0.903 & 0.800 & 1.142 & 0.559 & 0.639 & 0.552 & 0.992 & 0.501 & 0.961 & 0.580 & 0.920 & 0.697 \\
\hline 
Apple & 3.736 & 3.535 & 1.853 & 2.100 & 2.069 & 1.816 & 1.380 & 1.787 & 1.305 & 1.818 & 1.280 & 1.749 \\
\hline 
Microsoft & 2.729 & 3.388 & 1.600 & 1.999 & 1.156 & 1.898 & 0.981 & 1.335 & 0.766 & 1.084 & 0.664 & 0.973 \\
\hline AXP & 2.350 & 2.13 & 2.142 & 1.500 & 1.456 & 1.264 & 1.726 & 1.065 & 1.759 & 0.984 & 1.625 & 0.928 \\ \hline  
\end{tabular}
\end{table}

\begin{table}[!ht]
\caption{Forecasting volatility: Short range forecast comparison of time-varying and time-constant model for startpoint at 1000. TV and TC stands for time-varying and time-constant models}
\centering
\label{tab:amse1000}
\begin{tabular}{| l || c | c || c | c || c | c || c | c || c | c || c | c ||}
\hline
&  \multicolumn{2}{c||}{ahead=25} &  \multicolumn{2}{c||}{ahead=50} & \multicolumn{2}{c||}{ahead=75} & \multicolumn{2}{c||}{ahead=100} & \multicolumn{2}{c||}{ahead=150} & \multicolumn{2}{c||}{ahead=200}   \\
 $Data$ & TC & TV & TC & TV & TC & TV & TC & TV & TC & TV & TC & TV \\
\hline\hline
USGBP & 0.068 & 0.019 & 0.064 & 0.017 & 0.062 & 0.015 & 0.061 & 0.014  & 0.059 & 0.013 & 0.058 & 0.014\\ \hline
USCHF & 4.344 & 4.009 & 1.796 & 1.584 & 1.074 & 0.891 & 0.799 & 0.624 & 0.559 & 0.392 & 0.419 & 0.278 \\\hline
USCAD & 0.029 & 0.028 & 0.024 & 0.027 & 0.022 & 0.024 & 0.019 & 0.023 & 0.017 & 0.021 & 0.015 & 0.020 \\\hline
EURGBP & 0.019 & 0.012 & 0.017 & 0.011 & 0.016 & 0.010 & 0.015 & 0.009 & 0.015 & 0.010 & 0.014 & 0.011 \\\hline
EURUSD & 0.039 & 0.028 & 0.034 & 0.029 & 0.032 & 0.031 & 0.032 & 0.034 & 0.029 & 0.040 & 0.027 & 0.045 \\ \hline
\hline 
SP Merval & 10.39 & 8.714 & 9.432 & 6.12 & 7.666 & 6.029 & 8,522 & 5.883 & 7.908 & 5.857 & 7.052 & 5.739 \\
\hline 
BSE & 2478 & 2126 & 2798 & 1437 & 1506 & 949 & 2288 & 783 & 2263 & 771 & 2362 & 945.9\\
\hline
SP500 OlD & 8.104 & 7.964 & 12.57 & 11.26 & 6.242 & 5.949 & 4.292 & 3.469 & 2.307 & 1.463 & 1.741 & 1.216 \\
\hline
SP500 & 0.371 & 0.503 & 0.328 & 0.276 & 0.259 & 0.236 & 0.231 & 0.196 & 0.183 & 0.166 & 0.155 & 0.152 \\
\hline 
DowJones & 0.484 & 0.509 & 0.351 & 0.400 & 0.263 & 0.335 & 0.210 & 0.292 & 0.167 & 0.267 & 0.145 & 0.238 \\
\hline
FTSE & 0.644 & 0.639 & 0.514 & 0.378 & 0.409 & 0.394 & 0.358 & 0.270 & 0.347 & 0.353 & 0.314 & 0.419 \\  
\hline
NASDAQ & 0.519 & 0.505 & 0.482 & 0.345 & 0.332 & 0.377 & 0.340 & 0.267 & 0.312 & 0.338 & 0.284 & 0.365 \\
\hline
Smallcap & 0.457 & 0.374 & 0.671 & 0.243 & 0.338 & 0.243 & 0.558 & 0.229 & 0.488 & 0.232 & 0.433 & 0.206 \\
\hline
NYSE & 0.375 & 0.329 & 0.374 & 0.207 & 0.271 & 0.212 & 0.287 & 0.153 & 0.260 & 0.194 & 0.237 & 0.228 \\
\hline 
DAX & 1.245 & 1.123 & 1.130 & 0.755 & 0.842 & 0.708 & 0.925 & 0.762 & 0.897 & 0.896 & 0.832 & 1.102 \\
\hline
Apple & 2.465 & 2.275 & 1.799 & 1.279 & 1.530 & 1.206 & 1.374 & 0.871 & 1.276 & 1.036 & 1.237 & 1.052 \\
\hline
Microsoft & 2.868 & 3.668 & 1.720 & 2.295 & 1.214 & 2.147 & 1.057 & 1.596 & 0.851 & 1.335 & 0.770 & 1.232 \\
\hline
AXP & 2.837 & 2.668 & 1.749 & 1.462 & 1.395 & 1.347 & 1.151 & 0.825 & 0.982 & 0.681 & 0.834 & 0.581 \\
\hline \hline

\end{tabular}
\end{table}

One can see from Table \ref{tab:amse500} and \ref{tab:amse1000} how for a wide range of datasets, starting points and forecasting horizons the AMSE of the time-varying forecast is considerably smaller than the one for the time-constant version. We believe, even if prediction and forecasting is more important from an economist's perspective, this POOS analysis provides a strong motivation to choose a time-varying model over a time-constant one.

\subsubsection{One-step ahead forecasting and semi-timevarying models from inference}
We use the following one-step ahead $AMSE_1$, inspired from \cite{tvgarch2013} to validate our models:

$$AMSE_1= \frac{1}{n}\sum_{i=1}^n (X_i^2 - \hat{\sigma}^2(i/n))^2$$.

Here, $X_t$ are the log-returns and $\hat{\sigma}^2(\cdot)$ refers to the fitted model using ARCH(1) for foreign exchange datasets, and GARCH(1,1) for stock market indices. In each row we exhibit the best model in bold.

\begin{table}[!ht]
\caption{Forecasting volatility: Choosing between asymptotically 95\% correct time-varying, semi-time-varying and time-constant models}
\centering
\label{tab:data_forecast}
\begin{tabular}{| l || c | c || c || c| }
\hline
 Index (opt $\hat{b}_n$) & Time-varying & Time-constant & Time-constant coefficients & Semi-timevarying model  \\
\hline\hline
USGBP (0.26) &  0.5701109 & 0.5956046 & $\alpha_1$ &  \textbf{0.5697231}  \\
\hline
USCHF (0.24) &  \textbf{45.42274}  & 45.61378 & None & NA \\
\hline
USCAD (0.22) &  \textbf{0.1604907}  & 0.1672808 & $\alpha_1$ & 0.1616965\\
\hline 
EURGBP (0.34) & \textbf{0.778459}  & 0.789287 & $\alpha_1$ &  0.7837281\\
\hline
EURUSD (0.23) &  \textbf{0.3234517}  & 0.3399221 & None & NA\\
\hline \hline
SP Merval (0.36) & \textbf{60.62419}  & 62.02226 & $\alpha_1, \beta_1$ & 61.58843 \\
\hline 
SP500 old (0.29) &  \textbf{26.48272} & 26.12195 & None & NA \\
\hline 
SP500 recent (0.22) & 3.54   & \textbf{3.527918} & $\alpha_1, \beta_1$ & 3.529355\\
\hline
Wilshire 5k (0.35) & \textbf{1.943517} & 1.964926 & $\alpha_0, \alpha_1, \beta_1$ & 1.972347 \\
\hline 
FTSE (0.45) &  \textbf{3.650331}  & 3.687818 & None &  NA\\
\hline 
NASDAQ (0.405) & \textbf{5.4258546} & 5.446805  & $\alpha_1,\beta_1$ & 5.611163 \\
\hline 
Smallcap (0.35) &  \textbf{12.36647} & 12.56344 & $\alpha_1,\beta_1$ &  12.53796\\
\hline 
NYSE (0.24)&  \textbf{4.830424} & 4.841198 & $\alpha_0,\beta_1,\beta_1$ & 5.287622  \\
\hline 
DAX (0.44) &  \textbf{9.908042} & 9.987583 & None & NA  \\
\hline 
Apple (0.27) &  48.06839 & \textbf{46.4968} & None & NA  \\
\hline 
Microsoft (0.39)&  \textbf{38.08752} & 38.33715 & None & NA  \\
\hline 
AXP (0.28) &  \textbf{37.39755} & 37.90381 & $\alpha_1,\beta_1$ & 38.1988  \\

\hline 

\end{tabular}
\end{table}

Note that the examples exhibited here show that often the time-constant model has poor one-step ahead forecasting quality compared to their time-varying analogue. However, from Table \ref{tab:data_forecast}, one can see that in some of these time-varying models, we have a subset of parameters not rejecting the hypothesis of time-constancy. We suspect that setting a subset of parameters to be time-constant and allowing the rest to vary over time can improve forecasting over both models. One finding of this analysis is that for some of the datasets such as USGBP, the semi-time-varying model may outperform the time-constant model in terms of forecasting. Note that it is easy to tailor and find time-constant fits that allow for even better forecasts, but those models do not have proper confidence (in terms of closeness to the true model) for the already observed data. For the numbers in the above table on the semi-time-varying column, we kept the time-varying coefficients as they are and searched for the best (in terms of $AMSE_1$) constant for the time-constant coefficients among the horizontal lines that fit within the bands entirely. Here we would like also to put a word of caution: Note that the semi-time-varying analysis is somewhat adhoc. In principle, one can also re-run the optimization by fitting only a proper subset as time-varying and the rest as time-constant. We have checked this with multiple of the above datasets and the AMSE were not too different from that reported above. The major takeaway from this analysis remains that our time-varying fit, albeit not meant for prediction and constructed only for building simultaneous confidence intervals, can achieve better forecasts than the corresponding time-constant fits. Additionally, our theory can also lead to new models which have only a subset of coefficients time-varying.

\section{Acknowledgement}
We are grateful to the editor, associate editor and two anonymous referees for their valuable comments and feedback in different rounds which has helped in significantly improving this paper. This research was partially supported by NSF/DMS 1405410.

\appendix

\section{Technical tools and more general assumptions}\label{sec:tech}

\subsection{The functional dependence measure}

To state the structure of dependence we use throughout the appendix, we introduce a functional dependence measure on the underlying process using the idea of coupling as done in \citet{MR2172215}. Let $q > 0$ and let $Z_i$, $i\in\IZ$ a stationary process which admits the causal representation 
\begin{eqnarray}\label{eq:representation2}
Z_i=J(\zeta_i, \zeta_{i-1},\ldots).
\end{eqnarray}
\noindent Suppose that $(\zeta^{*}_i)_{i \in \IZ}$ is an independent copy of $(\zeta_i)_{i \in \IZ}$. For some random variable $Z$, let $\|Z\|_q := (\IE |Z|^q)^{1/q}$ denote the $\sL_q$-norm of $Z$. For $i \geq 0$, define the functional dependence measure 
\begin{eqnarray}\label{eq:fdm}
\delta^{Z}_{q}(i)= \| Z_i - Z_{i}^{*} \|_q,
\end{eqnarray}
\noindent where $Z_{i}^{*} = J(\sF_{i}^{*})$ with
\begin{eqnarray}\label{eq:copy of Fi}
\sF^{*}_i= (\zeta_i, \zeta_{i-1},  \cdots, \zeta_1,\zeta_0^{*}, \zeta_{-1}, \zeta_{-2}, \cdots )
\end{eqnarray}
 a coupled version of $\sF_i$ with $\zeta_0$ in $\sF_i$ replaced by an independent copy $\zeta_0^{*}$. 
Note that $\delta_q^{Z}(i)$ measures the dependence of $Z_i$ on $\zeta_{0}$ in terms of the $q$th moment. The tail cumulative dependence measure $\Delta^{Z}_{q}(j)$ for $j \geq 0$ is defined as 
\begin{eqnarray}\label{eq:cum fdm} 
\Delta^{Z}_{q}(j)=\displaystyle\sum_{i=j}^{\infty}\delta_{q}^{Z}(i). 
\end{eqnarray}

Let $\alpha > 0$. We define the adjusted dependence measure $\|\cdot \|_{q,\alpha}$ as follows (cf. \cite{zhangwu2017}): For a stationary process $Z_i = J(\zeta_i, \zeta_{i-1},...)$, let
\[
    \|Z\|_{q,\alpha} := \sup_{m\ge 0}(m+1)^{\alpha}\Delta^{Z}_q(m).
\]

\subsection{The class \texorpdfstring{$\sH(M,\chi,\bar C)$}{H(M,chi,C)} for Case 1} \label{ssc:smooth} To prove uniform convergence of $L_{n,b_n}^c$ and its derivatives w.r.t. $\theta$, we require the objective function $\ell$ introduced in Section \ref{section_estimator} to be Lipschitz continuous in the direction of $\theta$ and to grow at most polynomially in the direction of $z = (y,x)$, where the degree is measured by real numbers $M \ge 1$. We will therefore ask $\ell$ and its derivatives to be in the class $\sH(M,\chi,\bar C)$ which is now defined.

Let $\chi = (\chi_i)_{i=1,2,\ldots}$ be a sequence of nonnegative real numbers with $|\chi|_1 := \sum_{i=1}^{\infty}\chi_i < \infty$, and $\bar C > 0$ be some constant. Let $|x|_{\chi,s} := (\sum_{j=1}^{\infty}\chi_j |x|^s)^{1/s}$, and put $|x|_{\chi} := |x|_{\chi,1}$. Put $\hat \chi =(1,\chi)$.

A function $g: \IR^{\IN} \times \Theta \to \IR$ is in $\sH(M,\chi,\bar C)$ if $\sup_{\theta \in \Theta}|g(0,\theta)| \le \bar C$,
    \[
        \sup_{z}\sup_{\theta \not= \theta'}\frac{|g(z,\theta) - g(z,\theta')|}{|\theta-\theta'|_1 (1 + |z|_{\hat\chi}^M)} \le \bar C
    \]
    and
    \begin{eqnarray*}
        \sup_{\theta}\sup_{z\not= z'}\frac{|g(z,\theta) - g(z',\theta)|}{|z-z'|_{\hat\chi}\cdot (1 + |z|_{\hat\chi}^{M-1} + |z'|_{\hat \chi}^{M-1})} \le \bar C.
    \end{eqnarray*}
    If $g$ is vector- or matrix-valued, $g \in \sH(M,\chi, \bar C)$ means that every component of $g$ is in $\sH(M,\chi,\bar C)$.

\subsection{A more general set of assumptions}\label{sec:assumption} We show the main theorems under a more general set of assumptions which contain more `high-level' properties of the process $Y_i$ and the objective function $\ell$. These assumptions can therefore be seen as an `intermediate step' in our proofs: We first derive these high-level properties of the processes $Y_i$ which hold under the more specific Assumptions \ref{ass:case1} and \ref{ass:case2}. These high-level assumptions are given in Assumption \ref{ass1} and \ref{ass3}. In Section \ref{sec:proofs_assumptions} at the end of the supplementary material we show that Assumption \ref{ass:case1} implies Assumption \ref{ass1} and Assumption \ref{ass:case2} implies Assumption \ref{ass3}. To keep the presentation concise, we only state Assumption \ref{ass1} here in the main part of the paper.

For each $t \in [0,1]$, let $\tilde Y_i(t) = \tilde J(t,\sF_i)$ with some measurable function $\tilde J$. This process serves as a stationary approximation of the observed process $Y_i$ (cf. the introduction in Section \ref{sec:main results}). The necessary properties are made rigorous in Assumption \ref{ass1} (for Case 1) or Assumption \ref{ass3} (for Case 2) below. Recall that $X_i := (Y_{j}: -\infty < j \le i-1)$ and $Z_i := (Y_i, X_i)$, and $\tilde X_i(t) := (\tilde Y_j(t): -\infty < j\le i-1)$,  $\tilde Z_i(t) := (\tilde Y_i(t), \tilde X_{i-1}(t))$.

\begin{assumption}\label{ass1}
Suppose that for some $r \ge 2$ and some $\gamma > 1$,
    \begin{enumerate}[label=(A\arabic*),ref=(A\arabic*)]
        \item\label{ass1_smooth} (Smoothness in $\theta$-direction) $\ell$ is twice continuously differentiable w.r.t. $\theta$. It holds that $\ell, \nabla_{\theta} \ell, \nabla^2_{\theta} \ell \in \sH(M,\chi,\bar C)$ for some $M\ge 1$, $\bar C > 0$ and $\chi = (\chi_i)_{i=1,2,\ldots}$ with $\chi_i = O(i^{-(1+\gamma)})$.
        \item\label{ass1_theta} (Assumptions on unknown parameter curve) $\Theta$ is compact and for all $t \in [0,1]$, $\theta(t)$ lies in the interior of $\Theta$. Each component of $\theta(\cdot)$ is in $C^3[0,1]$.
        \item\label{ass1_model} (Correct model specification) For all $t \in [0,1]$, the function $\theta \mapsto L(t,\theta) := \IE \ell(\tilde Z_0(t), \theta)$ is uniquely minimized by $\theta(t)$.
        \item\label{ass1_matrix} The eigenvalues of the matrices $V(t)$, $I(t)$ and $\Lambda(t)$ defined in \reff{eq:vdef_paper}, \reff{eq:idef_paper} and \reff{eq:lambdadef_paper} are bounded from below by some $\lambda_0 > 0$, uniformly in $t$.
    \item\label{ass1_stat} (Stationary approximation) There exist $C_A,C_B,D > 0$ such that for all $n\in\IN$, $i = 1,\ldots,n$, $t,t' \in [0,1]$:
        \[
            \max\{\|Y_i\|_{rM}, \ \|\tilde Y_0(t)\|_{rM}\} \le D,
        \]
        \begin{equation}
            \|Y_{i} - \tilde Y_{i}(i/n)\|_{rM} \le C_A n^{-1},\quad\quad \|\tilde Y_{0}(t) - \tilde Y_{0}(t')\|_{rM} \le C_B |t-t'|\label{eq:ass1_stat1}
        \end{equation}
        \item\label{ass1_dep} (Weak dependence) $\sup_{t\in [0,1]}\delta_{rM}^{\tilde Y(t)}(k) = O(k^{-(1+\gamma)})$.
    \end{enumerate}
\end{assumption}

\section{Proofs of the theorems}\label{sec:proofs}

In this section, we show the validity of the theorems stated in the paper under the more general Assumption \ref{ass1} or Assumption \ref{ass3}, respectively. We make use of the elementary lemmas derived in Section \ref{sec:elementaryresults}. Let us introduce some notation. For $\eta = (\eta_1,\eta_2) \in \Theta \times (\Theta'\cdot b_n) =: E_n$, define
\[
    L^{\circ,c}_{n,b_n}(t,\eta) := (nb_n)^{-1}\sum_{i=1}^{n}K_{b_n}(i/n-t)\ell(Z_i^c, \eta_1 + \eta_2(i/n-t)b_n^{-1})
\]
and $\hat L_{n,b_n}^{\circ}$, $L_{n,b_n}^{\circ}$ similarly as  $L_{n,b_n}^{\circ,c}$ but with $Z_i^c$ replaced by $\tilde Z_i(i/n)$ or $Z_i$, respectively. 
Furthermore, put
\[
    \eta_{b_n}(t) = (\theta(t)\tran, b_n \theta'(t)\tran)\tran.
\]
$\eta_{b_n}(t)$ is estimated by
\[
    \hat \eta_{b_n}(t) = (\hat \theta_{b_n}(t)\tran, b_n \widehat \theta_{b_n}'(t)\tran)\tran \in \argmin_{\eta \in E_n}L_{n,b_n}^{\circ,c}(t,\eta).
\]

\subsection{Proof of Theorem \ref{theorem:bahadur_paper}}

\begin{proof}[Proof of Theorem \ref{theorem:bahadur_paper}]
    By Proposition \ref{example:tvrec},  Assumption \ref{ass1} is fulfilled with some $r > 2$. By  Proposition \ref{example:garch}, Assumption \ref{ass3} is fulfilled with some $r > 2$. In the following, we will only use the more general Assumptions \ref{ass1} or \ref{ass3}, respectively.
    
By Lemma \ref{lemma:empprocess2}(i),(iii)(a) and Lemma \ref{lemma:bias}(a) (in case Assumption \ref{ass1} holds) or Lemma \ref{lemma:empprocess2}(i),(iii)(c) and Lemma \ref{lemma:bias}(a)  (if Assumption \ref{ass3} holds) applied to $g = \ell$, we have that
\[
    \sup_{t\in \sT_n}\sup_{\eta \in E_n}|L_{n,b_n}^{\circ,c}(t,\eta) - L^{\circ}(t,\eta)| = O_{\IP}(\beta_n + (nb_n)^{-1}) + O(b_n),
\]
where
\[
    L^{\circ}(t,\eta) := \int_{-1}^{1}K(x) L(t,\eta_1 + \eta_2 x) dx.
\]
That is, $L_{n,b_n}^{\circ,c}(t,\eta)$ converges  to $L^{\circ}(t,\eta)$ uniformly in $t,\eta$ if $b_n =o(1)$ and $\beta_n = o(1)$. By Lemma \ref{lemma:hoelder}, $\eta \mapsto L^{\circ}(t,\eta)$ is Lipschitz continuous in both components. Since $\theta(t)$ is the unique minimizer of $\theta \mapsto L(t,\theta)$, we have that $\eta(t) = (\theta(t)\tran,0)\tran$ is the unique minimizer of $\eta \mapsto L^{\circ}(t,\eta)$.
Since $\hat \eta_{b_n}(t) = (\hat \theta_{b_n}(t)\tran, b_n \widehat\theta_{b_n}'(t)\tran)\tran$ is a minimizer of $L_{n,b_n}^{\circ,c}(t,\eta)$, standard arguments yield
\[
     \sup_{t \in \sT_n}|\hat \eta_{b_n}(t) - \eta(t)| = o_{\IP}(1).
\]
Since $\sup_{t\in \sT_n}|\eta(t) - \eta_{b_n}(t)| = o(1)$, we have
\begin{equation}
    \sup_{t \in \sT_n}|\hat \eta_{b_n}(t) - \eta(t)| = o_{\IP}(1).\label{mle:consistency}
\end{equation}
Thus for $n$ large enough, $\hat \eta_{b_n}(t)$ is in the interior of $E_n$ uniformly in $t$. By a Taylor expansion, we obtain for each $t \in \sT_n$:
\begin{equation}
    \hat \eta_{b_n}(t) - \eta_{b_n}(t) = -\big[ V^{\circ}(t) + R_{n,b_n}(t) \big]^{-1}\cdot \nabla_{\eta} L_{n,b_n}^{\circ,c}(t,\eta_{b_n}(t)),\label{eq:mlediff2}
\end{equation}
where
\[
    R_{n,b_n}(t) = \nabla_{\eta}^2 L_{n,b_n}^{\circ,c}(t,\bar \eta(t)) - V^{\circ}(t)
\]
with some $\bar \eta(t) \in E_n$ satisfying $|\bar \eta(t) - \eta_{b_n}(t)|_1 \le |\hat \eta_{b_n}(t) - \eta_{b_n}(t)|_1$, and
 \begin{equation}
     V^{\circ}(t) := \begin{pmatrix}
        1 & 0\\
        0 & \mu_{K,2}
    \end{pmatrix} \otimes V(t).\label{vcircdef_t}
\end{equation}

By Lemma \ref{lemma:empprocess2}(i),(iii)(a) and Lemma \ref{lemma:bias}(a) (if Assumption \ref{ass1} holds) or Lemma \ref{lemma:empprocess2}(i),(iii)(c) and Lemma \ref{lemma:bias}(b) (if Assumption \ref{ass3} holds) applied to $g = \nabla_{\theta}^2\ell$ and $\hat K(x) = K(x)$, $\hat K(x) = K(x)x$ or $\hat K(x) = K(x)x^2$, respectively, we have for some fixed $\iota' > 0$:
\begin{equation}
     \sup_{t\in \sT_n}\sup_{|\eta - \eta_{b_n}(t)| < \iota'}|\nabla_{\eta}^2 L_{n,b_n}^{\circ,c}(t,\eta) - V^{\circ}(t,\eta)| = O_{\IP}(\beta_n + (nb_n)^{-1}) + O(b_n),\label{eq:glmkonv2}
\end{equation}
where
\begin{equation}
    V^{\circ}(t,\eta) = \int_{-1}^{1}K(x) \begin{pmatrix}
    1 & x\\
    x & x^2
    \end{pmatrix} \otimes V(t,\eta_1 + \eta_2 x) dx.\label{eq:vcircdef_teta}
\end{equation}
For the moment, let $\tilde h_i(t) = \nabla_{\theta}\ell(\tilde Z_i(t),\theta(t))$. Note that $\IE \tilde h_0(t) = \IE \nabla_{\theta}\ell(\tilde Z_0(t), \theta(t)) = 0$ by Assumption \ref{ass1}\ref{ass1_model}, \ref{ass1_smooth} (or Assumption \ref{ass3}\ref{ass3_model}, \ref{ass3_smooth}).

By Lemma \ref{lemma:depmeasure}(i) (if Assumption \ref{ass1} holds) or Lemma \ref{lemma:depmeasuregarch}(i) (if Assumption \ref{ass3} holds), we have $\sup_{t}\delta_{2+\varsigma}^{\tilde h(t)_j}(k) = \sup_{t}\delta_{2+\varsigma}^{\nabla_{\theta_j}\ell(\tilde Z(t),\theta(t))}(k) = O(k^{-(1+\gamma)})$ for each $j = 1,\ldots,d_{\Theta}$. Using Lemma \ref{lemma:hoelder} (if Assumption \ref{ass1} holds) or Lemma \ref{lemma:hoeldergarch} (if Assumption \ref{ass3} holds), we see that the conditions of Lemma \ref{lemma:localscb} are fulfilled and thus, applied to $\tilde h_i(t)$,
\begin{equation}
    \sup_{t\in \sT_n}\big|(nb_n)^{-1}\sum_{i=1}^{n}K_{b_n}(i/n-t) \nabla_{\theta}\ell(\tilde Z_i(i/n), \theta(i/n))\big| = O_{\IP}( (nb_n)^{-1/2}\log(n)).\label{eq:glmkonv100}
\end{equation}
With Lemma \ref{lemma:biasapproxbahadur}, we obtain 
\[
    \sup_{t\in \sT_n}\big|\nabla_{\eta}\hat L_{n,b_n}^{\circ}(t,\eta_{b_n}(t)) - \IE \nabla_{\eta}\hat L_{n,b_n}^{\circ}(t,\eta_{b_n}(t))\big| = O_{\IP}((nb_n)^{-1/2}\log(n)  + \beta_n b_n^2).
\]
Since $\IE \nabla_{\theta} \ell(\tilde Z_0(t), \theta(t)) = 0$, we obtain with Lemma \ref{lemma:bias2} (a bias expansion result) and Lemma \ref{lemma:empprocess2}(i):
\begin{equation}
    \sup_{t\in \sT_n}|\nabla_{\eta_j}L_{n,b_n}^{\circ,c}(t,\eta_{b_n}(t))| = O_{\IP}((nb_n)^{-1/2}\log(n) + (nb_n)^{-1} + \beta_n b_n^2 + b_n^{1+j}),\label{eq:nablal}
\end{equation}
where $j = 1,2$. Since $\theta \mapsto V(t,\theta) = \IE \nabla^2_{\theta}\ell(\tilde Z_0(t), \theta)$ is Lipschitz continuous (apply Lemma \ref{lemma:hoelder} in case of Assumption \ref{ass1} or Lemma \ref{lemma:hoeldergarch} in case of Assumption \ref{ass3} to $\nabla^2 \ell$), the same holds for $\eta \mapsto V^{\circ}(t,\eta)$. We conclude that with some constant $C > 0$,
\begin{equation}
    \sup_{t\in \sT_n}|R_{n,b_n}(t)| \le \sup_{t\in \sT_n}\sup_{\eta \in E_n}|\nabla^2_{\eta}L_{n,b_n}^{\circ,c}(t,\eta) - V^{\circ}(t,\eta)| + C \sup_{t\in \sT_n}|\hat \eta_{b_n}(t) - \eta_{b_n}(t)|.\label{eq:rnb}
\end{equation}
Inserting \reff{eq:nablal}, \reff{eq:rnb} and \reff{mle:consistency} into \reff{eq:mlediff2}, we obtain
\begin{equation}
    \sup_{t\in \sT_n}|\hat \eta_{b_n,j}(t) - \eta_{b_n,j}(t)| = O_{\IP}((nb_n)^{-1/2}\log(n) + (nb_n)^{-1} + \beta_n b_n^2 + b_n^{1+j}),\label{eq:mlerate}
\end{equation}
where $j = 1,2$. Inserting \reff{eq:mlerate}, \reff{eq:glmkonv2} into \reff{eq:rnb}, we get $\sup_{t\in \sT_n}|R_{n,b_n}(t)| = O_{\IP}(\beta_n + b_n + (nb_n)^{-1})$. Together with
\begin{eqnarray*}
    && \big|V^{\circ}(t) \big(\hat \eta_{b_n}(t) - \eta_{b_n}(t)\big)  - \nabla L_{n,b_n}^{\circ,c}(t,\eta_{b_n}(t))\big|\\
    &\le& \big|\big[I_{2k\times 2k} +  V^{\circ}(t)^{-1} R_{n,b_n}(t)\big]^{-1} - I_{2k\times 2k}^{-1}\big|\cdot |\nabla_{\eta} L_{n,b_n}^{\circ,c}(t,\eta_{_n}(t))|\\
    &\le& \big|\big[I_{2k\times 2k} +  V^{\circ}(t)^{-1}R_{n,b_n}(t)\big]^{-1}\big| \cdot \big|V^{\circ}(t)^{-1}R_{n,b_n}(t)\big|\cdot |\nabla_{\eta} L_{n,b_n}^{\circ,c}(t,\eta_{b_n}(t))|,
\end{eqnarray*}
and \reff{eq:nablal} we obtain the assertion \reff{eq:bahadur_paper}. The other result \reff{eq:bahadur_localstat_paper} follows from Lemma \ref{lemma:empprocess2}(i), Lemma \ref{lemma:biasapproxbahadur} and Lemma \ref{lemma:bias2}.
\end{proof}

\subsection{Proof of Theorem \ref{th:scb_paper}}

In this section, we prove Theorem \ref{th:scb_paper} by proving its assertion under the morel general  Assumption \ref{ass1} or \ref{ass3}, respectively. We first cite some auxiliary results: Lemma \ref{lem:extreme value} is a confidence band result for i.i.d. Gaussian vectors, Lemma \ref{lemma:localscb} extends this result to sums of dependent variables by using a Gaussian approximation result (Theorem \ref{theorem:gaussapprox_localstat}, cf. \cite{MR2827528}). Theorem \ref{th:scb_paper} is then proven by applying Lemma \ref{lemma:localscb} to a Bahadur representation of $\hat \theta_{b_n}$ from Theorem \ref{theorem:bahadur_paper}.

From Lemma 1 in \cite{zhouwu10}, we adopt the following SCB result for Gaussian random vectors: 
\begin{lemma}\label{lem:extreme value}
Let $F_n(t)=\sum_{i=1}^n \hat K_{b_n}(t_i-t)V_i$, where $V_i, i \in \IZ$ are i.i.d. $N(0,I_{s\times s})$. $b_n \to 0$ and $nb_n/\log^2(n) \to \infty$. Let $m^*= 1/b_n.$ Then
\begin{equation}
\lim_{n \to \infty} \IP\Big( \frac{1}{\sigma_{\hat K,0} \sqrt{nb_n}}\sup_{t \in \sT_n}|F_n(t)|- B_{\hat K}(m^*) \leq \frac{u}{\sqrt{2 \log(m^*)}}\Big)= \exp (-2 \exp(-u) ). \label{eq:extreme value result}
\end{equation}

\noindent where $B_{\hat K}$ is defined in (\ref{eq:Bk}).
\end{lemma}

 For the following results, let us assume that there exists some measurable function $\tilde H(\cdot,\cdot)$ such that for each $t \in [0,1]$, $\tilde h_i(t) = \tilde H(t,\sF_i) \in \IR^s$ is well-defined. Put $S_{\tilde h}(i) := \sum_{j=1}^{i}\tilde h_j(j/n)$.
\begin{theorem}[Theorem 1 and Corollary 2 from \citet{MR2827528}]\label{theorem:gaussapprox_localstat} Assume that for each component $j = 1,\ldots,s$:
\begin{enumerate}
    \item[(a)] $\sup_{t \in [0,1]}\|\tilde h_{0}(t)_j\|_{2+\varsigma} < \infty$,
    \item[(b)] $\sup_{t\not= t'\in [0,1]}\|\tilde h_{0}(t)_j - \tilde h_{0}(t')_j\|_2 / |t-t'| < \infty$,
    \item[(c)] $\sup_{t\in[0,1]}\delta_{2+\varsigma}^{\tilde h(t)_j}(k) = O(k^{-(\gamma+1)})$ with some $\gamma \ge 1$. 
\end{enumerate}
for some $\varsigma\leq 2$. Then on a richer probability space, there are i.i.d. $V_1,V_2,\ldots \sim N(0, I_{s\times s})$ and a process $S_{\tilde h}^{0}(i) = \sum_{j=1}^{i}\Sigma_{\tilde h}(j/n) V_j$ such that $(S_{\tilde h}(i))_{i=1,\ldots,n} \overset{d}{=} (S_{\tilde h}^{0}(i))_{i=1,\ldots,n}$ and
\[
    \max_{i=1,\ldots,n}|S_{\tilde h}(i) - S_{\tilde h}^{0}(i)| = O_{\IP}(\pi_n).
\]
where 
\begin{equation}
\pi_n = n^{(2 \varsigma+2\gamma+\gamma\varsigma)/(2 \varsigma+8\gamma+4\gamma\varsigma)}\log(n)^{2\gamma(3+\varsigma)/(\varsigma+4\gamma+2 \gamma \varsigma)} \label{eq:pin}
\end{equation}

and
\[
    \Sigma_{\tilde h}(t) = \big(\sum_{j\in\IZ}\IE[\tilde h_0(t) \tilde h_j(t)\tran]\big)^{1/2}.
\]
\end{theorem}

\noindent The following lemma is an analogue of Lemma 2 in \cite{zhouwu10}. Since we use other Gaussian approximation rates from Theorem \ref{theorem:gaussapprox_localstat}, we shortly state the proof for completeness.

\begin{lemma}\label{lemma:localscb}
    Let the assumptions and notations from Theorem \ref{theorem:gaussapprox_localstat} hold. Define
    \[
        D_{\tilde h}(t) := (nb_n)^{-1}\sum_{i=1}^{n}\hat K_{b_n}(i/n-t) \tilde h_i(i/n).
    \]
    Assume that $\Sigma_{\tilde h}(t)$ is Lipschitz-continuous and that its smallest eigenvalue is bounded away from 0 uniformly on $[0,1]$. Assume that $\log(n)^4\big(b_n n^{
    (2\gamma+\varsigma \gamma-\varsigma)/(\varsigma+4\gamma+2\gamma \varsigma)}\big)^{-1} \to 0$ and $b_n \log(n)^{3/2} \to 0$. Then
    \begin{equation}
    \lim_{n\to\infty}\IP\Big( \frac{\sqrt{n b_n}}{\sigma_{\hat K,0}} \sup_{t \in \sT_n}\Big| \Sigma_{\tilde h}^{-1}(t)D_{\tilde h}(t)\Big| - B_{\hat K}(m^{*}) \le \frac{u}{\sqrt{2 \log(m^{*})}}\Big) = \exp(-2 \exp(-u)), \label{eq:lemma_scb}
\end{equation}
\end{lemma}
\begin{proof}[Proof of Lemma \ref{lemma:localscb}]
    By Theorem \ref{theorem:gaussapprox_localstat} and summation-by-parts, there exist i.i.d. $V_i \sim N(0, I_{s\times s})$ such that
    \begin{equation}
        \sup_{t\in (0,1)}|D_{\tilde h}(t) - \Xi(t)| = O_{\IP}\Big(\frac{n^{\frac{2 \varsigma+2\gamma+\gamma\varsigma}{2 \varsigma+8\gamma+4\gamma\varsigma}}\log(n)^{\frac{2\gamma(3+\varsigma)}{\varsigma+4\gamma+2 \gamma \varsigma}}}{n b_n}\Big) = O_{\IP}\Big(\frac{\log(n)^2 \big(b_n n^{\frac{2\gamma+\varsigma \gamma-\varsigma}{\varsigma+4\gamma+2\gamma \varsigma}}\big)^{-1/2}}{ (nb_n)^{1/2}\log(n)^{1/2}}\Big),\label{eq:localscb1}
    \end{equation}
    where $\Xi(t) = (nb_n)^{-1}\sum_{i=1}^{n}\hat K_{b_n}(i/n-t)\Sigma_{\tilde h}(i/n) V_i$. Here, \reff{eq:localscb1} is  $o_{\IP}((nb_n)^{-1/2}\log(n)^{-1/2})$ due to
    \[
        \log(n)^4\big(b_n n^{
    (2\gamma+\varsigma \gamma-\varsigma)/(\varsigma+4\gamma+2\gamma \varsigma)}\big)^{-1} \to 0.
    \]
    Since $\Sigma_{\tilde h}(\cdot)$ is Lipschitz continuous by Assumption (b), we can use a standard chaining argument in $t$ (as it was done in Lemma \ref{lemma:biasapproxbahadur} for $\Pi_n(t)$) and the fact that $(nb_n)^{-1}\sum_{i=1}^{n}(\Sigma_{\tilde h}(i/n) - \Sigma_{\tilde h}(t)) \hat K_{b_n}(i/n-t) V_i \sim N(0, v_n)$, with $|v_n|_{\infty} \le C \frac{b_n}{n}$ for some constant $C > 0$ to obtain
    \begin{eqnarray}
        &&\sup_{t\in (0,1)}|\Xi(t) - (nb_n)^{-1}\Sigma_{\tilde h}(t)\sum_{i=1}^{n}\hat K_{b_n}(i/n-t) V_i|\nonumber\\
        &=& \sup_{t\in (0,1)}\big|(nb_n)^{-1}\sum_{i=1}^{n}\hat K_{b_n}(i/n-t)(\Sigma_{\tilde h}(i/n) - \Sigma_{\tilde h}(t)) V_i\big|\nonumber\\
        &=& O_{\IP}\Big(\frac{b_n \log(n)}{(nb_n)^{1/2}}\Big) = O_{\IP}\Big(\frac{b_n \log(n)^{3/2}}{(nb_n)^{1/2}\log(n)^{1/2}}\Big),\label{eq:localscb2}
    \end{eqnarray}
    which is $o_{\IP}((nb_n)^{-1/2}\log(n)^{-1/2})$ due to $b_n \log(n)^{3/2} \to 0$. So the result follows from Lemma \ref{lem:extreme value} in view of \reff{eq:localscb1} and \reff{eq:localscb2}.
\end{proof}

\begin{proof}[Proof of Theorem \ref{th:scb_paper}]
    By Proposition \ref{example:tvrec}, Assumption \ref{ass:case1} implies Assumption \ref{ass1} with arbitrarily large $\gamma > 0$. By Proposition \ref{example:garch},  Assumption \ref{ass:case2} implies Assumption \ref{ass3} with arbitrarily large $\gamma > 0$. We now prove the statement under the more general Assumptions \ref{ass1} or \ref{ass3}, respectively.
    
    Choose $\gamma > 0$ large enough such that  $\frac{2\gamma+\varsigma \gamma-\varsigma}{\varsigma+4\gamma+2\gamma \varsigma} >  \alpha_{exp}$. We therefore have
    \begin{equation}
        \log(n)^4 \big( b_n n^{\frac{2\gamma+\varsigma \gamma-\varsigma}{\varsigma+4\gamma+2\gamma \varsigma}}\big)^{-1} \to 0\label{th:scb_paper_proofeq1}
    \end{equation}
    by assumption.
    
    Let $\tilde k_i(t) := \nabla_{\theta}\ell(\tilde Z_i(t),\theta(t))$ and $\hat K(x) = K(x)$ or $\hat K(x) = K(x) x$, respectively. Define 
    \[
        \Omega_C(t) := (nb_n)^{-1}\sum_{i=1}^{n}\hat K_{b_n}(i/n-t)A_C(i/n)\tran \tilde k_i(i/n)
    \]
    and $D_{\tilde k}(t) = (nb_n)^{-1}\sum_{i=1}^{n}\hat K_{b_n}(i/n-t) \tilde k_i(i/n)$.
    Similar to the discussion of $\Pi_n(t)$ in the proof of Lemma \ref{lemma:biasapproxbahadur} (note that the rates  in \reff{eq:lemmabiasapproxbahadur1} and \reff{eq:lemmabiasapproxbahadur2} then change to $O(b_n)$ instead of $O(b_n^2)$), we can show that
    \begin{equation}
        \sup_{t\in (0,1)}|\Omega_C(t) - A_C(t)\tran\cdot D_{\tilde k}(t)| = O_{\IP}(\beta_n b_n) = O_{\IP}\Big(\frac{b_n^{1/2}\log(n)}{(nb_n)^{1/2}\log(n)^{1/2}}\Big),\label{eq:proofbahadur1}
    \end{equation}
    which is $o_{\IP}((nb_n)^{-1/2}\log(n)^{-1/2})$ since $b_n \log(n)^2 \to 0$.
    
    $\tilde h_i(t) := A_C(t)\tran \tilde k_i(t)$ is a locally stationary process with long-run variance $\Sigma^2_{\tilde h}(t) = \Sigma^2_C(t)$. By Lemma \ref{lemma:localscb} (which is applicable due to \reff{th:scb_paper_proofeq1}), we have that
    \begin{equation}
        \lim_{n\to\infty}\IP\Big( \frac{\sqrt{n b_n}}{\sigma_{\hat K,0}} \sup_{t \in \sT_n}\big| \Sigma_C^{-1}(t)\Omega_C(t)\big| - B_{\hat K}(m^{*}) \le \frac{u}{\sqrt{2 \log(m^{*})}}\Big) = \exp(-2 \exp(-u)).\label{eq:proofbahadur2}
    \end{equation}
    By Theorem \ref{theorem:bahadur_paper}, we have
    \begin{eqnarray}
        && \sup_{t\in \sT_n}\big| V(t) \{\hat \theta_{b_n}(t) - \theta(t)\} - b_n^2 \frac{\mu_{K,2}}{2}V(t) \theta''(t) - D_{\tilde k}(t)\big|\nonumber\\
        &=&  O_{\IP}\big(b_n^3 + (nb_n)^{-1}b_n^{-1/2}\log(n)^{3/2} + (nb_n)^{-1/2}b_n \log(n)\big)\nonumber\\
        &=& O_{\IP}\Big(\frac{(nb_n^7 \log(n))^{1/2} + (nb_n^2 \log(n)^{-4})^{-1/2} + b_n \log(n)^{3/2}}{(nb_n)^{1/2}\log(n)^{1/2}}\Big)\label{eq:applicatbahadur1},
    \end{eqnarray}
    which is $o_{\IP}((nb_n)^{-1/2}\log(n)^{-1/2})$ since $n b_n^7 \log(n) \to 0$, $n b_n^{2} \log(n)^{-4} \to \infty$ and $b_n \log(n)^2 \to 0$. Together with \reff{eq:proofbahadur1} and \reff{eq:proofbahadur2} (with $\hat K = K$), this implies \reff{eq:scb_paper}.
    
\end{proof}

\bibliographystyle{imsart-nameyear}
\bibliography{tvarch}


\newpage
\textbf{Supplement:} This material  contains the remaining proofs of the results in the paper.

\setcounter{section}{2}
\section{Remaining proofs and intermediate lemmata for the proofs of the main theorems}

In this section, we give the proofs for the Propositions \ref{prop:v_i_estimation} and \ref{prop:prop2} in Section \ref{sec:implementation}. Moreover, we provide the remaining high-level lemmas for the proofs of the main results Theorem \ref{theorem:bahadur_paper} and Theorem \ref{th:scb_paper} of the paper in Section \ref{sec:intermediatelemma} below.

\begin{proof}[Proof of Proposition \ref{prop:v_i_estimation}]
 (i) Lemma \ref{lemma:empprocess2}(i),(iii), Lemma \ref{lemma:bias} and the notation therein applied to $g = \nabla_{\theta}^2\ell$ imply
\begin{eqnarray}
    &&\sup_{t\in \sT_n}|\hat\mu_{K,0,b_n}(t)\hat V_{b_n}(t) - \hat\mu_{K,0,b_n}(t)V(t)|\nonumber\\
    &\le& \sup_{t \in \sT_n, \eta \in E_n}|G_n^c(t,\eta) - \hat G_n(t,\eta)| + \sup_{t \in \sT_n, \eta \in E_n}|\hat G_n(t,\eta)|\nonumber\\
    &&\quad\quad + \sup_{t \in \sT_n, \eta \in E_n}|\IE \hat B_n(t,\eta) - V^{\circ}(t,\eta)| +\sup_{t \in \sT_n}|V^{\circ}(t,\hat\eta_{b_n}) - \hat\mu_{K,0,b_n}(t)V(t)|\nonumber\\
    &=& O_{\IP}( (nb_n)^{-1}) + o_{\IP}(\beta_n) + O(b_n) + \sup_{t \in \sT_n}|V^{\circ}(t,\hat\eta_{b_n}) - \hat\mu_{K,0,b_n}(t)V(t)|.\label{proof:v_i_est_eq1}
\end{eqnarray}
We obtain similar as in the proof of Theorem \ref{theorem:bahadur_paper}
 (\reff{eq:mlerate} therein) that
\[
    \sup_{t \in \sT_n}|\hat \eta_{b_n}(t) - \eta_{b_n}(t)| = O_{\IP}((nb_n)^{-1/2}\log(n) + (nb_n)^{-1} + \beta_n b_n^2 + b_n^{2}).
\]
Since $\eta \mapsto V^{\circ}(t,\eta)$ is Lipschitz continuous by Lemma \ref{lemma:hoelder}, the result follows from \reff{proof:v_i_est_eq1} and $b_n \log(n)\to 0$.

(ii) follows similarly due to $\nabla_{\theta}\ell \cdot \nabla_{\theta}\ell\tran \in \sH(2M, \chi, \bar{\bar C})$ with some $\bar{\bar C} > 0$.

\end{proof}

\begin{proof}[Proof of Proposition \ref{prop:prop2}]
    We proceed similar as in the proof of Theorem \ref{theorem:bahadur_paper}
    . Now we use the explicit result of Lemma \ref{lemma:bias}(a) applied to $g = \ell$ (both for Assumption \ref{ass1} and \ref{ass3}), we obtain
    \[
        \sup_{t\in (0,1)}\sup_{\eta \in E_n}|L_{n,b_n}^{\circ}(t,\eta) - \tilde L_{b_n}^{\circ}(t,\eta)| = O_{\IP}(\beta_n + (nb_n)^{-1}) + O(b_n),
    \]
    where $\tilde L_{b_n}^{\circ}(t,\eta) = \int_{-t/b_n}^{(1-t)/b_n}K(x) L(t,\eta_1 + \eta_2 x) d x$. By optimality of $\hat \eta_{b_n}(t)$,
    \begin{eqnarray*}
        0 &\le& L_{n,b_n}^{\circ}(t,\theta(t)) - L_{n,b_n}^{\circ}(t,\hat\eta_{b_n}(t))\\
        &\le& \tilde L_{b_n}^{\circ}(t,\theta(t)) - \tilde L_{b_n}^{\circ}(t,\hat\eta_{b_n}(t)) + 2\sup_{\eta \in E_n} |L_{n,b_n}^{\circ}(t,\eta) - \tilde L_{b_n}^{\circ}(t,\eta)|.
    \end{eqnarray*}
    This implies
    \begin{eqnarray}
       && \min\Big\{\int_{-1}^{0}K(x) \big\{L(t,\hat \theta_{b_n}(t) + b_n\widehat \theta_{b_n}'(t) x)-L(t, \theta(t)) \big\}  dx,\nonumber\\
        &&\quad\quad \int_{0}^{1}K(x) \big\{L(t,\hat \theta_{b_n}(t) + b_n\widehat \theta_{b_n}'(t) x)-L(t, \theta(t)) \big\}  dx\Big\} \le 2\sup_{\eta \in E_n} |L_{n,b_n}^{\circ}(t,\eta) - \tilde L_{b_n}^{\circ}(t,\eta)|.\label{eq:proofboundary}
    \end{eqnarray}
    Assume that for some $\iota > 0$, $\limsup_{n\to\infty}\sup_{t\in (0,1)}|\hat \eta_{b_n}(t) - (\theta(t)\tran,0)\tran| \ge \iota$. Then there exists $t \in (0,1)$ such that either (c1)
    \[
        |\hat \theta_{b_n}(t) - \theta(t)| \ge \frac{1}{2}|b_n \widehat \theta_{b_n}'(t)|
    \]
    and thus $|\hat \theta_{b_n}(t) - \theta(t)| > \iota/3$, or (c2)
    \[
        |\hat \theta_{b_n}(t) - \theta(t)| < \frac{1}{2}|b_n \widehat \theta_{b_n}'(t)|,
    \]
    and thus $|b_n \widehat \theta_{b_n}'(t)| > 2\iota/3$.\\
    In case (c1), we have $|\hat \theta_{b_n}(t) + b_n\widehat \theta_{b_n}'(t) x - \theta(t)| \ge |\hat \theta_{b_n}(t) - \theta(t)| - |x| |b_n\widehat \theta_{b_n}'(t)| \ge \frac{\iota}{6}$ for $x \in [0,\frac{1}{4}]$, thus with some $c_0 > 0$,
    \[
        \int_{0}^{1}K(x) \big\{L(t,\hat \theta_{b_n}(t) + b_n\widehat \theta_{b_n}'(t) x)-L(t, \theta(t)) \big\}  dx \ge \int_{0}^{1/4}K(x) \big\{L(t,\hat \theta_{b_n}(t) + b_n\widehat \theta_{b_n}'(t) x)-L(t, \theta(t)) \big\}  dx \ge c_0
    \]
    since $\theta \mapsto L(t,\theta)$ is continuous and attains its unique minimum at $\theta = \theta(t)$.\\
    In case (c2), we have $|\hat \theta_{b_n}(t) + b_n\widehat \theta_{b_n}'(t) x - \theta(t)| \ge |x| |b_n\widehat \theta_{b_n}'(t)| - |\hat \theta_{b_n}(t) - \theta(t)| \ge \frac{\iota}{6}$ for $x \in [\frac{3}{4},1]$, thus with some $c_0 > 0$,
    \[
        \int_{0}^{1}K(x) \big\{L(t,\hat \theta_{b_n}(t) + b_n\widehat \theta_{b_n}'(t) x)-L(t, \theta(t)) \big\}  dx \ge \int_{3/4}^{1}K(x) \big\{L(t,\hat \theta_{b_n}(t) + b_n\widehat \theta_{b_n}'(t) x)-L(t, \theta(t)) \big\}  dx \ge c_0.
    \]
    In both cases, \reff{eq:proofboundary} becomes a contradiction. Therefore,
    \[
        \sup_{t\in (0,1)}|\hat \eta_{b_n}(t) - \eta_{b_n}(t)| = o_{\IP}(1).
    \]

    Using summation-by-parts and Gaussian approximation similar to that presented in Theorem \ref{theorem:gaussapprox_localstat} for the process $\nabla_{\theta}\ell(\tilde Z_i(i/n), \theta(i/n))$, there exists i.i.d. $V_1,V_2,\ldots \sim N(0, I_{s\times s})$ on a richer probability space such that, for $\pi_n$ as in \reff{eq:pin} 
    
    \begin{equation}
    \sup_{t\in (0,1)}\big|(nb_n)^{-1}\sum_{i=1}^{n}K_{b_n}(i/n-t) (\nabla_{\theta}\ell(\tilde Z_i(i/n), \theta(i/n))-V_i)\big\}| = O_{\IP}((nb_n)^{-1}\pi_n)=O_{\IP}( (nb_n)^{-1/2}\log(n)). 
    \end{equation}
    Thus one can replace $\sup_{t \in \sT_n}$ by $\sup_{t \in (0,1)}$ in \reff{eq:glmkonv100}. A careful examination of the rest of the proof of Theorem \ref{theorem:bahadur_paper} (with Lemma \ref{lemma:bias2}\reff{bias2:eq1} replaced by Lemma \ref{lemma:bias2}\reff{bias2:eq2}) now yields the result
    \begin{equation}
       \sup_{t\in (0,1)} |\tilde V_{b_n}^{\circ}(t)\cdot (\hat \eta_{b_n}(t) - \eta_{b_n}(t)) - \nabla_{\eta} L_{n,b_n}^{\circ,c}(t,\eta_{b_n}(t))\big| = O_{\IP}(\tau_n^{(1)}),\label{eq:boundary2}
    \end{equation}
    where (we shortly write $\hat \mu_{K,j}(t) = \hat \mu_{K,j,b_n}(t)$)
    \[
        \tilde V_{b_n}^{\circ}(t) = \begin{pmatrix}
            \hat \mu_{K,0}(t) & \hat \mu_{K,1}(t)\\
             \hat \mu_{K,1}(t) & \hat \mu_{K,2}(t)
        \end{pmatrix} \otimes V(t).
    \]

    By Lemma \ref{lemma:empprocess2}(i), Lemma \ref{lemma:bias2} and Lemma \ref{lemma:biasapproxbahadur}, we obtain furthermore with $U_{i,n}(t) = (K_{b_n}(i/n-t), K_{b_n}(i/n-t)\cdot (i/n-t) b_n^{-1})\tran$:
    \begin{eqnarray}
        && \sup_{t\in (0,1)}\big| \nabla_{\eta} L_{n,b_n}^{\circ,c}(t,\eta_{b_n}(t)) - b_n^2 \begin{pmatrix}
            \hat \mu_{K,2}(t)\\
            \hat \mu_{K,3}(t)
        \end{pmatrix} \otimes [V(t) \theta''(t)]\nonumber\\
        &&\quad\quad\quad\quad - (nb_n)^{-1}\sum_{i=1}^{n}U_{i,n}(t) \otimes \nabla_{\theta}\ell(\tilde Z_i(i/n),\theta(i/n)) \big| =  O_{\IP}(\beta_n b_n^2 + b_n^3 + (nb_n)^{-1}).\label{eq:boundary3}
    \end{eqnarray}
    Recalling the proof of Lemma \ref{lemma:localscb}, \reff{eq:localscb1} and \reff{eq:localscb2} and the proof of Theorem \ref{th:scb_paper}, \reff{eq:proofbahadur1} we see that there exist i.i.d. $V_i \sim N(0, I_{s\times s})$ such that both for $\hat K = K$ and $\hat K(x) = K(x) \cdot x$,
    \begin{eqnarray}
       &&  \sup_{t\in (0,1)}\big|A_C(t)\tran (nb_n)^{-1}\sum_{i=1}^{n}\hat K_{b_n}(i/n-t)\nabla_{\theta}\ell(\tilde Z_i(i/n),\theta(i/n))\nonumber\\
       &&\quad\quad\quad\quad -  \Sigma_C(t) (nb_n)^{-1}\sum_{i=1}^{n}\hat K_{b_n}(i/n-t) V_i\big|\nonumber\\
       &=& O_{\IP}\Big(\frac{\log(n)^2 \big(b_n n^{\frac{2\gamma+\varsigma \gamma-\varsigma}{\varsigma+4\gamma+2\gamma \varsigma}}\big)^{-1/2}}{ (nb_n)^{1/2}\log(n)^{1/2}} + \frac{b_n \log(n)^{3/2}}{(nb_n)^{1/2}\log(n)^{1/2}}+\frac{b_n^{1/2}\log(n)}{(nb_n)^{1/2}\log(n)^{1/2}}\Big)\nonumber\\
       &&\quad\quad\quad\quad =: O_{\IP}(w_n).\label{eq:boundary4}
    \end{eqnarray}
    With \reff{eq:boundary2} and
    \begin{eqnarray*}
        \tilde V_{b_n}^{\circ}(t)^{-1} &=& \begin{pmatrix}
            \hat \mu_{K,0}(t) & \hat \mu_{K,1}(t)\\
             \hat \mu_{K,1}(t) & \hat \mu_{K,2}(t)
        \end{pmatrix}^{-1} \otimes V(t)^{-1}\\
        &=& \frac{1}{\hat \mu_{K,2}(t) N_{b_n}^{(0)}(t)}\begin{pmatrix}
            \hat \mu_{K,2}(t) V(t)^{-1} & -\hat \mu_{K,1}(t) V(t)^{-1}\\
            -\hat \mu_{K,1}(t)V(t)^{-1} & \hat \mu_{K,0}(t) V(t)^{-1}
        \end{pmatrix},
    \end{eqnarray*}
    we obtain:
    \begin{eqnarray*}
        && \sup_{t\in (0,1)}\Big| N_{b_n}^{(0)}(t)\cdot \{\hat \theta_{b_n,C}(t) - \theta_C(t)\}\\
        &&\quad\quad\quad\quad - \Big[ A_C(t)\tran \nabla_{\eta_1} L_{n,b_n}^{\circ,c}(t,\eta_{b_n}(t)) -\frac{\hat \mu_{K,1}(t)}{\hat \mu_{K,2}(t)} A_C(t)\tran \nabla_{\eta_2} L_{n,b_n}^{\circ,c}(t,\eta_{b_n}(t))\Big]\Big| = O_{\IP}(\tau_n^{(1)}).
    \end{eqnarray*}
    With \reff{eq:boundary3} and \reff{eq:boundary4}, we have
    \begin{eqnarray*}
        && \sup_{t\in (0,1)}\Big|N^{(0)}_{b_n}(t) \cdot \{\hat \theta_{b_n,C}(t) - \theta_C(t)\}\\
        &&\quad\quad\quad\quad\quad + b_n^2 N^{(1)}_{b_n}(t) \theta_{C}''(t) - \Sigma_C(t)\big\{ Q_{b_n}^{(0)}(t) - \frac{\hat \mu_{K,1}(t)}{\hat \mu_{K,2}(t)} Q_{b_n}^{(1)}(t)\big\}\Big|\\
        &=& O_{\IP}(\tau_n^{(1)} + (\beta_n b_n^2 + b_n^3 + (nb_n)^{-1}) + w_n),
    \end{eqnarray*}
    which finishes the proof.
\end{proof}

\subsection{Intermediate Lemmas for $\hat L_{n,b_n}^{\circ}$}\label{sec:intermediatelemma}
In this section, we show some lemmas for $\hat L_{n,b_n}^{\circ}$ which are needed to prove the main results. To do so, we make use of the elementary lemmas derived in Section \ref{sec:elementaryresults}. Lemma \ref{lemma:bias2} derives a bias expansion of $\nabla_{\eta}\hat L_{n,b_n}^{\circ}(t,\eta_{b_n}(t))$, Lemma \ref{lemma:biasapproxbahadur} shows an approximation of $\nabla_{\eta}\hat L_{n,b_n}^{\circ}(t,\eta_{b_n}(t))$ by a localized sum of $\nabla_{\theta}\ell(\tilde Z_i(i/n),\theta(i/n))$, that is, the lemma justifies the replacement of the locally stationary process $Z_i$ in $\nabla_{\eta}\hat L_{n,b_n}^{\circ}(t,\eta_{b_n}(t))$ by $\tilde Z_i(i/n)$ with a certain convergence rate. Lemma \ref{lemma:property_lipschitz_pi} discusses Lipschitz properties of a quantity $\Pi_n$ which occurs in the proof of Lemma \ref{lemma:biasapproxbahadur}.

\begin{lemma}\label{lemma:bias2}
    Let $\eta_{b_n}(t) = (\theta(t)\tran, b_n \theta'(t)\tran)\tran$. Let Assumption \ref{ass1} hold with $r = 1$ or let Assumption \ref{ass3} hold with $r = 2+\varsigma$, $\varsigma > 0$.
        Then uniformly in $t \in \sT_n$,
    \begin{equation}
        \IE \nabla_{\eta_1}\hat L_{n,b_n}^{\circ}(t,\eta_{b_n}(t)) = b_n^2 \frac{\mu_{K,2}}{2}V(t) \theta''(t) + O(b_n^3 + (nb_n)^{-1}).\label{bias2:eq1}
    \end{equation}
    Furthermore, it holds uniformly in $t\in (0,1)$ that
    \begin{equation}
        \IE \nabla_{\eta}\hat L_{n,b_n}^{\circ}(t,\eta_{b_n}(t)) = \frac{b_n^2}{2}\int_{-t/b_n}^{(1-t)/b_n} K(x)\begin{pmatrix}
             x^2\\
             x^3
        \end{pmatrix} dx \otimes [V(t) \theta''(t)] + O(b_n^3 + (nb_n)^{-1}).\label{bias2:eq2}
    \end{equation}
\end{lemma}
\begin{proof}[Proof of Lemma \ref{lemma:bias2}]
Let $U_{i,n}(t) = (K_{b_n}(i/n-t), K_{b_n}(i/n-t) (i/n-t) b_n^{-1})\tran$. By a Taylor expansion of $\theta(i/n)$ around $t$, we have
 \[
    \theta(i/n) = \theta(t) + \theta'(t)(i/n-t) + r_n(t),
 \]
 where $r_n(t) = \theta''(t)\frac{(i/n-t)^2}{2} + \theta'''(\tilde t) \frac{(i/n-t)^3}{6}$ and $\tilde t$ is between $t$ and $i/n$. We conclude that
 \begin{eqnarray}
    &&\nabla_{\eta}\hat L_{n,b_n}^{\circ}(t,\eta_{b_n}(t)) - (nb_n)^{-1}\sum_{i=1}^{n}U_{i,n}(t) \otimes \nabla_{\theta}\ell(\tilde Z_i(i/n), \theta(i/n))\nonumber\\
    &=& (nb_n)^{-1}\sum_{i=1}^{n}U_{i,n}(t) \otimes \big\{\int_0^{1} \nabla_{\theta}^2\ell(\tilde Z_i(i/n), \theta(i/n) + s r_n(t)) d s \cdot r_n(t)\big\}\label{proofbias:eq1}.
 \end{eqnarray}
Using $\nabla_{\theta}^2\ell \in \sH(M,\chi,\bar C)$ (if Assumption \ref{ass1} holds) or $\nabla_{\theta}^2\ell \in \sH(M(1+s),\chi,\bar C)$ with $s > 0$ small enough (if Assumption \ref{ass3} holds), we obtain with Lemma \ref{lemma:hoelder} for $|i/n-t| \le b_n$:
\begin{equation}
        \|\nabla_{\theta}^2\ell(\tilde Z_i(i/n), \theta(i/n) + s r_n(t)) - \nabla_{\theta}^2 \ell(\tilde Z_i(t),\theta(t))\|_1 = O(b_n + n^{-1}).\label{proofbias:eq3}
\end{equation}
Using \reff{proofbias:eq1}, $\IE \nabla_{\theta}\ell(\tilde Z_i(i/n),\theta(i/n)) = 0$ (which follows from Assumption \ref{ass1}\ref{ass1_smooth},\ref{ass1_model} or Assumption \ref{ass3}\ref{ass3_smooth}, \ref{ass3_model}) and \reff{proofbias:eq3}, we obtain
\begin{eqnarray}
    && \IE \nabla_{\eta} \hat L_{n,b_n}^{\circ}(t,\eta_{b_n}(t))\nonumber\\
    &=& (nb_n)^{-1}\sum_{i=1}^{n}U_{i,n}(t) \otimes \big\{\IE \nabla_{\theta}^2\ell(\tilde Z_i(t), \theta(t)) \cdot \theta''(t) \frac{(i/n-t)^2}{2}\big\} + O(b_n^3 + n^{-1})\nonumber\\
    &=& \frac{b_n^2}{2}\int_{-t/b_n}^{(1-t)/b_n}K(x)\begin{pmatrix}
       x^2\\
        x^3
    \end{pmatrix} dx \otimes [V(t) \theta''(t)] + O(b_n^3 + n^{-1} + (nb_n)^{-1}),\label{mle_consistency_step4}
\end{eqnarray}
which shows \reff{bias2:eq2}. Equation \reff{bias2:eq1} follows since $t\in \sT_n$ and  symmetry of $K$ imply
\[
    \int_{-t/b_n}^{(1-t)/b_n}K(x)\begin{pmatrix}
       x^2\\
        x^3
    \end{pmatrix} dx = \int_0^{1}K(x)\begin{pmatrix}
        x^2\\
        x^3
    \end{pmatrix} dx = \begin{pmatrix}
        \mu_{K,2}\\
        0
    \end{pmatrix}.
\]
\end{proof}

\begin{lemma}[Lipschitz properties of $\Pi_n$]\label{lemma:property_lipschitz_pi} Let $s \ge 0$. Suppose that Assumption \ref{ass1} holds with $r \ge 1$ or Assumption \ref{ass3} holds with $r > 1$.

Define
\[
    \Pi_n(t) := (nb_n)^{-1}\sum_{i=1}^{n}(M_i^{(2)}(t,i/n) - \IE M_i^{(2)}(t,i/n)),
\]
where
    \[
        M_i^{(2)}(t,u) = \hat K_{b_n}(u-t)\cdot \int_0^{1}M_i(t,u) d s \cdot d_u(t),
    \]
    $M_i(u,t) = \nabla_{\theta}^2 \ell(\tilde Z_i(u),\theta(t) + s d_u(t))$ and $d_u(t) = \theta(u) - \theta(t) - (u-t) \theta'(t)$. Then there exist come constants $\tilde C, \iota' > 0$ such that
\[
\Big\|\sup_{t\not=t', |t-t'| < \iota'}\frac{|\Pi_n(t) - \Pi_n(t')|}{|t - t'|_1}\Big\|_1 \le \tilde C.
\]
\end{lemma}
\begin{proof}[Proof of Lemma \ref{lemma:property_lipschitz_pi}]
    We have
    \begin{eqnarray*}
        && |M_i^{(2)}(t,u) - M_i^{(2)}(t',u)|\\
        &\le& |\hat K_{b_n}(u-t) - \hat K_{b_n}(u-t')|\cdot |M_i(t,u)|\cdot |d_u(t)|\\
        &&\quad\quad + |\hat K_{b_n}(u-t')|\cdot |M_i(t,u) - M_i(t',u)|\cdot |d_u(t)|\\
        &&\quad\quad + |\hat K_{b_n}(u-t')|\cdot |M_i(t',u)|\cdot |d_u(t) - d_u(t')|.
    \end{eqnarray*}
    If Assumption \ref{ass1} holds, we have $g = \nabla_{\theta}^2\ell \in \sH(M,\chi,\bar C)$. Elementary calculations show that 
    \begin{eqnarray*}
        |M_i(t,u)| &\le& \sup_{\theta \in \Theta}|g(\tilde Z_i(u),\theta)|,\\
        |M_i(t,u) - M_i(t',u)| &\le& \sup_{\theta \in \Theta}\frac{|g(\tilde Z_i(u),\theta) - g(\tilde Z_i(u),\theta')|}{|\theta - \theta'|_1}\cdot \{|\theta(t) - \theta(t')|_1 + |d_u(t) - d_u(t')|_1\},
    \end{eqnarray*}
    As long as $|t-u| < 1$ and $|t-t'|$ is small enough, we obtain $|t'-u| \le 1$. So in the case that either $|t-u| < 1$ or $|t'-u| < 1$, Lipschitz continuity of $\theta(\cdot), \theta'(\cdot)$ implies that there exists some constant $\tilde C > 0$ such that $|d_u(t) - d_u(t')|_1 \le \tilde C |t-t'|, |\theta(t) - \theta(t')|_1 \le \tilde C |t-t'|$, $|d_u(t)|_1 \le \tilde C$.
    
    This implies
    \begin{eqnarray}
        |M_i^{(2)}(t,u) - M_i^{(2)}(t',u)| &\le& \tilde C b_n^{-1}L_{\hat K}\sup_{\theta \in \Theta}|g(\tilde Z_i(u),\theta)|\cdot |t-t'|\nonumber\\
        &&\quad + 2|\hat K|_{\infty}\tilde C^2  \cdot \sup_{\theta \in \Theta}\frac{|g(\tilde Z_i(u),\theta) - g(\tilde Z_i(u),\theta')|}{|\theta - \theta'|_1} |t-t'|\nonumber\\
        &&\quad + |\hat K|_{\infty}\tilde C\cdot \sup_{\theta \in \Theta}|g(\tilde Z_i(u),\theta)|\cdot |t-t'|.\label{eq:property_lipschitz_pi_1}
    \end{eqnarray}
    With Lemma \ref{lemma:hoelder}(i) we obtain the result.
    
    Suppose now that Assumption \ref{ass3} holds. As long as $|t-t'|$ is small enough and $n$ is large enough,
    $|u-t| \le b_n$ (or $|u-t'| \le b_n$) and the twice differentiability of $\theta(\cdot)$ imply that $\sup_{\nu\in[0,1]}|\theta(u) - (\theta(t) + \nu d_u(t))|_1 < \iota$, $\sup_{\nu\in[0,1]}|\theta(u) - (\theta(t') + \nu d_u(t'))|_1 < \iota$. Put $\tilde \ell_{\tilde \theta}(y,x,\theta) = g(F(x,\tilde \theta,y),x,\theta)$ and $\tilde g = \nabla_{\theta}^2\tilde \ell$. By Assumption \ref{ass3}, $\tilde g \in \sH_{\iota}^{mult}(M(1+s),\chi^{(s)},\bar C^{(s)})$ for all $s > 0$ small enough. Then
    \begin{eqnarray*}
        |M_i(t,u)| &\le& \sup_{|\theta - \theta(u)|_1<\iota}|\tilde g_{\theta(u)}(\zeta_i,\tilde X_i(u),\theta)|,\\
        |M_i(t,u) - M_i(t',u)| &\le& \bar C \cdot \sup_{\theta\not=\theta',|\theta-\theta(u)|_1 < \iota, |\theta' - \theta(u)|_1 < \iota}\frac{|\tilde g_{\theta(u)}(\zeta_i,\tilde X_i(u),\theta) - \tilde g_{\theta(u)}(\zeta_i,\tilde X_i(u),\theta')|}{|\theta - \theta'|_1}\\
        &&\quad\quad\times \{|\theta(t) - \theta(t')|_1 + |d_u(t) - d_u(t')|_1,
    \end{eqnarray*}
    giving appropriate results for \reff{eq:property_lipschitz_pi_1} and thus the assertion with Lemma \ref{lemma:hoeldergarch}.
\end{proof}

\begin{lemma}\label{lemma:biasapproxbahadur}
    Let $U_{i,n}(t) := K_{b_n}(i/n-t)\cdot (1, (i/n-t)b_n^{-1})\tran$. Let Assumption \ref{ass1} or \ref{ass3} hold with some $r = 2 + \varsigma$, $\varsigma > 0$. Then it holds that
    \begin{eqnarray*}
        &&\sup_{t\in (0,1)}\big|\nabla_{\eta}\hat L_{n,b_n}^{\circ}(t,\eta_{b_n}(t)) - \IE \nabla_{\eta}\hat L_{n,b_n}  ^{\circ}(t,\eta_{b_n}(t))\\
        &&\quad\quad\quad\quad\quad - (nb_n)^{-1}\sum_{i=1}^{n}U_{i,n}(t) \otimes \nabla_{\theta}\ell(\tilde Z_i(i/n),\theta(i/n))\big| = O_{\IP}(\beta_n b_n^2).
    \end{eqnarray*}
\end{lemma}
\begin{proof}
Note that $\IE \nabla_{\theta}\ell(\tilde Z_i(i/n),\theta(i/n)) = 0$ by Assumption \ref{ass1}\ref{ass1_smooth},\ref{ass1_model} or Assumption \ref{ass3}\ref{ass3_smooth},\ref{ass3_model}. Put
\begin{eqnarray*}
&&\Pi_{n}(t)\\
&:=& (nb_n)^{-1}\sum_{i=1}^{n}U_{i,n}(t) \otimes \big\{ [\nabla_{\theta}\ell(\tilde Z_i(i/n), \theta(t) +  (i/n-t)\theta'(t)) - \nabla_{\theta}\ell(\tilde Z_i(i/n), \theta(i/n))]\\
&&\quad\quad\quad\quad\quad\quad\quad\quad- \IE[\nabla_{\theta}\ell(\tilde Z_i(i/n), \theta(t) + (i/n-t)\theta'(t)) - \nabla_{\theta}\ell(\tilde Z_i(i/n), \theta(i/n))]\}.
\end{eqnarray*}
We have to prove that $\sup_{t\in \sT_n}\big|\Pi_{n}(t)\big| = O_{\IP}(\delta_n b_n^2$). 
Define $M_i(t,u) := \int_0^{1}\nabla_{\theta}^2 \ell(\tilde Z_i(u), \theta(t) + s(\theta(u) - \theta(t) - (u-t)\theta'(t))) d s$ and $ M_i^{(2)}(t,u) = U_{i,n}(t) \otimes \big\{M_i(t,u) \{\theta(u)-\theta(t) - (u-t)\theta'(t)\}\big\}$. By a Taylor expansion of $\nabla_{\theta}\ell$ w.r.t. $\theta$, we have
\[
    \Pi_{n}(t) = (nb_n)^{-1}\sum_{i=1}^{n}(M_i^{(2)}(t,i/n) - \IE M_i^{(2)}(t,i/n)).
\]
We now apply a similar technique as in the proof of Lemma \ref{lemma:empprocess2}(iii), namely we use a chaining argument similar to \reff{eq:proof_strategy_emp} to prove
\[
    \IP\big(\sup_{t \in (0,1)}|\Pi_n(t)| > Q\beta_n b_n^2\big) \to 0,
\]
for some $Q > 0$ large enough. Define the discretization $\sT_{n,r} := \{l/r: l = 1,\ldots,r\}$ with $r = n^5$. By Lemma \ref{lemma:property_lipschitz_pi}, we have with Markov's inequality for $Q > 0$:
\[
    \IP\big(\sup_{|t-t'| \le r^{-1}}|\Pi_n(t) - \Pi_n(t')| > Q \beta_n b_n^2/2\big) = O\big(\frac{b_n^{-2}r^{-1}}{\beta_n b_n^2}\big),
\]
which converges to 0. Choose $\alpha = 1/2$.  By Lemma \ref{lemma:depmeasure}(iii) or Lemma \ref{lemma:depmeasuregarch}(iii) applied with $q = 2 + s$ ($s$ small enough), we obtain that  $\sup_{u}\Delta_{2+s}^{\sup_{t}|M^{(2)}(t,u)|}(k) = O(k^{-(1+\gamma)})$. Thus
    \begin{eqnarray}
        \tilde W_{2+s,\alpha} &:=& \sup_{u\in[0,1]} \sup_{t\in[0,1]}\| \sup_{t,\eta}|M_i^{(2)}(t,u)| \|_{2+\varsigma,\alpha} = \sup_{m \ge 0} (m+1)^{\alpha}\Delta_{2+s}^{\sup_{t}|M^{(2)}(t,u)|}(m)\nonumber\\
        &=& O(b_n^2)\label{eq:lemmabiasapproxbahadur1}
    \end{eqnarray}
    (the constant being independent of $n$) and
    \begin{eqnarray}
        \tilde W_{2,\alpha} &:=& \sup_{t,u}\|M_i^{(2)}(t,u)\|_{2,\alpha} = \sup_{m\ge 0}(m+1)^{\alpha} \sup_{u\in[0,1]}\sup_{t}\Delta_{2}^{M^{(2)}(t,u)}(m)\nonumber\\
        &=& O(b_n^2)\label{eq:lemmabiasapproxbahadur2}
    \end{eqnarray}
    (the constant being independent of $n$). We now apply Theorem 6.2 of \cite{zhangwu2017} (the proof therein also works for the uniform functional dependence measure) with $q = 2+s$, $\alpha = 1/2$ to $(M_i^{(2)}(t,i/n))_{t \in \sT_{n,r}}$, where $l = 1 \vee \#(\sT_{n,r}) \le 5 \log(n)$. For $Q$ large enough, we obtain with some constant $C_{\alpha,s} > 0$:
\begin{eqnarray*}
        &&\IP\big(\sup_{t' \in \sT_{n,r}} |\Pi_{n}(t')| \ge Q\beta_n b_n^2/2\big)\\
        &\le& \frac{C_{\alpha,s}n\cdot l^{1+s/2}\tilde W_{2+s,\alpha}^{2+s}}{(Q/2)^{2+s}(\beta_n b_n^2 (nb_n))^{2+s}}+C_{\alpha,s} \exp \Big( -\frac{C_{\alpha,s}(Q/2)^2 (\beta_n b_n^2 (n b_n))^2}{n \tilde W_{2,\alpha}^2}\Big)\\
        &\lesssim&  n^{-\varsigma/2}+ \exp\big(-\frac{(n b_n) b_n^{-1} \log(n)}{n} \big)\\
        &\to& 0,
    \end{eqnarray*}
which finishes the proof.
\end{proof}

\section{Elementary Results}
\label{sec:elementaryresults}

This section summarizes some basic results for H\"older-continuous functions $g(Z_i,\theta)$ of the observations $Z_i = (Y_i,X_i)$, $X_i = (Y_{j}:-\infty < j \le i-1)$ and the parameter $\theta\in \Theta$. They are then used in the proofs of the theorems. Depending on Case 1 or Case 2, we introduce different Lemmas which nearly state the same result under different conditions.

\subsection{Uniform upper bounds and chaining results for means of locally stationary processes}

For $t \in (0,1)$ and $\eta \in E_n = \Theta \times (\Theta' \cdot b_n)$ and some Lipschitz continuous function $\hat K$ (Lipschitz constant $L_{\hat K}$) and compact support $[-1,1]$ ($\hat K$ bounded by $|\hat K|_{\infty}$), define $\hat K_{b_n}(\cdot) := \hat K(\cdot/b_n)$ and 
    \begin{equation}
        G_n(t,\eta) := (nb_n)^{-1}\sum_{i=1}^{n}\hat K_{b_n}(i/n-t)\cdot \{g(Z_i,\eta_1 + \eta_2(i/n-t)b_n^{-1}) - \IE g(Z_i,\eta_1 + \eta_2 (i/n-t)b_n^{-1})\}.\label{eq:proof_notation}
    \end{equation}
    Let $G_{n}^c(t,\eta)$, $\hat G_n(t,\eta)$ denote the same quantities but with $Z_i$ replaced by $Z_i^c$ or $\tilde Z_i(i/n)$, respectively.
    
    In this subsection, we derive some basic results for $G_n(t,\eta)$, $G_{n}^c(t,\eta)$, and $\hat G_n(t,\eta)$, respectively. In the proofs of the theorems in the paper, the results are mainly applied to $g = \nabla_{\theta}^k \ell$ with $k \in \{0,1,2\}$. Lemma \ref{lemma:property_lipschitz} summarizes Lipschitz properties of $\hat G_n$ in both Cases 1 and 2, Lemma \ref{lemma:empprocess2} provides results on the stochastic behavior of $\hat G_n$ uniformly in $t,\eta$ based on a simple chaining approach and the Lipschitz results from Lemma \ref{lemma:property_lipschitz}. The last Lemma \ref{lemma:bias} discusses the bias of several terms connected to $\hat G_n$.

\begin{lemma}[Lipschitz properties of $\hat G_n$]\label{lemma:property_lipschitz} Let $s \ge 0$. 
\begin{enumerate}
    \item[(i)] Let $g \in \sH(M(1+s),\chi,\bar C)$. Let Assumption \ref{ass1}\ref{ass1_stat} hold with $r \ge 1+s$. Then there exists some constant $\tilde C > 0$ such that
    \[
        \sup_{t\in[0,1]}\Big\|\sup_{\eta\not=\eta'}\frac{|\hat G_n(t,\eta) - \hat G_n(t,\eta')|}{|\eta - \eta'|_1}\Big\|_1 \le \tilde C,
    \]
    and
    \[
        \Big\|\sup_{t\not=t'}\sup_{\eta\not=\eta'}\frac{|\hat G_n(t,\eta) - \hat G_n(t',\eta')|}{|t-t'| + |\eta - \eta'|_1}\Big\|_1 \le \tilde C b_n^{-2},
    \]
    \item[(ii)] (for tvGARCH) Let $g$ be such that $\tilde g_{\tilde \theta}(y,x,\theta) := g(F(x,\tilde \theta,y),x,\theta)$ fulfills $\tilde g \in \sH_{\iota}^{mult}(M(1+s),\chi^{(s)},\bar C^{(s)})$ with $\chi_i^{(s)} = O(i^{-(1+\gamma)})$. Let Assumption \ref{ass3}\ref{ass3_stat} hold with $r \ge 1 + s$ and let $\theta(\cdot)$ be continuous. Then there exists some constant $\tilde C^{(s)} > 0$ such that
    \[
        \sup_{t\in[0,1]}\Big\|\underset{|\eta - \eta_{b_n}(t)|_1 < \iota/2,|\eta' - \eta_{b_n}(t)|_1 < \iota/2}{\sup_{\eta\not=\eta'}}\frac{|\hat G_n(t,\eta) - \hat G_n(t,\eta')|}{|\eta - \eta'|_1}\Big\|_1 \le \tilde C^{(s)},
    \]
    and
    \[
        \Big\|\sup_{t\not=t'}\underset{|\eta - \eta_{b_n}(t)|_1 < \iota/2,|\eta' - \eta_{b_n}(t')|_1 < \iota/2}{\sup_{\eta\not=\eta'}}\frac{|\hat G_n(t,\eta) - \hat G_n(t',\eta')|}{|t-t'| + |\eta - \eta'|_1}\Big\|_1 \le \tilde C^{(s)} b_n^{-2},
    \]
\end{enumerate}
\end{lemma}
\begin{proof}[Proof of Lemma \ref{lemma:property_lipschitz}]
    (i) Since $g \in \sH(M(1+s),\chi,\bar C)$ and $|i/n-t| \le b_n$ inside the sum, it holds that
    \begin{equation}
        |\hat G_n(t,\eta) - \hat G_n(t,\eta')| \le \bar C |\eta - \eta'|_1 \cdot (nb_n)^{-1}\sum_{i=1}^{n}|\hat K_{b_n}(i/n-t)|\cdot \{2 +|\tilde Z_i(i/n)|_{\hat \chi}^{M(1+s)} + \| |\tilde Z_i(i/n)|_{\hat \chi}^{M(1+s)}\|_1\}\label{eq:property_lipschitz1}
    \end{equation}
    Furthermore, $(nb_n)^{-1}\sum_{i=1}^{n}|\hat K_{b_n}(i/n-t)| \le |\hat K|_{\infty}$. We conclude that
    \[
        \Big\|\sup_{\eta \not= \eta'} \frac{|\hat G_n(t,\eta) - \hat G_n(t,\eta')|}{|\eta - \eta'|_1}\Big\|_1 \le 2\bar C |\hat K|_{\infty}\big\{1 + \sup_i\big\| |\tilde Z_i(i/n)|_{\hat \chi}^{M(1+s)}\big\|_1\big\} \le 2 \bar C |\hat K|_{\infty}(1 + (D|\hat \chi|_1)^{M(1+s)}).
    \]
    
    This yields the first assertion. Since $g \in \sH(M(1+s),\chi,\bar C)$, we have with some constant $\tilde C > 0$:
    \begin{eqnarray*}
        && |\hat G_n(t,\eta) - \hat G_n(t',\eta')|\\
        &\le& (nb_n)^{-1}\sum_{i=1}^{n}|\hat K_{b_n}(i/n-t) - \hat K_{b_n}(i/n-t')| \cdot \sup_{\theta}\{|g(\tilde Z_i(i/n),\theta)| + \| g(\tilde Z_i(i/n),\theta)\|_1\}\\
        &&+  (nb_n)^{-1}\sum_{i=1}^{n}|\hat K_{b_n}(i/n-t')|\cdot |g(\tilde Z_i(i/n),\eta_1 + \eta_2(i/n-t)b_n^{-1}) - g(\tilde Z_i(i/n), \eta_1' + \eta_2' (i/n-t')b_n^{-1})|\\
        &\le& \big[b_n^{-2} L_{\hat K}|t-t'| + b_n^{-1} |\hat K|_{\infty} \{|\eta - \eta'|_1 + |\eta_2| \cdot |t-t'|b_n^{-1}\big]\\
        &&\quad\quad\quad\quad\quad\quad\times\frac{1}{n}\sum_{i=1}^{n}\{2 + |\tilde Z_i(i/n)|_{\hat \chi}^{M(1+s)} + \||\tilde Z_i(i/n)|_{\hat \chi}^{M(1+s)}\|_1\}
    \end{eqnarray*}
    Since $E_n$ is compact, we have $\sup_{\eta \in E_n}|\eta_2|_1 < \infty$. Together with $\||\tilde Z_i(i/n)|_{\hat \chi}^{M(1+s)}\|_1 \le (D|\hat \chi|_1)^{M(1+s)}$, we obtain the result.
    
    (ii) We now have
    \begin{eqnarray*}
        &&|\hat G_n(t,\eta) - \hat G_n(t,\eta')|\\
        &\le& (nb_n)^{-1}\sum_{i=1}^{n}|\hat K_{b_n}(i/n-t)|\cdot \big|\tilde g_{\theta(i/n)}(\zeta_i,\tilde X_i(i/n),\eta_1 + \eta_2 (i/n-t)b_n^{-1}) \\
        &&\quad\quad\quad\quad\quad- \tilde g_{\theta(i/n)}(\zeta_i,\tilde X_i(i/n), \eta_1' + \eta_2'(i/n-t)b_n^{-1})\big|.
    \end{eqnarray*}
    Here, $|\eta - \eta_{b_n}(t)| < \iota/2$ implies $|(\eta_1 + \eta_2(i/n-t)b_n^{-1}) - \theta(t)| < \iota$ for $n$ large enough. Since $\theta(\cdot)$ is uniformly continuous, $|\theta - \theta(t)|_1 < \iota$, $|i/n-t| \le b_n$ imply $|\theta - \theta(i/n)|_1 < \iota$ for $n$ large enough. Since $\tilde g \in \sH_{\iota}^{mult}(M,\chi^{(s)},\bar C^{(s)})$, we obtain
    \begin{eqnarray*}
        &&|\hat G_n(t,\eta) - \hat G_n(t,\eta')|\\
        &\le& \bar C^{(s)} |\eta - \eta'|_1 (nb_n)^{-1}\sum_{i=1}^{n}|\hat K_{b_n}(i/n-t)| \cdot \{(1 + |\tilde X_i(i/n)|_{\chi}^{M})^{1+s}(1+|\zeta_i|^{M})^{1+s}\\
        &&\quad\quad\quad\quad\quad\quad\quad\quad\quad\quad\quad\quad\quad\quad\quad\quad + \|(1 + |\tilde X_i(i/n)|_{\chi}^{M})^{1+s}(1+|\zeta_i|^{M})^{1+s}\|_1\},
    \end{eqnarray*}
    giving the result.\\
    We have
    \begin{eqnarray*}
        && |\hat G_n(t,\eta) - \hat G_n(t',\eta')|\\
        &\le& (nb_n)^{-1}\sum_{i=1}^{n}|\hat K_{b_n}(i/n-t) - \hat K_{b_n}(i/n-t')|\\
        &&\quad\quad\times \sup_{|\eta - \eta_{b_n}(t)| < \iota/2}\{|\tilde g_{\theta(i/n)}(\zeta_i,\tilde X_i(i/n),\eta_1 + \eta_2(i/n-t)b_n^{-1})|\\
        &&\quad\quad\quad\quad\quad\quad\quad\quad\quad + \| \tilde g_{\theta(i/n)}(\zeta_i,\tilde X_i(i/n),\eta_1 + \eta_2 (i/n-t) b_n^{-1})\|_1\}\\
        &&+  (nb_n)^{-1}\sum_{i=1}^{n}|\hat K_{b_n}(i/n-t')|\cdot |\tilde g_{\theta(i/n)}(\zeta_i,\tilde X_i(i/n),\eta_1 + \eta_2(i/n-t)b_n^{-1}) \\
        &&\quad\quad\quad\quad\quad\quad\quad\quad\quad\quad\quad\quad\quad\quad - \tilde g_{\theta(i/n)}(\zeta_i,\tilde X_i(i/n), \eta_1' + \eta_2' (i/n-t')b_n^{-1})|.
    \end{eqnarray*}
    The same argumentation as before allows us to use the Lipschitz properties of $\tilde g_{\theta(i/n)}$ w.r.t. $\theta$, giving the result.
\end{proof}

\begin{lemma}\label{lemma:empprocess2} Let $\gamma > 1$. For $s \ge 0$, let $\chi_i^{(s)} = (\chi_i^{(s)})_{i\in\IN}$ be a sequence with $\chi_i^{(s)} = O(i^{-(1+\gamma)})$. Recall the notation from \reff{eq:proof_notation}. Suppose that either (in the assertion (a) below) Assumption \ref{ass1}\ref{ass1_stat}, \ref{ass1_dep} or (in the assertions (b),(c) below) Assumption \ref{ass3}\ref{ass3_stat}, \ref{ass3_dep} hold with some $r$ specified below. 
    \begin{enumerate}
        \item[(i)] Let $r \ge 1+\varsigma$, $\varsigma \ge 0$ and assume either that $\varsigma = 0$ and $g \in \sH(M,\chi^{(0)},\bar C^{(0)})$ or $\varsigma > 0$ and for all $s > 0$ small enough, $g \in \sH(M(1+s),\chi^{(s)},\bar C^{(s)})$. Then $$\|\sup_{t\in (0,1)}\sup_{\eta \in E_n}|\hat G_n(t,\eta) - G_n^c(t,\eta)|\|_1 = O((nb_n)^{-1}).$$
        \item[(ii)] Fix $t \in [0,1]$ and assume that $n b_n \to \infty$. Let $r \ge 1 + \varsigma$, $\varsigma > 0$.\\
        (a) If for all $s > 0$ small enough, $g \in \sH(M(1+s),\chi^{(s)},\bar C^{(s)})$, then
        $$\sup_{\eta \in E_n}|\hat G_n(t,\eta)| = o_{\IP}(1).$$
        (b) If for all $s > 0$ small enough, $\tilde g_{\tilde \theta}(y,x,\theta) := g(F(y,x,\tilde \theta),x,\theta)$ fulfills $\tilde g \in \sH^{mult}_{\iota}(M,\chi^{(s)},\bar C^{(s)})$, then
        $$\sup_{|\eta - \eta_{b_n}(t)| < \iota}|\hat G_n(t,\eta)| = o_{\IP}(1) \quad\quad \text{if} \quad b_n \to 0.$$
        (c) If for all $s > 0$ small enough, $g = \ell$ fulfills \reff{eq:additionalgarchlipschitz} and $g \in \sH(2M(1+s),\chi^{(s)},\bar C^{(s)})$, then
        $$\sup_{\eta \in E_n}|\hat G_n(t,\eta)| = o_{\IP}(1).$$
        \item[(iii)] Let $r \ge 2 + \varsigma$, $\varsigma > 0$. Define $\beta_n = \log(n)^{1/2} (nb_n)^{-1/2} b_n^{-1/2}$.\\
        (a) If for all $s>0$ small enough, $g \in \sH(M(1+s),\chi^{(s)},\bar C^{(s)})$, then $$\sup_{t\in (0,1)}\sup_{\eta \in E_n}|\hat G_n(t,\eta)| = O_{\IP}(\beta_n).$$ (b) If $g$ is such that $\tilde g_{\tilde \theta}(y,x,\theta) := g(F(y,x,\tilde \theta),x,\theta)$ fulfills $\tilde g \in \sH^{mult}_{\iota}(M,\chi^{(s)},\bar C^{(s)})$ for $s>0$ small enough, then
        \[
            \sup_{t\in (0,1)}\sup_{|\eta - \eta_{b_n}(t)|_1 < \iota}|\hat G(t,\eta)| = O_{\IP}(\beta_n).
        \]
        (c) If for all $s>0$ small enough, $g$ fulfills \reff{eq:additionalgarchlipschitz} and $g \in \sH(2M(1+s),\chi^{(s)},\bar C^{(s)})$, then
        \[
            \sup_{t\in (0,1)}\sup_{\eta \in E_n}|\hat G(t,\eta)| = O_{\IP}(\beta_n).
        \]
    \end{enumerate}
\end{lemma}
\begin{proof}[Proof of Lemma \ref{lemma:empprocess2}] 
We abbreviate $\chi = \chi^{(s)}$ and $\bar C = \bar C^{(s)}$.\\
(i) By Lemma \ref{lemma:hoelder}(i),(ii), we obtain that for some $C > 0$:
\begin{eqnarray*}
    \|\sup_{\theta \in \Theta}|g(Z_i, \theta) - g(Z_i^c,\theta)| \|_1 &\le& C \sum_{j=0}^{\infty}\hat \chi_j \|Z_{ij} - Z_{ij}^c\|_{M} \le 2C\sum_{j=i}^{\infty}\chi_j \|Z_{ij}\|_{M} \le 2CD \sum_{j=i}^{\infty}\chi_j.
\end{eqnarray*}
Similarly, we have for some $C > 0$ that
\begin{eqnarray*}
    \|\sup_{\theta \in \Theta}|g(Z_i, \theta) - g(\tilde Z_i(i/n),\theta)| \|_1 &\le& C \sum_{j=0}^{\infty} \hat\chi_j \|Y_{ij} - \tilde Y_{ij}(i/n)\|_{M}\Big)\\
    &\le& C C_A |\chi|_1 n^{-1}.
\end{eqnarray*}
Thus
\begin{eqnarray*}
    && \|\sup_{t\in (0,1)}\sup_{\eta \in E_n}|\hat G_n(t,\eta) - G_n^c(t,\eta)|\|_1\\
    &\le& |K|_{\infty}(nb_n)^{-1}\sum_{i=1}^{n}\|\sup_{\theta \in \Theta}|g(\tilde Z_i(i/n), \theta) - g(Z_i^c,\theta)| \|_1\\
    &\le& 2CD|K|_{\infty} (nb_n)^{-1}\sum_{i=1}^{n}\sum_{j=i}^{\infty}\chi_j + |K|_{\infty}CC_A |\chi|_1 (nb_n)^{-1} = O((nb_n)^{-1}).
\end{eqnarray*}
The last step is due to $\chi_j = O(j^{-(1+\gamma)})$, since this implies   $\sum_{i=1}^{n}\sum_{j=i}^{\infty}\chi_j = O(1)$. The proofs under Assumption \ref{ass3} are similar in view of Lemma \ref{lemma:hoeldergarchlikelihood}.

    (ii) (a) Fix $Q > 0$. Let $\kappa > 0$. Let $E_{n}^{(\kappa)}$ be a discretization of $E_n$ such that for each $\eta \in E_{n}$ one can find $\eta' \in E_{n}^{(\kappa)}$ with $|\eta - \eta'|_1 \le \kappa$. Note that $\# E_{n}^{(\kappa)}$ does not need to depend on $n$. Then
    \begin{eqnarray}
        \IP\big(\sup_{\eta \in E_n}|\hat G_n(t,\eta)| > Q\big) &\le& \# E_{n}^{(\kappa)}\sup_{\eta \in E_n}\IP\big(|\hat G_n(t,\eta)| > Q/2\big)\nonumber\\
        &&\quad\quad\quad + \IP(\sup_{|\eta - \eta'|_1 \le \kappa}|\hat G_n(t,\eta) - \hat G_n(t,\eta')| > Q/2).\label{eq:empprocess_markov}
    \end{eqnarray}
    By Markov's inequality, we have for $0 \le s \le \varsigma$,
    \[
        \IP\big(|\hat G_n(t,\eta)| > Q/2\big) \le \frac{\|\hat G_n(t,\eta)\|_{1+s}^{1+s}}{(Q/2)^{1+s}}.
    \]
    Using Burkholder's moment inequality (cf. \cite{burk88}) and Lemma \ref{lemma:depmeasure}(i) applied for $q =1+s$, $s > 0$ small enough, the computation
    \begin{eqnarray}
        &&\|\hat G_n(t,\eta)\|_{1+s} \\
        &\le& (nb_n)^{-1}\sum_{l=0}^{\infty}\Big\|\sum_{i=1}^{n}\hat K_{b_n}(i/n-t) P_{i-l}g(\tilde Z_i(i/n),\eta_1 + \eta_2(i/n-t) b_n^{-1})\Big\|_{1+s}\nonumber\\
        &\le& s^{-1}(nb_n)^{-1}\sum_{l=0}^{\infty}\Big(\Big\|\sum_{i=1}^{n}\hat K_{b_n}(i/n-t)^2 P_{i-l}^2g(\tilde Z_i(i/n),\eta_1 + \eta_2(i/n-t)b_n^{-1})\Big\|_{(1+s)/2}^{(1+s)/2}\Big)^{1/(1+s)}\nonumber\\
        &\le& s^{-1}(nb_n)^{-s/(1+s)}|\hat K|_{\infty}\sum_{l=0}^{\infty}\sup_{t\in[0,1]}\delta_{1+s}^{\sup_{\theta \in \Theta}|g(\tilde Z(t),\theta)|}(l) = O((nb_n)^{-s/(1+s)}), \nonumber \label{eq:empprocess_moment}
    \end{eqnarray}
     shows that the first summand in \reff{eq:empprocess_markov} tends to zero. For the second summand, Lemma \ref{lemma:property_lipschitz}(i) implies
    \[
        \IP(\sup_{|\eta - \eta'|_1 \le \kappa}|\hat G_n(t,\eta) - \hat G_n(t,\eta')| > Q/2) \le \frac{2\tilde C\kappa}{Q},
    \]
    which can be made arbitrary small by choosing $\kappa$ small enough. So we have shown that \reff{eq:empprocess_markov} tends to zero for $n\to\infty$.\\
    (b) The proof is similar to (a) by using \ref{lemma:property_lipschitz}(ii) and Lemma \ref{lemma:depmeasuregarch}(i) instead of Lemma \ref{lemma:property_lipschitz}(i) and Lemma \ref{lemma:depmeasure}(i).\\
    (c) The proof is similar to (a) by using Lemma \ref{lemma:depmeasure}(i)(*) instead of Lemma \ref{lemma:depmeasure}(i).

    (iii) (a) We use a chaining argument. Let $r = n^{3}$ and let $E_{n,r}$ be a discretization of $E_n$ such that for each $\eta \in E_n$ one can find $\eta' \in E_{n,r}$ with $|\eta - \eta'| \le r^{-1}$. Define $\sT_{n,r} := \{i/r: i = 1,\ldots,r\}$ as a discretization of $(0,1)$. Then $\# (E_{n,r} \times \sT_{n,r}) = O(r^{2d_{\Theta}+1})$.
   For some constant $Q > 0$, we have
    \begin{eqnarray}
        && \IP\Big(\sup_{\eta \in E_n, t \in (0,1)}|\hat G_{n}(t,\eta)| > Q \beta_n\Big)\nonumber\\
        &\le& \IP\Big(\sup_{\eta \in E_{n,r}, t \in \sT_{n,r}}|\hat G_{n}(t,\eta)| > Q \beta_n/2\Big)\nonumber\\
        &&\quad\quad + \IP\Big(\sup_{|\eta-\eta'| \le r^{-1}, |t-t'| \le r^{-1}}|\hat G_{n}(t,\eta) - \hat G_{n}(t',\eta')| > Q \beta_n/2\Big).\label{eq:proof_strategy_emp}
    \end{eqnarray}
    Let $\alpha = 1/2$. Let $M_i(t,\eta,u) := \hat K_{b_n}(u-t) g(\tilde Z_i(u),\eta_1 + \eta_2 (u-t)b_n^{-1})$. By Lemma \ref{lemma:depmeasure}(ii) applied with $q = 2 + s$, $s > 0$ small enough, we have $\sup_{u}\Delta^{\sup_{t,\eta}|M(t,\eta,u)|}_{2+s}(k) = O(k^{-(1+\gamma)})$. Thus
    \begin{eqnarray*}
        W_{2+s,\alpha} &:=&  \sup_{u\in[0,1]}\| \sup_{t,\eta}|M_i(t,\eta,u)| \|_{2+s,\alpha} = \sup_{m \ge 0} (m+1)^{\alpha}\sup_{u\in[0,1]}\sup_{t,\eta}\Delta_{2+s}^{\sup_{t,\eta}|M(t,\eta,u)|}(m) < \infty.
    \end{eqnarray*}
    (independent of $n$) and
    \begin{eqnarray*}
        W_{2,\alpha} &:=& \sup_{u\in[0,1]}\sup_{t,\eta}\|M_i(t,\eta,u)\|_{2,\alpha} = \sup_{m\ge 0}(m+1)^{\alpha} \sup_{u\in[0,1]}\sup_{t,\eta}\Delta_{2}^{M(t,\eta,u)}(m) < \infty
    \end{eqnarray*}
    (independent of $n$).
    Note that $l = 1 \wedge \log \#(E_{n,r} \times \sT_{n,r}) \le 3(2d_{\Theta}+1) \log(n)$ and $Q \beta_n (n b_n) = Q n^{1/2} \log(n)^{1/2} \geq \sqrt{nl}W_{2,\alpha}+n^{1/(2+s)} l^{3/2}W_{2+s,\alpha} \gtrsim n^{1/2} \log(n)^{1/2} + n^{1/(2+s)} \log(n)^{3/2}$ for $Q$ large enough.
    By applying Theorem 6.2 of \cite{zhangwu2017} (the proof therein also works for the uniform functional dependence measure) with $q = 2+s$ and $\alpha = 1/2$ to $(M_i(t,\eta,i/n))_{t\in \sT_{n,r}, \eta \in E_{n,r}}$, we have with some constant $C_{\alpha} > 0$:
     \begin{eqnarray}
        &&\IP\big(\sup_{\eta' \in E_{n,r}, t' \in \sT_{n,r}} |\hat G_{n}(t',\eta')| \ge Q\beta_n/2\big)\nonumber\\
        &\le& \frac{C_{\alpha}n\cdot l^{1+s/2}W_{2+s,\alpha}^{2+s}}{(Q/2)^{2+s}(\delta_n (nb_n))^{2+s}}+C_{\alpha} \exp \Big( -\frac{C_{\alpha}(Q/2)^2 (\beta_n (n b_n))^2}{n W_{2,\alpha}^2}\Big)\nonumber\\
        &\lesssim& n^{-s/2} + \exp\big(-\frac{(n b_n) b_n^{-1} \log(n)}{n} \big)\nonumber\\
        &\to& 0.\label{eq:proof_strategy_emp2}
    \end{eqnarray}
    By Markov's inequality and Lemma \ref{lemma:property_lipschitz}(i),
    \begin{equation}
        \IP\big(\sup_{|\eta - \eta'|_1 \le r^{-1}, |t - t'| \le r^{-1}} |\hat G_n(t,\eta) - \hat G_n(t',\eta')| \ge C\beta_n/2\big) = O\Big(\frac{b_n^{-2}r^{-1}}{\beta_n}\Big).\label{eq:proof_strategy_emp3}
    \end{equation}
    We have $b_n^{-2}r^{-1}\beta_n^{-1} = b_n^{-2} n^{-3} (n b_n)^{1/2} b_n^{1/2} \log(n)^{-1/2} \to 0$. Inserting \reff{eq:proof_strategy_emp2} and \reff{eq:proof_strategy_emp3} into \reff{eq:proof_strategy_emp}, we obtain the result.\\
    (b) The proof is similar to (a) by using \ref{lemma:property_lipschitz}(ii) and Lemma \ref{lemma:depmeasuregarch}(ii) instead of Lemma \ref{lemma:property_lipschitz}(i) and Lemma \ref{lemma:depmeasure}(ii).\\
    (c) The proof is similar to (a) by using \ref{lemma:property_lipschitz}(ii)(*) instead of Lemma \ref{lemma:property_lipschitz}(ii).
\end{proof}

\begin{lemma}\label{lemma:bias} Let $g:\IR^{\IN} \times \Theta \to \IR$, and define
\[
    \hat B_{n}(t,\eta) = (nb_n)^{-1}\sum_{i=1}^{n}\hat K_{b_n}(i/n-t)g(\tilde Z_i(i/n),\eta_1 + \eta_2(i/n-t)b_n^{-1}).
\]
\begin{enumerate}
    \item[(a)] If Assumption \ref{ass1}\ref{ass1_stat} is fulfilled with $r \ge 1+s$, $s \ge 0$ and $g \in \sH(M(1+s),\chi,\bar C)$, then
\[
    \sup_{t \in (0,1)}\sup_{\eta \in E_n}|\IE \hat B_n(t,\eta) - \int_{-t/b_n}^{(1-t)/b_n} \hat K(x)\IE g(\tilde Z_0(t),\eta_1 +
     \eta_2 x)dx| = O((nb_n)^{-1} + b_n).
\]
    \item[(b)] If Assumption \ref{ass3}\ref{ass3_stat} is fulfilled with $r \ge 1+s$ and $g$ is such that $\tilde g_{\tilde \theta}(y,x,\theta) := g(F(y,x,\tilde \theta),x,\theta)$ fulfills $\tilde g \in \sH^{mult}_{\iota}(M,\chi,\bar C)$, then
    \[
    \sup_{t \in (0,1)}\sup_{|\eta - \eta_{b_n}(t)| < \iota}|\IE \hat B_n(t,\eta) - \int_{-t/b_n}^{(1-t)/b_n} \hat K(x)\IE g(\tilde Z_0(t),\eta_1 +
     \eta_2 x)dx| = O((nb_n)^{-1} + b_n).
\]
\end{enumerate}
If the supremum is taken over $t \in \sT_n$ instead of $t \in (0,1)$, then $\int_{-t/b_n}^{(1-t)/b_n}$ can be replaced by $\int_{-1}^{1}$.
\end{lemma}
\begin{proof}[Proof of Lemma \ref{lemma:bias}]
    (a) Let $\tilde B_n(t,\eta) := (n b_n)^{-1}\sum_{i=1}^{n}\hat K_{b_n}(i/n-t) g(\tilde Z_i(t), \eta_1 + \eta_2 (i/n-t)b_n^{-1})$. By Lemma \ref{lemma:hoelder}(i), we have with some constant $\tilde C > 0$ that 
    \begin{eqnarray*}
        && \|g(\tilde Z_0(i/n),\eta_1 + \eta_2 (i/n-t)b_n^{-1}) - g(\tilde Z_0(t), \eta_1 + \eta_2 (i/n-t)b_n^{-1})\|_1\\
        &\le& \tilde C\sum_{i=0}^{\infty}\hat\chi_i \|\tilde Y_{-i}(i/n) - \tilde Y_{-i}(t)\|_M \le \tilde C C_B |\hat \chi|_1 b_n.
    \end{eqnarray*}
    Thus
    \begin{eqnarray*}
       && \|\hat B_n(t,\eta) - \tilde B_n(t,\eta)\|_1\\
        &\le& (n b_n)^{-1}\sum_{i=1}^{n}|\hat K_{b_n}(i/n-t)|\\
        &&\quad\quad\quad\quad\quad\quad\times\|g(\tilde Z_i(i/n),\eta_1 + \eta_2 (i/n-t)b_n^{-1}) - g(\tilde Z_i(t), \eta_1 + \eta_2 (i/n-t)b_n^{-1})\|_1\\
        &\le& \tilde C |\hat K|_{\infty} C_B(1+|\chi|_1) b_n.
    \end{eqnarray*}
    Since $\hat K$ is of bounded variation and $\theta \mapsto \IE g(\tilde Z_0(t),\theta)$ is Lipschitz continuous due to $g \in \sH(M,\chi,\bar C)$ and Lemma \ref{lemma:hoelder}, a Riemannian sum argument yields
    \begin{eqnarray*}
        \tilde B_n(t,\eta) &=& (nb_n)^{-1}\sum_{i=1}^{n}\hat K_{b_n}(i/n-t) \IE g(\tilde Z_0(t),\eta_1 + \eta_2(i/n-t)b_n^{-1})\\
        &=& \int_{-t/b_n}^{(1-t)/b_n} \hat K(x) \IE g(\tilde Z_0(t),\eta_1 + \eta_2 x) dx + O( (nb_n)^{-1}),
    \end{eqnarray*}
    uniformly in $t \in (0,1)$, $\eta \in E_n$.\\
    (b) The proof is the same by using Lemma \ref{lemma:hoeldergarch} with $q = 1$ instead of Lemma \ref{lemma:hoelder}.
\end{proof}

\subsection{Basic Lipschitz, bias and dependence results}

Lemmas \ref{lemma:hoelder}, \ref{lemma:hoeldergarchlikelihood} and  \ref{lemma:hoeldergarch} state how the deviation of $g(Z_i,\theta) - g(Z_i',\theta)$ can be controlled by $Z_i - Z_i'$. In the GARCH case, this needs two results due to different treatments for the first and second derivative of the likelihood.


In Lemmas \ref{lemma:depmeasure} and \ref{lemma:depmeasuregarch}, the dependence measure of $g(\tilde Z_i(t),\theta)$ is calculated based on the H\"older-type results in Lemmas \ref{lemma:hoelder}, \ref{lemma:hoeldergarchlikelihood} and  \ref{lemma:hoeldergarch} for both Cases 1 and 2.

\begin{lemma}\label{lemma:hoelder} Let $q > 0$. Let $g \in \sH(M,\chi,\bar C)$.  Let $\hat Z = (\hat Z_j)_{j\in\IN_0}$, $\hat Z' = (\hat Z_j')_{j\in\IN_0}$ be sequences of random variables. Assume that there exists some $D > 0$ such that for all $j\in\IN_0$,
\begin{equation}
    \|\hat Z_j\|_{qM} \le D, \quad\quad \|\hat Z_j'\|_{qM} \le D.\label{eq:hoelder0}
\end{equation}
Then there exists some constant $C > 0$ only dependent on $M$, $D$, $\chi$ and $\tilde D$ (only in (ii)) such that
    \begin{eqnarray}
        \| \sup_{\theta \in \Theta}|g(\hat Z, \theta) - g(\hat Z',\theta)|\|_q &\le& \bar C\cdot C \sum_{j=0}^{\infty}\hat \chi_j\|\hat Z_j - \hat Z_j'\|_{qM},\label{eq:hoelder1}\\
        \big\|\sup_{\theta\not=\theta'}\frac{|g(\hat Z, \theta) - g(\hat Z, \theta')|}{|\theta-\theta'|_1}\big\|_q &\le& \bar C \cdot C,\label{eq:hoelder2}\\
         \|\sup_{\theta \in \Theta}|g(\hat Z, \theta)|\|_q &\le& \bar C \cdot C,  \label{eq:hoelder3}
    \end{eqnarray}
\end{lemma}
\begin{proof}[Proof of Lemma \ref{lemma:hoelder}] Note that
\[
    \| |\hat Z|_{\hat \chi} \|_{qM} \le \sum_{j=1}\hat \chi_j \|\hat Z_j\|_{qM} \le D |\hat \chi|_1.
\]
We have by H\"older's inequality that
    \begin{eqnarray}
        && \|\sup_{\theta \in \Theta}|g(\hat Z,\theta)-g(\hat Z',\theta)|\|_q\nonumber\\
        &\le& \bar C\big\| |\hat Z - \hat Z'|_{\hat \chi} (1+|\hat Z|_{\hat \chi}^{M-1} + |\hat Z'|_{\hat \chi}^{M-1})\big\|_q\nonumber\\
        &\le& \bar C\big\| |\hat Z - \hat Z'|_{\hat \chi} \big\|_{qM} \big(1 + \big\||\hat Z|_{\hat \chi}\big\|_{qM}^{M-1} + \big\||\hat Z'|_{\hat \chi}\big\|_{qM}^{M-1}\big)\nonumber\\
        &\le& \bar C(1+2(D |\hat \chi|_1)^{M-1})\cdot \sum_{j=1}^{\infty}\hat \chi_j \|\hat Z_j - \hat Z_j'\|_{qM},\nonumber
    \end{eqnarray}
    which shows \reff{eq:hoelder1}. The proof of \reff{eq:hoelder3} is obvious from \reff{eq:hoelder1} and $\sup_{\theta\in\Theta}|g(0,\theta)| \le \bar C$. Finally,
    \[
        \Big\|\sup_{\theta \not= \theta'}\frac{|g(\hat Z,\theta)-g(\hat Z,\theta')|}{|\theta - \theta'|_1}\|_q \le  \bar C\big\|1 + |\hat Z|_{\hat \chi}^{M}\big\|_q \le C(1+ D|\hat \chi|_1).
   \]
\end{proof}

The following lemma states the same results as Lemma \ref{lemma:hoelder} under a different continuity condition on $g$ as it is given in the tvGARCH case.

\begin{lemma}[for tvGARCH]\label{lemma:hoeldergarchlikelihood} Let $q > 0$ and $s > 0$. Let $\hat Z'$ be as in Lemma \ref{lemma:hoelder} satisfying \reff{eq:hoelder0} with $M$ replaced by $M(1+s)$. Let $g = \ell$ satisfy \reff{eq:additionalgarchlipschitz} and $g \in \sH(M(1+s),\chi^{(s)},\bar C^{(s)})$. Then there exists some constant $C^{(s)} > 0$ only dependent on $M$, $D$, $\chi^{(s)}$ such that
    \begin{eqnarray}
        &&\| \sup_{\theta \in \Theta}|g(\hat Z, \theta) - g(\hat Z',\theta)|\|_q\nonumber\\
        &\le& \bar C^{(s)}\cdot C^{(s)} \sum_{j=0}^{\infty}\hat \chi_j^{(s)}\big(\|\hat Z_j - \hat Z_j'\|_{qM(1+s)} + \|\hat Z_j - \hat Z_j'\|_{qM(1+s)}^{s}\big),\label{eq:hoelder1garchlikelihood}
    \end{eqnarray}
    and
    \begin{eqnarray}
         \|\sup_{\theta \in \Theta}|g(\hat Z, \theta)|\|_q &\le& \bar C^{(s)} \cdot C^{(s)},  \label{eq:hoelder3garchlikelihood}
    \end{eqnarray}
    where $\hat \chi^{(s)} = (1,\chi^{(s)})$.
\end{lemma}
\begin{proof}[Proof of Lemma \ref{lemma:hoeldergarchlikelihood}]
    By H\"older's inequality,
    \begin{eqnarray*}
        && \big\|\sup_{\theta \in \Theta}|g(\hat Z, \theta) - g(\hat Z',\theta)|\big\|_q\\
        &\le& \bar C^{(s)}\big\| |\hat Z - \hat Z'|_{\chi^{(s)},s}\cdot (1+|\hat Z|_{\hat \chi}^M + |\hat Z'|_{\hat \chi}^M)\big\|_q\\
        &&\quad\quad + \bar C^{(s)}\big\| |\hat Z - \hat Z'|_{\chi^{(s)},1} \cdot (1+|\hat Z|_{\hat \chi}^{M-1} + |\hat Z'|_{\hat \chi}^{M-1})^{1+s}\big\|_q\\
        &\le&  \bar C^{(s)}\sum_{j=0}^{\infty}\hat \chi_j^{(s)}\| \hat Z_j - \hat Z_j' \|_{q(M+s)}^s \cdot \big(1 + \| |\hat Z|_{\hat \chi}\|_{q(M+s)}^{M} + \| |\hat Z'|_{\hat \chi}\|_{q(M+s)}^{M}\big)\\
        &&\quad\quad + \bar C^{(s)}\sum_{j=0}^{\infty}\hat \chi_j^{(s)}\|\hat Z_j - \hat Z_j'\|_{qM(1+s)}\big(1 + \| |\hat Z|_{\hat \chi}\|_{q(M+s)}^{M-1} + \| |\hat Z'|_{\hat \chi}\|_{q(M+s)}^{M-1}\big)\\
        &\le& \bar C^{(s)} \big(1+(2|\hat \chi|_1)^M + (2|\hat \chi|_1)^{M-1}\big)\cdot \sum_{j=0}^{\infty}\hat \chi_j^{(s)}\big(\|\hat Z_j - \hat Z_j'\|_{qM(1+s)}^s + \|\hat Z_j - \hat Z_j'\|_{qM(1+s)}\big).
    \end{eqnarray*}
    This shows \reff{eq:hoelder1garchlikelihood}. The second result \reff{eq:hoelder3garchlikelihood} follows from \reff{eq:hoelder1garchlikelihood} with $\hat Z' = 0$ and $\sup_{\theta \in \Theta}|g(0,\theta)| < \bar C$ by assumption.
\end{proof}

\begin{lemma}[for tvGARCH]\label{lemma:hoeldergarch} Let $q > 0, \iota > 0$. Let $\tilde g \in \sH_{\iota}^{mult}(M,\chi,\bar C)$. Let $\hat X = (\hat X_j)_{j\in\IN}$, $\hat X' = (\hat X_j')_{j\in\IN}$ be sequences of random variables. Assume that there exists some $D > 0$ such that for all $j\in\IN$,
\begin{equation}
    \|\hat X_j\|_{qM} \le D, \quad\quad \|\hat X_j'\|_{qM} \le D.\label{eq:hoelder0_new}
\end{equation}
Let $\zeta_0$ be independent of $\hat X$, $\hat X'$ with $\|\zeta_0\|_{qM} \le D$. Then there exists some constant $C > 0$ only dependent on $M$, $D$, $\chi$, $\bar C$  such that
    \begin{eqnarray}
        \big\| \sup_{|\theta - \tilde \theta| < \iota}|\tilde g_{\tilde \theta}(\zeta_0,\hat X, \theta) - \tilde g_{\tilde \theta}(\zeta_0,\hat X',\theta)|\big\|_q &\le& \bar C\cdot C \sum_{j=1}^{\infty}\chi_j\|\hat X_j - \hat X_j'\|_{qM},\label{eq:hoelder1garch}\\
        \big\|\sup_{\theta\not=\theta',|\theta-\tilde \theta|_1 < \iota, |\theta' - \tilde \theta|_1 < \iota}\frac{|\tilde g_{\tilde \theta}(\zeta_0,\hat X, \theta) - \tilde g_{\tilde \theta}(\zeta_0,\hat X, \theta')|}{|\theta - \theta'|_1}\big\|_q &\le& \bar C \cdot C,\label{eq:hoelder2garch}\\
         \big\|\sup_{|\theta - \tilde \theta|_1 < \iota}|\tilde g_{\tilde \theta}(\zeta_0,\hat X,\theta)|\big\|_q &\le& \bar C \cdot C.  \label{eq:hoelder3garch}
    \end{eqnarray}
\end{lemma}
\begin{proof}[Proof of Lemma \ref{lemma:hoeldergarch}]
    With H{\" o}lder's inequality,
    \begin{eqnarray*}
        &&\big\|\sup_{|\theta - \tilde \theta| < \iota}|g_{\tilde \theta}(\zeta_0, \hat X, \theta) - g_{\tilde \theta}(\zeta_0, \hat X', \theta)| \big\|_q\\
        &\le& \bar C\big\| |\hat X - \hat X'|_{\chi} (1 + |\hat X|_{\chi}^{M-1} + |\hat X'|_{\chi}^{M-1}) (1+|\zeta_0|^M)\big\|_q\\
        &\le& \bar C\big\| |\hat X - \hat X'|_{\chi}\big\|_{qM} (1 + \big\||\hat X|_{\chi}\big\|_{qM}^{M-1} + \big\||\hat X'|_{\chi}\big\|_{qM}^{M-1}) (1+\|\zeta_0\|_{qM}^M)\\
        &\le& \bar C(1+2(D|\chi|_1)^{M-1})(1+D^M)\cdot \sum_{j=1}^{\infty}\chi_j \|\hat X_j - \hat X_j'\|_{qM}.
    \end{eqnarray*}
    This shows \reff{eq:hoelder1garch}. The result \reff{eq:hoelder2garch} follows similarly as in Lemma \ref{lemma:hoelder}. Using \reff{eq:hoelder1garch} with $\hat X' = 0$ and
    \[
       \big\|\sup_{|\theta - \tilde \theta| < \iota}|\tilde g_{\tilde \theta}(\zeta_0,0,\theta)|\big\|_q \le \bar C\|1+|\zeta_0|^{M}\|_q \le \bar C (1 + D^M),
    \]
    we obtain \reff{eq:hoelder3garch}.
\end{proof}

\begin{lemma}\label{lemma:depmeasure}
    Let $q \ge 1$. Suppose that Assumption \ref{ass1}\ref{ass1_stat}, \ref{ass1_dep} hold with some $r \ge q$. Let $g \in \sH(M,\chi,\bar C)$, where $\chi_i = O(i^{-(1+\gamma)})$. Then it holds that
    \begin{enumerate}
        \item[(i)] $\sup_{t\in[0,1]}\delta_q^{\sup_{\theta}|g(\tilde Z(t),\theta)|}(j) = O(j^{-(1+\gamma)})$.
        \item[(ii)] For $M_i(t,\eta,u) := \hat K_{b_n}(u-t) g(\tilde Z_i(u),\eta_1 + \eta_2 (u-t) b_n^{-1})$, we have
        \[
            \sup_{u\in[0,1]}\sup_{t,\eta}\delta_q^{M(t,\eta,u)}(j) = O(j^{-(1+\gamma)}),\quad\quad  \sup_{u\in[0,1]}\delta^{\sup_{t,\eta}|M(t,\eta,u)|}_q(j) =  O(j^{-(1+\gamma)}).
        \]
        \item[(iii)] Let $d_{u}(t) = \theta(u) - \theta(t)-(u-t)\theta'(t)$ and $M_i^{(2)}(t,u) := \hat K_{b_n}(u-t) \{\int_0^{1}g(\tilde Z_i(u),\theta(t) + s d_{u}(t)) ds \} \cdot d_{u}(t)$. Then it holds for each component that
        \[
            \sup_{u\in[0,1]}\delta_q^{M^{(2)}(t,u)}(j) = O(b_n^2 j^{-(1+\gamma)}), \quad\quad \sup_{u\in[0,1]}\delta^{\sup_{t}|M^{(2)}(t,u)|}_q(j) = O(b_n^2 j^{-(1+\gamma)}).
        \]
    \end{enumerate}
    (*) If instead Assumption \ref{ass3}\ref{ass3_stat}, \ref{ass3_dep} hold with some $r > q$ and $g = \ell$ fulfills \reff{eq:additionalgarchlipschitz} for all $s  > 0$ small enough, then the statements above remain valid.
\end{lemma}
\begin{proof}
    (i) Let $\tilde Z_j(t)^{*}$ be a coupled version of $\tilde Z_j(t)$ where $\zeta_0$ is replaced by $\zeta_0^{*}$. By Lemma \ref{lemma:hoelder} we obtain that with some constant $\tilde C > 0$:
    \begin{eqnarray}
        && \delta^{\sup_{\theta}|g(\tilde Z(t),\theta)|}_q(j)\nonumber\\
        &=& \|\sup_{\theta}|g(\tilde Z_j(t),\theta)| - \sup_{\theta}|g(\tilde Z_j(t)^{*},\theta)|\|_q\nonumber\\
        &\le& \|\sup_{\theta}|g(\tilde Z_j(t),\theta) - g(\tilde Z_j(t)^{*},\theta)|\|_q\nonumber\\
        &\le& \tilde C \sum_{i=0}^{\infty}\hat \chi_i \|\tilde Z_{j-i}(t) - \tilde Z_{j-i}(t)^{*}\|_{qM} \le \tilde C \sum_{i=0}^{j}\hat\chi_i \delta_{qM}^{\tilde Y(t)}(j-i).\label{eq:lemma_dep1}
    \end{eqnarray}
    In case (*), let $s > 0$ be such that $q(1+s) < r$.  Then we have by Lemma \ref{lemma:hoeldergarchlikelihood}, there exists some $\tilde C > 0$ such that
    \begin{eqnarray}
        \delta_q^{\sup_{\theta\in \Theta}|g(\tilde Z(t),\theta)|}(j) &\le& \tilde C \sum_{i=0}^{\infty} \hat \chi_i^{(s)}\big(\|\tilde Z_{j-i}(t) - \tilde Z_{j-i}(t)^{*}\|_{qM(1+s)}  + \|\tilde Z_{j-i}(t) - \tilde Z_{j-i}(t)^{*}\|_{qM(1+s)}^s\big)\nonumber\\
        &\le& \tilde C \sum_{i=0}^{j}\chi_i^{(s)}\big(\delta_{qM(1+s)}^{\tilde Y(t)}(j-i) + [\delta_{qM(1+s)}^{\tilde Y(t)}(j-i)]^s\big). \label{eq:lemma_dep3}
    \end{eqnarray}
    Note that if two sequences $a_i,b_i$ with $a_i = b_i = 0$ for $i < 0$ obey $a_i,b_i = O(i^{-(1+\gamma)})$ then the convolution $c_j = \sum_{i=1}^{\infty}a_i b_{j-i+1}$ still obeys $c_j = O(j^{-(1+\gamma)})$ due to
    \begin{eqnarray*}
        |c_j| &\le& \sum_{i=1, i \ge (j+1)/2}^{j+1}|a_i|\cdot |b_{j-i+1}| + \sum_{i=1, |j-i| \ge (j+1)/2}^{j+1}|a_i| |b_{j-i+1}|\\
        &\le& \big(\frac{j+1}{2}\big)^{-(1+\gamma)}\sum_{i=1}^{j+1}|b_{j-i+1}| + \big(\frac{j+1}{2}\big)^{-(1+\gamma)}\sum_{i=1}^{j+1}|a_i| = O(j^{-(1+\gamma)}).
    \end{eqnarray*}
    Together with Assumption \ref{ass1_dep} and \reff{eq:lemma_dep1}  \underline{or} (in case (*)) Assumption \ref{ass3}\ref{ass3_dep} and \reff{eq:lemma_dep3}, this shows $\sup_{t\in[0,1]}\delta_r^{g(\tilde Z(t),\theta)}(j) = O(j^{-(1+\gamma)})$.\\
    
    The proof for (ii),(iii) is the same since
    \begin{eqnarray*}
        \big|\sup_{t,\eta}|M_i(t,\eta,u)| - \sup_{t,\eta}|M_{i}(t,\eta,u)^{*}|\big| &\le& \sup_{t,\eta}|M_{i}(t,\eta,u) - M_{i}(t,\eta,u)^{*}|\\
        &\le& |\hat K|_{\infty}\sup_{\theta}|g(\tilde Z_i(u),\theta) - g(\tilde Z_{i}(u)^{*},\theta)|
    \end{eqnarray*}
    and (since $|d_{u}(t)|_{\infty}  \le \sup_{s}|\theta''(s)|_{\infty}\cdot b_n^2$ if $|t-u| \le b_n$), for each $l$,
    \begin{eqnarray*}
        && \big|\sup_{t}|\tilde M_i^{(2)}(t,u)_l| - \sup_{t}|\tilde M_{i}^{(2)}(t,u)_l^{*}|\big| \le \sup_{t}|M_{i}^{(2)}(t,u)_l - M_{i}^{(2)}(t,u)_l^{*}|\\
        &\le& |\hat K|_{\infty}\sup_{s}|\theta''(s)|_{\infty} b_n^2 \\
        &&\quad\quad\times\sup_{t}\int_0^{1}|g(\tilde Z_i(u),\theta(t)+s d_{u}(t)) - g(\tilde Z_i(u)^{*},\theta(t) + s d_{u}(t))| ds\\
        &\le& |\hat K|_{\infty}\sup_{s}|\theta''(s)|_{\infty} b_n^2 \sup_{\theta \in \Theta}|g(\tilde Z_i(u),\theta) - g(\tilde Z_i(u)^{*},\theta)|.
    \end{eqnarray*}
\end{proof}

\begin{lemma}[for tvGARCH]\label{lemma:depmeasuregarch}
    Let $q \ge 1$. Suppose that Assumption \ref{ass3}\ref{ass3_stat}, \ref{ass3_dep} hold with some $r > q$. For $s > 0$, let $\chi^{(s)} = (\chi_i^{(s)})_{i\in\IN}$ be a sequence with  $\chi_i^{(s)} = O(i^{-(1+\gamma)})$. Let $g$ be such that $\tilde g_{\tilde \theta}(y,x,\theta) := g(F(x,\tilde \theta,y),x,\theta)$ fulfills $\tilde g \in \sH_{\iota}^{mult}(M,\chi^{(s)},\bar C^{(s)})$ for all $s > 0$ small enough. Then
    \begin{enumerate}
        \item[(i)] $\sup_{t\in[0,1]}\delta_q^{\sup_{|\theta-\theta(t)|_1 < \iota}|g(\tilde Z(t),\theta)|}(j) = O(j^{-(1+\gamma)})$.
        \item[(ii)] For $n$ large enough,
        \[
            \sup_{u\in[0,1]}\sup_{t,|\eta - \eta_{b_n}(t)|_1 < \iota/2}\delta_q^{M(t,\eta,u)}(j) = O(j^{-(1+\gamma)}), \quad\quad \sup_{u\in[0,1]}\delta^{\sup_{t,|\eta - \eta_{b_n}(t)|_{1} < \iota/2}|M(t,\eta,u)|}_q(j) =  O(j^{-(1+\gamma)}).
        \]
        \item[(iii)] For $n$ large enough,  $\sup_{u\in[0,1]}\delta_q^{M^{(2)}(t,u)}(j) = O(b_n^2 j^{-(1+\gamma)})$, and $\sup_{u\in[0,1]}\delta^{\sup_{t}|M^{(2)}(t,u)|}_q(j) = O(b_n^2 j^{-(1+\gamma)})$.
    \end{enumerate}
\end{lemma}
\begin{proof}[Proof of Lemma \ref{lemma:depmeasuregarch}]
    (i) Let $\tilde Z_j(t)^{*}$ be a coupled version of $\tilde Z_j(t)$ where $\zeta_0$ is replaced by $\zeta_0^{*}$. By Lemma \ref{lemma:hoeldergarch} we obtain that with some constant $\tilde C > 0$:
    \begin{eqnarray*}
        && \delta^{\sup_{|\theta-\theta(t)|_1 < \iota}|g(\tilde Z(t),\theta)|}_q(j)\nonumber\\
        &\le& \|\sup_{|\theta-\theta(t)|_1 < \iota}|\tilde g_{\theta(t)}(\zeta_j,\tilde X_j(t),\theta) - \tilde g_{\theta(t)}(\zeta_j,\tilde X_j(t)^{*},\theta)|\|_q\nonumber\\
        &\le& \tilde C\sum_{i=1}^{\infty}\chi_i \|\tilde X_{j-i+1}(t) - \tilde X_{j-i+1}(t)^{*}\|_{qM}\nonumber\\
        &\le& \tilde C \sum_{i=1}^{\infty}\chi_i \delta_{qM}^{\tilde Y(t)}(j-i+1).
    \end{eqnarray*}
    The result now follows as in the proof of Lemma \ref{lemma:depmeasure}(i) with Assumption \ref{ass3}\ref{ass3_dep}.
    
    (ii) We have for $n$ large enough that
    \[
        |\eta - \eta_{b_n}(t)|_1 = |\eta_1 - \theta(t)|_1 + |\eta_2 - b_n\theta'(t)|_1 < \iota/2\quad\text{ implies }\quad|(\eta_1 + \eta_2(u-t)b_n^{-1}) - \theta(t)|_1 \le |\eta_1 - \theta(t)|_1 + |\eta_2|_1 < \iota
    \]
    and $|\theta - \theta(t)|_1 < \iota$, $|u-t| \le b_n$ implies $|\theta - \theta(u)|_1 < \iota$ due to uniform continuity of $\theta(\cdot)$. Therefore, we have for $n$ large enough:
    \begin{eqnarray*}
        &&\big| \sup_{t,|\eta-\eta_{b_n}(t)|_1 < \iota/2}|M_i(t,\eta,u)| - \sup_{t,|\eta-\eta_{b_n}(t)|_1 < \iota/2}|M_i(t,\eta,u)^{*}|\big|\\
        &\le& \sup_{t,|\eta - \eta_{b_n}(t)|_1 < \iota/2}|\hat K_{b_n}(u-t)|\cdot |g(\tilde Z_i(u),\eta_1 + \eta_2 (u-t)b_n^{-1}) - g(\tilde Z_i(u),\eta_1 + \eta_2 (u-t)b_n^{-1})|\\
        &\le& \sup_{t,|\theta - \theta(t)|_1 < \iota}|\hat K_{b_n}(u-t)|\cdot |\tilde g_{\theta(u)}(\zeta_i,\tilde X_i(u),\theta) - \tilde g_{\theta(u)}(\zeta_i,\tilde X_i(u)^{*},\theta)|\\
        &\le& |\hat K|_{\infty}\cdot \sup_{|\theta - \theta(u)|_1 < \iota}|\tilde g_{\theta(u)}(\zeta_i,\tilde X_i(u),\theta) - \tilde g_{\theta(u)}(\zeta_i,\tilde X_i(u)^{*},\theta)|.
    \end{eqnarray*}
    The rest works as in (i).\\
    (iii) For $n$ large enough, it holds that $|u-t| \le b_n$ implies that $\sup_{s\in[0,1]}|\theta(t) + s d_u(t) - \theta(u)| < \iota$ due to uniform continuity of $\theta(\cdot)$. Thus
    \begin{eqnarray*}
        && \big| \sup_{t}|M_i^{(2)}(t,u)| - \sup_t |M_i^{(2)}(t,u)^{*}|\big|\\
        &\le& |\hat K|_{\infty}\sup_s|\theta''(s)|_{\infty} b_n^2 \sup_{|\theta - \theta(u)| < \iota}|\tilde g_{\theta(u)}(\zeta_i,\tilde X_i(u),\theta) - \tilde g_{\theta(u)}(\zeta_i, \tilde X_i(u)^{*}, \theta)|.
    \end{eqnarray*}
    The rest works as in (i).
\end{proof}

\section{Proofs of the Assumption sets}\label{sec:proofs_assumptions}
In this section, we prove that Assumption \ref{ass:case1} implies Assumption \ref{ass1} (Case 1) and that Assumption \ref{ass:case2} implies Assumption \ref{ass3} (Case 2).

We need the following general statements (cf. \cite{householder}, page 46 or \cite{dahlhaus2009}, proof of Proposition 2.1). Let $|x|_1 := \sum_{j=1}^{d}|x_j|$ denote the 1-norm for $x \in \IR^d$.

\begin{lemma}\label{lemma_matrixnorm_householder}
    Let $A \in \IR^{d \times d}$ be a matrix and let $\rho(A) := \max\{|\lambda|: \lambda\text{ eigenvalue of }A\}$ be the largest absolute eigenvalue of $A$. Let $\varepsilon > 0$. Then there exists an invertible matrix $M = M(\varepsilon) \in \IR^{d\times d}$, such that the norm
    \[
        |x|_M := |M^{-1}x|_1, \quad x\in\IR^d
    \]
    on $\IR^d$ and the corresponding matrix norm $|A|_M := \sup\{|A x|_M:|x|_M=1\}$ on $\IR^{d\times d}$ satisfies
    \[
        |A|_M \le \rho(A) + \varepsilon.
    \]
\end{lemma}

Similar as in the \cite{dahlhaus2009}, proof of Proposition 2.1, we obtain the following conclusion.

\begin{lemma}\label{lemma_matrixnorm}
    Let $A:[0,1] \to \IR^{d \times d}$ be a continuous function. Let $\rho := \sup_{u\in[0,1]}\rho(A(u)) < 1$.  Let $\varepsilon > 0$. Then there exist invertible matrices $M_1,...,M_L \in \IR^{d\times d}$ such that the following holds: There exists a partition $[0,1] = \bigcup_{k=1}^{L}I_k$ into intervals  $I_k$ such that for $u \in I_k$,
    \[
        |A(u)|_{M_k} \le \rho + \varepsilon.
    \]
    Furthermore, there exists a constant $c_0 > 0$ such that for any $A \in \IR^{d\times d}$,
    \[
        |A|_1 := \sum_{i,j=1}^{d}|A_{i,j}| \le c_0\cdot  \inf_{k=1,...,L}|A|_{M_k}.
    \]
\end{lemma}
\begin{proof}[Proof of Lemma \ref{lemma_matrixnorm}]
    We adopt the proof of \cite{kunsch1995note}. For $u\in[0,1]$, let $M(u)$ denote the matrix associated to $A(u)$ and $\frac{\varepsilon}{2}$ from Lemma \ref{lemma_matrixnorm_householder}. Since $v \mapsto A(v)$ is continuous, there exists $\delta(u) > 0$ such that for all $v \in (u-\delta(u),u+\delta(u))$,
    \begin{equation}
        |A(v)|_{M(u)} \le |A(v) - A(u)|_{M(u)} + |A(u)|_{M(u)} \le \frac{\varepsilon}{2} + \rho(A(u)) + \frac{\varepsilon}{2} \le \rho + \varepsilon.\label{lemma_matrixnorm_proof_eq1}
    \end{equation}
    Since $[0,1]$ is compact and $((u-\delta(u),u+\delta(u)))_{u\in[0,1]}$ is a covering of open sets, there exist finitely many $u_1,...,u_L \in [0,1]$ such that $[0,1] \subset \bigcup_{k=1}^{L}(u_k-\delta(u_k),u_k+\delta(u_k))$. Let $M_k := M(u_k)$.
    
    We now prove the first assertion. Then for any $v \in [0,1]$, let $k \in \{1,...,L\}$ be such that $v \in (u_k-\delta(u_k),u_k+\delta(u_k))$. Then by \reff{lemma_matrixnorm_proof_eq1}, $|A(v)|_{M_k} \le \rho + \varepsilon$. The second assertion follows since all norms are equivalent on $\IR^{d\times d}$, so in particular $|\cdot|_M$ and $|\cdot|_1$ are equivalent.
\end{proof}

The following result adopts a lemma of \cite{duflo1997}, Lemma 6.2.10 (Section 6.2) therein, to time-varying iterative models.

\begin{lemma}\label{lemma_randomiterativemodel}
    Suppose that $z_t$, $t > -p$ and $\eta_t$, $t > 0$ are two sequences of positive real numbers such that for each $t\in\IN$, there exist $a_{1},...,a_{p}:[0,1] \to [0,\infty)$ with $\sup_{u\in[0,1]}\sum_{k=1}^p a_{k}(u) < 1$ and 
    \[
        z_t \le \sum_{k=1}^{p}a_{k}(t/n) z_{t-k} + \eta_t, \quad t=1,...,n.
    \]
    Then there exists $0 \le a < 1$ and some constant $c \ge 0$ such that
    \[
        z_t \le c\cdot \Big(\sum_{k=0}^{t-1}a^k \eta_{t-k} + a^t \cdot |(z_0,...,z_{-p+1})\tran|_1\Big), \quad t = 1,...,n.
    \]
\end{lemma}
\begin{proof}[Proof of Lemma \ref{lemma_randomiterativemodel}]
    Define the companion matrix
    \[
        C(u) := \begin{pmatrix}a_{1}(u) & a_{2}(u) & \dots & \dots &  a_{p}(u)\\
        1 & 0 & \dots & \dots & 0\\
        0 & 1 & 0 & \dots & 0\\
        \vdots & \ddots & \ddots & \ddots & \vdots\\
        0 & \dots & 0 & 1 & 0\\
        \end{pmatrix}
    \]
    to the characteristic polynomial $Q_{u}(x) = 1 - a_{1}(u) z - ... - a_{p}(u) z^p$. Because of  $\sup_{u\in[0,1]}\sum_{k=1}^p a_{k}(u) < 1$, the polynomial $Q_u$ is causal, thus $\tilde a := \sup_{u\in[0,1]}\rho(C(u)) < 1$.
    
    Define $\xi_t := -z_t + (\sum_{k=1}^{p}a_{k}(t/n)z_{t-k} + \eta_t) \ge 0$. Then $z_t^{(p)} := (z_t,...,z_{t-p+1})\tran$ and $\varsigma_t := (\eta_t - \xi_t, 0,...,0)\tran$ satisfy
    \[
        z_t^{(p)} = C(t/n)\cdot z_{t-1}^{(p)} + \varsigma_t, \quad t \in \IN.  
    \]
    Let $H_t := (\eta_t, 0,...,0)\tran$. With $C_{s,t} := C(t/n) \cdot ... \cdot C((s+1)/n)$, we obtain (the inequalities are meant component-wise)
    \begin{eqnarray*}
        z_t^{(p)} &=& C_{0,t}\cdot z_0^{(p)} + \sum_{k=0}^{t-1}C_{t-k,t}\varsigma_{t-k} \le C_{0,t}\cdot z_0^{(p)} + \sum_{k=0}^{t-1}C_{t-k,t}H_{t-k}.
    \end{eqnarray*}
    Fix some $\alpha \in (\tilde \alpha,1)$. Application of Lemma \ref{lemma_matrixnorm} to $C(u)$ and $\varepsilon = \alpha - \tilde \alpha$ yields $L\in\IN$, $M_1,...,M_L \in \IR^{p\times p}$ and a partition $[0,1] = \bigcup_{k=1}^{L}I_k$ into intervals $I_k$ (larger $k$ means that $I_k$ contains larger values) such that for $u \in I_k$,
    \[
        |C(u)|_{M_k} \le \alpha.
    \]
    We obtain again with Lemma \ref{lemma_matrixnorm} that
    \begin{eqnarray}
        |C_{s,t}|_1 &=& \Big|\prod_{k=1}^{L}\Big(\prod_{u \in \{\frac{s+1}{n},...,\frac{t}{n}\}\cap I_k}C(u)\Big)\Big|_1 \le \prod_{k=1}^{L}\Big|\prod_{u \in \{\frac{s+1}{n},...,\frac{t}{n}\}\cap I_k}C(u)\Big|_1\nonumber\\
        &\le& c_0^L\prod_{k=1}^{L}\Big|\prod_{u \in \{\frac{s+1}{n},...,\frac{t}{n}\}\cap I_k}C(u)\Big|_{M_k}\le c_0^L\prod_{k=1}^{L}\prod_{u \in \{\frac{s+1}{n},...,\frac{t}{n}\}\cap I_k}\big|C(u)\big|_{M_k}\nonumber\\
        &\le& c_0^L\prod_{k=1}^{L}\prod_{u \in \{\frac{s+1}{n},...,\frac{t}{n}\}\cap I_k}\alpha =  c_0^{L}\alpha^{t-s}.\label{lemma_randomiterativemodel_eq1}
    \end{eqnarray}
\end{proof}

\subsection{Case 1: Recursively defined models}

\begin{proposition}\label{example:tvrec} If Assumption \ref{ass:case1} holds, then Assumption \ref{ass1} is fulfilled with every $r = 2 + \tilde a$, $\tilde a < a$, the corresponding $M$ and $\gamma > 2$ arbitrarily large.

It holds that $V(t) = \Lambda(t)$. If (i) $\IE \zeta_0^3 = 0$, or (ii) $\mu(x,\theta) \equiv 0$ or (iii) $\sigma(x,\theta) \equiv \beta_0$ and $\IE m(\tilde X_0(t)) = 0$, then $$I(t) = \big(\begin{smallmatrix}I_k & 0\\
0 & (\IE \zeta_0^4 - 1) I_{l+1}/2\end{smallmatrix}\big)\cdot V(t),$$ where $I_d$ denotes the $d$-dimensional identity matrix.
\end{proposition}

\begin{proof}[Proof of Proposition \ref{example:tvrec}] Choose $0 < \tilde a< a$ small enough such that \reff{example:parameterconditions} holds with $\|\zeta_0\|_{(2+\tilde a)M}$ replaced by $\|\zeta_0\|_{2M}$ (this is possible due to continuity of the term in $\tilde a = 0$).
Let $q = (2+\tilde a)M$. Let $\nu = (\nu_0,\ldots,\nu_l)\tran$ and $m = (m_1,\ldots,m_k)\tran$.

Define $W_n(y,t) := G_{\zeta_n}(G_{\zeta_{n-1}}(...G_{\zeta_1}(y,t)...,t),t)$, where
\[
    G_{\zeta}(y,t) := \mu(y,\theta(t)) + \sigma(y,\theta(t))\zeta.
\]
We have
\begin{eqnarray*}
    |\sigma(y,\theta)^2 - \sigma(y',\theta)^2| &\le& \sum_{i=0}^{l}\beta_i |\nu_i(y) - \nu_i(y')|\\
    &\le& \sum_{i=0}^{l}\sqrt{\beta_i} |y-y'|_{\rho_{i\cdot},1}\cdot \big(\sqrt{\beta_i \nu_i(y)} + \sqrt{\beta_i \nu_i(y')}\big)\\
    &\le& \sum_{i=0}^{l}\sqrt{\beta_i} |y-y'|_{\rho_{i\cdot},1} \cdot \big(\sigma(y,\theta) + \sigma(y',\theta)\big),
\end{eqnarray*}
i.e.
\[
    |\sigma(y,\theta) - \sigma(y',\theta)| \le \sum_{i=0}^{l}\sqrt{\beta_i}|y-y'|_{\rho_{i\cdot},1},
\]
We have
\begin{eqnarray}
    \|G_{\zeta_0}(y,t) - G_{\zeta_0}(y',t)\|_q &\le& |\mu(y,t) - \mu(y',t)| + |\sigma(y,t) - \sigma(y',t)|\cdot \|\zeta_0\|_q\nonumber\\
    &\le& \sum_{i=1}^{k}|\alpha_i(t)|\cdot |m_i(y) - m_i(y')| + \sum_{i=0}^{l}\sqrt{\beta_i(t)}|y-y'|_{\rho_{i\cdot,1}}\nonumber\\
    &\le& \sum_{i=1}^{k}|\alpha_i(t)|\cdot |y-y'|_{\kappa_{i\cdot},1} + \sum_{i=0}^{l}\sqrt{\beta_i(t)}|y-y'|_{\rho_{i\cdot,1}}\nonumber\\
    &=& \sum_{i=1}^{k}|\alpha_i(t)|\cdot \sum_{j=1}^{p}\kappa_{ij}|y_j-y_j'| + \sum_{i=0}^{l}\sqrt{\beta_i(t)}\sum_{j=1}^{p}\rho_{ij}|y_j-y_j'|\nonumber\\
    &=& \sum_{j=1}^{p}\Big(\sum_{i=1}^{k}\kappa_{ij}|\alpha_i(t)| + \sum_{i=0}^{l}\rho_{ij}\sqrt{\beta_i(t)}\Big)\cdot |y_j - y_j'|.\label{proooooooof1}
\end{eqnarray}
Define
\[
    a_j(t) := \sum_{i=1}^{k}\kappa_{ij}|\alpha_i(t)| + \sum_{i=0}^{l}\rho_{ij}\sqrt{\beta_i(t)}.
\]
Then we have with $z_n(t) := \|W_n(y,t) - W_n(y',t)\|_q$, $n \ge p+1$:
\[
    z_n(t) \le \sum_{j=1}^{p}a_j(t)\cdot z_{n-j}(t).
\]
Since $\sup_{t\in[0,1]}\sum_{j=1}^{p}a_j(t) < 1$ for $\theta(t) \in \Theta$,  Lemma \ref{lemma_randomiterativemodel} implies that with some $a\in(0,1)$, $c > 0$,
\[
    z_n(t) \le c\cdot a^{n-p}\cdot |(z_{p},...,z_1)\tran|_1.
\]
By definition of $W_n(\cdot)$ and \reff{proooooooof1}, $\|z_j(t)\|_q \le c\cdot |y-y'|_1$ for $j \in \{1,...,p\}$. Thus, for some constant $c > 0$,
\[
    \|W_n(y,t) - W_n(y',t)\|_q \le c\cdot a^{n}\cdot |y-y'|_1.
\]
By \cite{wushao2004}, Theorem 2, we obtain $\sup_{t\in[0,1]}\delta^{\tilde Y(t)}(k) = O(\tilde \rho^k)$ with some $\tilde \rho \in (0,1)$ and \[
    D :=\sup_{t\in[0,1]}\|\tilde Y_0(t)\|_q < \infty.
\]
Thus, \ref{ass1}\ref{ass1_dep} holds 
with arbitrarily large $\gamma > 0$.



By Lipschitz continuity of $\theta$ with constant $L_{\theta}$, we have
\begin{equation}
    |\mu(y,\theta(t)) - \mu(y,\theta(t'))| \le L_{\theta}|t-t'| \sum_{i=1}^{k}|m_i(y)|,\label{eq:ex4}
\end{equation}
and
\begin{eqnarray*}
    |\sigma(y,\theta(t))^2 - \sigma(y,\theta(t'))^2| &\le& L_{\theta}|t-t'| \sum_{i=0}^{l}\sqrt{\nu_i(y)}\frac{1}{2\beta_{min}^{1/2}}\big(\sqrt{\beta_i(t) \nu_i(y)} + \sqrt{\beta_i(t')\nu_i(y)}\big)\\
    &\le& \frac{L_{\theta}}{2 \beta_{min}^{1/2}}|t-t'| \sum_{i=0}^{l}\sqrt{\nu_i(y)} (\sigma(y,\theta(t)) + \sigma(y,\theta(t'))),
\end{eqnarray*}
which shows that
\begin{equation}
    |\sigma(y,\theta(t)) - \sigma(y,\theta(t'))| \le \frac{L_{\theta}}{2 \beta_{min}^{1/2}}\sum_{i=0}^{l}\sqrt{\nu_i(y)}\label{eq:ex5}.
\end{equation}
Note that \reff{example:tvrec_eq1} implies
\[
    m_i(y), \sqrt{\nu_i(y)} \le C_1|y|_1 + C_2,
\]
with some constants $C_1, C_2 > 0$. By \reff{eq:ex4}, \reff{eq:ex5}, we have for $t\not=t'$
\begin{eqnarray*}
    \|G_{\zeta_0}(y,t) - G_{\zeta_0}(y,t')\|_q &\le& |\mu(y,\theta(t)) - \mu(y,\theta(t'))| + \|\zeta_0\|_q |\sigma(y,\theta(t)) - \sigma(y,\theta(t'))|\\
    &\le& C_3|t-t'|\big(1 + |y|_1\big),
\end{eqnarray*}
with some constant $C_3 > 0$. This implies for $q \ge 1$,
\begin{eqnarray*}
    \|\tilde Y_i(t) - \tilde Y_i(t')\|_q &=&  \|G_{\zeta_i}(\tilde Y_{i-1}(t),...,\tilde Y_{i-p}(t),t) - G_{\zeta_i}(\tilde Y_{i-1}(t'),...,\tilde Y_{i-p}(t'),t')\|_q\\
    &\le& \|G_{\zeta_i}(\tilde Y_{i-1}(t),...,\tilde Y_{i-p}(t),t) - G_{\zeta_i}(\tilde Y_{i-1}(t'),...,\tilde Y_{i-p}(t'),t)\|_q\\
    &&\quad\quad + \|G_{\zeta_i}(\tilde Y_{i-1}(t'),...,\tilde Y_{i-p}(t'),t) - G_{\zeta_i}(\tilde Y_{i-1}(t'),...,\tilde Y_{i-p}(t'),t')\|_q\\
    &\le& \sum_{j=1}^{p}a_j(t)\cdot \|Y_{i-j}(t) - Y_{i-j}(t')\|_q + C_3 |t-t'|\big(1+\|\,|(\tilde Y_{i-1}(t'),...,\tilde Y_{i-p}(t'))|_1\|_q\big)\\
    &\le& \sum_{j=1}^{p}a_j(t)\cdot \|Y_{i-j}(t) - Y_{i-j}(t')\|_q + C_3 \big(1+pD\big)\cdot |t-t'|.
\end{eqnarray*}
with $D= \sup_{t}\|\tilde Y_0(t)\|_q < \infty$. We have $\rho := \sup_{t\in[0,1]}\sum_{j=1}^{p}a_j(t) < 1$, thus by stationarity,
\[
    \|\tilde Y_0(t) - \tilde Y_0(t')\|_q \le \underbrace{\Big(\sum_{j=1}^{p}a_j(t)\Big)}_{\le \rho}\cdot \|\tilde Y_0(t) - \tilde Y_0(t')\|_q + C_3 \big(1+pD\big)\cdot |t-t'|.
\]
Thus
\begin{equation}
     \|\tilde Y_0(t) - \tilde Y_0(t')\|_q \le \frac{C_3 \big(1+pD\big)}{1-\rho}\cdot |t-t'|.\label{proofffff1}
\end{equation}
It remains to show that
\[
    \|Y_i - \tilde Y_i(i/n)\|_q \le C n^{-1}.
\]
Here,
\begin{eqnarray*}
    &&\|Y_i - \tilde Y_i(i/n)\|_q =  \|G_{\zeta_i}( Y_{i-1},...,Y_{i-p},i/n) - G_{\zeta_i}(\tilde Y_{i-1}(i/n),...,\tilde Y_{i-p}(i/n),i/n)\|_q\\
    &\le& \sum_{j=1}^p a_j(i/n) \|Y_{i-j}-\tilde Y_{i-j}(i/n) \|_q\\
    &\le& \sum_{j=1}^p a_j(i/n) \|Y_{i-j}-\tilde Y_{i-j}((i-j)/n \vee 0) \|_q+ \sum_{j=1}^p a_j(i/n) \|\tilde Y_{i-j}((i-j)/n \vee 0)-\tilde Y_{i-j}(i/n) \|_q \\
    &\le& \sum_{j=1}^p a_j(i/n) \|Y_{i-j}-\tilde Y_{i-j}((i-j)/n \vee 0) \|_q +\rho C_1 p/n 
\end{eqnarray*}
for some $C_1>0$, where the last step is due to \reff{proofffff1}. Define $z_i:=\|Y_i - \tilde Y_i(i/n \vee 0)\|_q$. Note that $z_i=0$ for $i \le 0$. Using $\eta_i=\rho C_1 p/n$ at Lemma \ref{lemma_randomiterativemodel}, one obtains that $$\|Y_i - \tilde Y_i(i/n)\|_q \leq C\rho p/n$$
for some $C>0$. It follows that $\sup_{i,n}\|Y_i\|_q \le \sup_{i,n}\|Y_i - \tilde Y_i(i/n)\|_q + \sup_{i,n}\|\tilde Y_i(i/n)\|_q < \infty$. Thus, we have shown Assumption \ref{ass1}\ref{ass1_stat},

We now inspect the properties of the function $\ell$. First note that the recursion of the stationary approximation,
\[
    \tilde Y_i(t) = \mu(\tilde X_i(t), \theta(t)) + \sigma(\tilde X_i(t),\theta(t)) \zeta_i,
\]
implies $\IE \tilde Y_0(t) = 0$ and $\IE \tilde Y_0(t)^2 = \IE \mu(\tilde X_0(t),\theta(t))^2 + \IE \sigma(\tilde X_0(t),\theta(t))^2 \ge \beta_{min}\nu_{min} > 0$. Furthermore, for $L(t,\theta) := \IE \ell(\tilde Z_0(t),\theta)$ it holds that
\begin{eqnarray}
    L(t,\theta) - L(t,\theta(t)) &=& \IE\Big(\frac{\mu(\tilde X_0(t),\theta) - \mu(\tilde X_0(t),\theta(t))}{\sigma(\tilde X_0(t),\theta)}\Big)^2\nonumber\\
    &&\quad\quad\quad\quad + \IE\Big[\frac{\sigma(\tilde X_0(t),\theta(t))^2}{\sigma(\tilde X_0(t),\theta)^2}  - \log\frac{\sigma(\tilde X_0(t),\theta(t))^2}{\sigma(\tilde X_0(t),\theta)^2}-1\Big].\label{eq:ex7}
\end{eqnarray}
In the following we use the notation $|x|_{A}^2 := x\tran A x$ for a weighted vector norm. Note that
\begin{equation}
    \IE\Big(\frac{\mu(\tilde X_0(t),\theta) - \mu(\tilde X_0(t),\theta(t))}{\sigma(\tilde X_0(t),\theta)}\Big)^2 \ge c_0 |\alpha - \alpha(t)|_{M_1(t)}^2,\label{eq:ex8}
\end{equation}
with $c_0 = (\max_{\theta \in \Theta}\max_i \theta_i^2)^{-1}$ and $M_1(t) := \IE[\frac{m(\tilde X_0(t)) m(\tilde X_0(t))\tran}{\Ii \nu(\tilde X(t))\nu(\tilde X(t))\tran \Ii}]$. If $M_1(t)$ was not positive definite, this would imply that there exists $v \in \IR^k$ such that $v' M(t)v = 0$, which in turn would imply $v'\mu(\tilde X_0(t)) \mu(\tilde X_0(t))v = 0$ a.s. and thus non-positive definiteness of $\IE[\mu(\tilde X_0(t))\mu(\tilde X_0(t))\tran]$ which is a contradiction to the assumption.\\
By a Taylor expansion of $f(x) = x-\log(x)-1$, we obtain
\begin{eqnarray}
    && \IE\Big[\frac{\sigma(\tilde X_0(t),\theta(t))^2}{\sigma(\tilde X_0(t),\theta)^2}  - \log\frac{\sigma(\tilde X_0(t),\theta(t))^2}{\sigma(\tilde X_0(t),\theta)^2}-1\Big]\nonumber\\
    &\ge& \frac{1}{2}\IE\Big[\frac{(\sigma(\tilde X_0(t),\theta)^2 - \sigma(\tilde X_0(t),\theta(t))^2)^2}{(\sigma(\tilde X_0(t),\theta)^2 - \sigma(\tilde X_0(t),\theta(t))^2)^2 + \sigma(\tilde X_0(t),\theta)^4}\Big]\nonumber\\
    &\ge& \frac{c_0}{10}|\beta - \beta(t)|_{M_2(t)}^2,\label{eq:ex9}
\end{eqnarray}
where $M_2(t) = \IE[\frac{\nu(\tilde X_0(t)) \nu(\tilde X_0(t))\tran}{\Ii \nu(\tilde X(t))\nu(\tilde X(t))\tran \Ii}]$ is positive definite by assumption (use a similar argumentation as above). By \reff{eq:ex7}, \reff{eq:ex8} and \reff{eq:ex9} we conclude that $\theta \mapsto L(t,\theta)$ is uniquely minimized in $\theta = \theta(t)$. This shows \ref{ass1}\ref{ass1_model}.

Omitting the arguments $z = (y,x)$ and $\theta$, we have
\begin{eqnarray}
    \ell &=& \frac{1}{2}\Big[\frac{(y-\langle \alpha,m\rangle)^2}{\langle \beta,\nu\rangle} + \log \langle \beta,\nu\rangle\Big]\label{eq:exlikelihood0},\\
    \nabla_{\theta}\ell &=& -\frac{\nabla_{\theta}m}{\sigma}\Big(\frac{y - m}{\sigma}\Big) + \frac{\nabla_{\theta}(\sigma^2)}{2\sigma^2}\Big[1 - \Big(\frac{y-m}{\sigma}\Big)^2\Big]\nonumber\\
    &=& \begin{pmatrix}
        -\frac{m}{\sigma}\Big(\frac{y-m}{\sigma}\Big)\\
        \frac{\nu}{2\sigma^2}\big[1-\big(\frac{y-m}{\sigma}\big)^2\big]
    \end{pmatrix} = \begin{pmatrix}
    \frac{m}{\langle \beta,\nu\rangle}(y - \langle \alpha,m\rangle)\\
    \frac{\nu}{2\langle \beta,\nu\rangle}\big(1 - \frac{(y - \langle \alpha,m\rangle)^2}{\langle \beta,\nu\rangle}\big)
    \end{pmatrix},\label{eq:exlikelihood1} \\
    \nabla_{\theta}^2\ell &=& \frac{\nabla_{\theta}m \nabla_{\theta}m\tran}{\sigma^2} + \big(\frac{y-m}{\sigma}\big)\cdot\Big[ \frac{\nabla_{\theta}m \nabla_{\theta}(\sigma^2)\tran + \nabla_{\theta}(\sigma^2) \nabla_{\theta}m\tran}{\sigma^3} - \frac{\nabla_{\theta}^2m}{\sigma}\Big]\nonumber\\
    &&\quad\quad\quad + \frac{\nabla_{\theta}^2(\sigma^2)}{2\sigma^2}\Big[1 - \big(\frac{y-m}{\sigma}\big)^2\Big] + \frac{\nabla_{\theta}(\sigma^2) \nabla_{\theta}(\sigma^2)\tran}{2\sigma^4}\Big[2\Big(\frac{y-m}{\sigma}\big)^2 - 1\Big]\nonumber\\
    &=& \begin{pmatrix}
        \frac{mm\tran}{\sigma^2} & \frac{y-m}{\sigma^2}\cdot m \nu\tran\\
        \frac{y-m}{\sigma^2}\cdot \nu m\tran & \frac{\nu \nu\tran}{2\sigma^4}\big[2 \big(\frac{y-m}{\sigma}\big)^2 - 1\big]
    \end{pmatrix}\nonumber\\
    &=& \begin{pmatrix}
    \frac{m m\tran}{\langle \beta,\nu\rangle} & \frac{y-\langle \alpha,m\rangle}{\langle \beta,\nu\rangle^2}\cdot m \nu\tran\\
    \frac{y-\langle \alpha,m\rangle}{\langle \beta,\nu\rangle^2}\cdot\nu m\tran & \frac{\nu \nu\tran}{2\langle \beta,\nu\rangle^2}\big[2 \frac{(y-\langle \alpha,m\rangle)^2}{\langle \beta,\nu\rangle}- 1\big]
    \end{pmatrix}\label{eq:exlikelihood2}.
\end{eqnarray}
Since $\zeta_1$ is independent of $\tilde X_{0}(t) \in \sF_{0}$ and $\IE \zeta_1 = 0$, $\IE \zeta_1^2  = 1$, we conclude that
\[
    \IE[\nabla_{\theta}\ell(\tilde Z_0(t),\theta(t))|\sF_{t-1}] = \IE\Big[-\frac{\mu(\tilde X_{j}(t),\theta(t))}{\sigma(\tilde X_0(t),\theta(t))}\zeta_{0} + \frac{\nu(\tilde X_0(t),\theta(t))}{2\sigma(\tilde X_j(t),\theta(t))^2}(1-\zeta_{0}^2)\big|\sF_{t-1}\Big] = 0,
\]
i.e. $\nabla_{\theta}\ell(\tilde Z_1(t),\theta(t))$ is a martingale difference sequence, showing that $V(t) = \Lambda(t)$.
We furthermore have that (we omit the arguments $(\tilde X_0(t),\theta(t))$ of $\mu$, $\sigma$ in the following):
\[
    V(t) = \IE \nabla_{\theta}^2\ell(\tilde Z_0(t),\theta(t)) = \begin{pmatrix}
        \IE\big[\frac{mm\tran}{\langle \beta,\nu\rangle}\big] & 0\\
        0 & \IE\big[\frac{\nu\nu\tran}{2\langle \beta,\nu\rangle^2}\big]
    \end{pmatrix}.
\]
With a similar argumentation as above, we conclude that $V(t)$ is positive definite (which then implies by continuity that the smallest eigenvalue of $V(t)$ is bounded away from 0 uniformly in $t$).
By the martingale difference property, $I(t) = \Lambda(t)$. Omitting the arguments $(\tilde X_0(t),\theta(t))$, 
\begin{eqnarray*}
    I(t) &=& \IE[\nabla_{\theta}(\tilde Z_j(t),\theta(t)) \nabla_{\theta}(\tilde Z_0(t),\theta(t))\tran]\\
    &=& \begin{pmatrix}
        \IE\big[\frac{mm\tran}{\sigma^2}\big] & \IE[\zeta_0^3]\cdot \IE\big[\frac{m \nu\tran}{2\sigma^3}\big]\\
        \IE[\zeta_0^3]\cdot \IE\big[\frac{\nu m\tran}{2\sigma^3}\big] & \frac{\IE[\zeta_0^4] - 1}{4}\cdot \IE\big[\frac{\nu \nu\tran}{\sigma^4}\big]
    \end{pmatrix}\\
    &=& \IE\Big[\frac{1}{\sigma^2}\begin{pmatrix}
        m\\
        \frac{\IE[\zeta_0^3]}{2\sigma}\nu
    \end{pmatrix}\tran \begin{pmatrix}
        m\\
        \frac{\IE[\zeta_0^3]}{2\sigma}\nu
    \end{pmatrix}\Big] + \begin{pmatrix}
        0 & 0\\
        0 & \big(\frac{\IE[\zeta_0^4] - \IE[\zeta_0^3]^2-1}{4}\big) \IE\big[\frac{\nu\nu\tran}{\sigma^4}\big]
    \end{pmatrix},
\end{eqnarray*}
which is positive semidefinite since $\IE[\zeta_0^3] = \IE[\zeta_0 (\zeta_0^2 - 1)] \le \IE[\zeta_0^2]^{1/2}\IE[(\zeta_0^2 - 1)^2]^{1/2} = (\IE[\zeta_0^4] - 1)^{1/2}$. Positive definiteness follows from the fact that $(v_1,v_2)\tran I(t)(v_1,v_2) = 0$ implies $\nu\tran v_2 = 0$ a.s. from the last summand and $v_1\tran m + \frac{\IE[\zeta_0^3]}{2\sigma}v_2 \tran \nu = 0$ a.s. from the first summand,  i.e. $v_1\tran m = 0$ a.s. which leads to a contradiction to either the positive definiteness of $\IE[\nu\nu\tran]$ or $\IE[mm\tran]$. So we obtain that Assumption \ref{ass1}\ref{ass1_matrix} is fulfilled.

A careful inspection of \reff{eq:exlikelihood0}, \reff{eq:exlikelihood1} and \reff{eq:exlikelihood2} shows that $\ell, \nabla_{\theta}\ell, \nabla_{\theta}^2\ell \in \sH(3,\tilde \chi,\tilde C)$ with some $\tilde C > 0$ and $\tilde \chi = (1,\ldots,1,0,0,\ldots)$ consisting of $\max\{k,l\}$ ones followed by zeros, which shows Assumption \ref{ass1}\ref{ass1_smooth}. In the special case $\mu(x,\theta) \equiv 0$, it seems as if no direct improvement of the value $M$ is possible. In the special case of $\sigma(x,\theta)^2 \equiv \beta_0$, we have
\begin{eqnarray*}
    \ell &=& \frac{1}{2}\Big[\frac{(y-\langle \alpha,m\rangle)^2}{\beta_0} + \log \beta_0\Big],\\
    \nabla_{\theta}\ell &=& \begin{pmatrix}
        \frac{m}{\beta_0}(y- \langle \alpha,m\rangle)\\
        \frac{1}{2\beta_0}(1-\frac{(y - \langle \alpha,m\rangle)^2}{\beta_0})
    \end{pmatrix},\\
    \nabla_{\theta}^2\ell &=& \begin{pmatrix}
        \frac{mm\tran}{\beta_0} & \frac{y-\langle \alpha,m\rangle}{\beta_0^2}m\\
        \frac{y-\langle \alpha,m\rangle}{\beta_0^2}m\tran & \frac{1}{2\beta_0^2}\big[2\frac{(y-\langle \alpha,m\rangle)^2}{\beta_0}-1\big]
    \end{pmatrix},
\end{eqnarray*}
which implies that $\ell, \nabla_{\theta}\ell, \nabla_{\theta}^2\ell \in \sH(2,\tilde \chi,\tilde C)$.

\end{proof}

\subsection{Case 2: tvGARCH}

\textbf{The class $\sH^{mult}_{\iota}(M,\chi,\bar C)$}. 
    
    In the tvGARCH case, we will need a more specific structure of $\ell$ to obtain the results under weak moment assumptions. To make use of independencies occuring in the analysis of the tvGARCH likelihood, let us introduce the class $\sH^{mult}_{\iota}(M,\chi,\bar C)$ which consists of functions $g:\IR \times \IR^{\IN} \times \Theta$ such that $\sup_{\theta \in \Theta}\frac{|g(y,0,\theta)|}{1+|y|^{M}} \le \bar C$ and
\[
    \sup_{|\theta-\tilde \theta|_1 < \iota}\sup_y \sup_{x\not= x'}\frac{|g(y,x,\theta) - g(y,x',\theta)|}{|x-x'|_{\chi}( 1 + |x|_{\chi}^{M-1} + |x'|_{\chi}^{M-1})(1+|y|^M)} \le \bar C,
\]
\[
    \sup_{x,y}\sup_{\theta\not=\theta', |\theta - \tilde \theta|_1 < \iota, |\theta' - \tilde \theta|_1 < \iota} \frac{|g(y,x,\theta) - g(y,x,\theta')|}{|\theta - \theta'|_1 (1+|x|_{\chi}^{M})(1+|y|^{M})} \le \bar C.
\]

A slightly different set of assumptions (Assumption \ref{ass3}) which is specifically designed for conditional heteroscedastic models (leading to weaker moment assumptions)  is given now.

\begin{assumption}[Heteroscedastic recursively defined time series case]\label{ass3}
    Let $\zeta_i$, $i\in\IZ$ be an i.i.d. sequence. Assume that for any $t\in [0,1]$ it holds that
    \[
        \tilde Y_i(t) = F(\tilde X_i(t), \theta(t), \zeta_i), \quad i \in \IZ,
    \]
    where $F$ is some measurable function. Let
    \[
        \tilde \ell_{\tilde \theta}(y,x,\theta) := \ell(F(x,\tilde \theta,y),x,\theta).
    \]
    Suppose that for some $r \ge 2$,
     \begin{enumerate}[label=(A\arabic*'),ref=(A\arabic*')]
      \item\label{ass3_smooth} $\ell$ is twice continuously differentiable w.r.t. $\theta$. There exists $M \ge 1$ such that for each $s > 0$, there exist $\chi^{(s)} = (\chi^{(s)}_j)_{j=1,2,...}$ with $\chi_j^{(s)} = O(j^{-(1+\gamma)})$ and $\bar C^{(s)} > 0$, such that
        \begin{itemize}
            \item $\ell, \nabla_{\theta} \ell, \nabla_{\theta}^2 \ell \in \sH(2M(1+s),\chi^{(s)},\bar C^{(s)})$.
            \item \begin{equation}
           \hspace{-0.3 in}    \sup_{\theta}\sup_{z\not=z'}\frac{|\ell(z,\theta) - \ell(z',\theta)|}{|z-z'|_{\hat \chi^{(s)},s}^s(1+|z|_{\hat \chi}^{M} + |z'|_{\hat \chi}^{M}) + |z-z'|_{\hat \chi^{(s)}}(1+|z|_{\hat \chi}^{M-1} + |z|_{\hat \chi}^{M-1})^{1+s})} \le \bar C^{(s)}.\label{eq:additionalgarchlipschitz}
            \end{equation}
            \item There exists $\iota > 0$ such that  $\nabla_{\theta}\tilde \ell, \nabla_{\theta}^2\tilde \ell \in \sH^{mult}_{ \iota}(M(1+s),\chi^{(s)},\bar C^{(s)})$
        \end{itemize}
        \item\label{ass3_theta} \ref{ass1_theta} holds,
        \item\label{ass3_model} \ref{ass1_model} holds,
        \item \label{ass3_matrix} \ref{ass1_matrix} holds,
        \item\label{ass3_stat} \ref{ass1_stat} holds  and $\|\zeta_{0}\|_{rM} \le D$.
        \item\label{ass3_dep} $\sup_{t\in[0,1]}\delta_{rM}^{\tilde Y(t)}(k) = O(\rho^k)$ with some $\rho \in (0,1)$.
    \end{enumerate}

\end{assumption}

We now show that Assumption \ref{ass:case2} implies Assumption \ref{ass3}.

\begin{proposition}
    \label{example:garch}
    Let Assumption \ref{ass:case2} hold. Then Assumption \ref{ass3} is fulfilled with every $r = 2+\tilde a$, $\tilde a < \frac{a}{2}$ and $M=2$. It holds that $\Lambda(t) = I(t) = ((\IE \zeta_0^4 - 1)/2)V(t)$.
\end{proposition}

For the following proposition, let us introduce the following notation. For a random vector $v\in\IR^d$ and $q > 1$, we define $\|v\|_q := (\|v_j\|_q)_{j=1,...,d}$. Similar for a random matrix $A \in \IR^{d\times d}$, define $\|A\|_q := (\|A_{jk}\|_q)_{j,k=1,...,d}$. If $A,v$ are independent, then for any $j \in \{1,...,d\}$,
\begin{equation}
    (\|A v\|_q)_j = \Big\|\sum_{k=1}^{d}A_{jk}v_k\Big\|_q \le \sum_{k=1}^{d}\|A_{jk}v_k\|_q = \sum_{k=1}^{d}\|A_{jk}\|_q \|v_k\|_q = (\|A\|_q\cdot \|v\|_q)_j.\label{product_rule_qnorm}
\end{equation}
In particular, we write $\|Av\|_q \le \|v\|_q$ in this case which means that the inequality holds component-wise.

\begin{proof}[Proof of Proposition \ref{example:garch}] For $t \in [0,1]$, we abbreviate $M_i(t) := M_i(\theta(t))$. Since $\IE[|\varepsilon_0|^{4+a}] < \infty$ for some $a > 0$, the mapping $\Phi: [0,1]^2\times [1,1+\frac{a}{2})\to [0,\infty), (t,t',q) \mapsto \rho(\|M_0(t)\otimes M_0(t')\|_q)$ is continuous in each $(t,t',1)$ with $\Phi(t,t',1) = \rho(\IE[M_0(t)\otimes M_0(t')]) < 1$. Thus there exists $\tilde q > 1$ such that
\begin{equation}
    \sup_{t,t'\in[0,1]}\rho(\|M_0(t)\otimes M_0(t')\|_{\tilde q}) = \sup_{t,t'\in[0,1]}\Phi(t,t',\tilde q) < 1.\label{largest_eigenvalue_garch_bounded0}
\end{equation}
In particular, we have
\begin{equation}
    \sup_{t\in[0,1]}\rho(\|M_0(t)^{\otimes 2}\|_{\tilde q}) < 1.\label{largest_eigenvalue_garch_bounded1}
\end{equation}
Similarly, since $\IE[\zeta_1^4] \ge 1$, we conclude from $\sup_{t\in[0,1]}\rho(\IE[M_0(t)]) < 1$ that (w.l.o.g. with the same $\tilde q > 1$ as above)
\begin{equation}
    \sup_{t\in[0,1]}\rho(\|M_0(t)\otimes I\|_{\tilde q}) < 1.\label{largest_eigenvalue_garch_bounded2}
\end{equation}

Let $M = 1$. Fix $t \in [0,1]$. Consider the recursion of the corresponding stationary approximation
\begin{eqnarray}
    \tilde Y_i(t) &=& \tilde \sigma_i(t)^2 \zeta_i^2,\nonumber\\
    \tilde \sigma_i(t)^2 &=& \alpha_0(t) + \sum_{j=1}^{m}\alpha_j(t) \tilde Y_{i-j}(t) + \sum_{j=1}^{l}\beta_{j}(t)\tilde \sigma_{i-j}(t)^2.\label{eq:garchstat}
\end{eqnarray}
Define
\begin{eqnarray*}
    \tilde P_i(t) &:=& (\tilde Y_i(t),\ldots,\tilde Y_{i-m+1}(t),\tilde \sigma_{i}(t)^2,\ldots,\tilde \sigma_{i-l+1}(t)^2)\tran,\\
    a_i(t) &:=& (\alpha_0(t) \zeta_i^2,0,\ldots,0,\alpha_0(t),0,\ldots,0)\tran.
\end{eqnarray*}
For brevity, let $M_i(t) = M_i(\theta(t))$. Following Section 3.1 in \cite{existencegarch}, the model \reff{eq:garchstat} admits the representation
\begin{equation}
    \tilde P_i(t) = M_i(t) \tilde P_{i-1}(t) + a_i(t).\label{eq:garchstat2}
\end{equation}
Therefore, $\tilde P_i(t) = G_{\zeta_i}(\tilde P_{i-1}(t),t)$ with $G_{\zeta_i}(y,t) = M_i(t)\cdot y + a_i(t)$. Let $W_n(y,t) := G_{\zeta_n}(G_{\zeta_{n-1}}(...G_{\zeta_1}(y,t)...))$. Then we have
\[
    W_n(y,t) - W_n(y',t) = M_n(t) (W_n(y,t) - W_n(y',t)) = ... = M_{n-1}(t) \cdot ... \cdot M_1(t) \cdot (y-y').
\]
Using $(AB) \otimes (CD) = (A \otimes C)(B \otimes D)$, we obtain
\begin{equation}
    (W_n(y,t) - W_n(y',t))^{\otimes 2} = M_n(t)^{\otimes 2} (W_n(y,t) - W_n(y',t))^{\otimes 2} = M_{n-1}(t)^{\otimes 2} \cdot ... \cdot M_1(t)^{\otimes 2} \cdot (y-y')^{\otimes 2}\label{representation_garch_explicit_111}
\end{equation}
 Thus, we obtain from \reff{representation_garch_explicit_111} and \reff{product_rule_qnorm} that
\begin{eqnarray*}
    \|(W_n(y,t) - W_n(y',t))^{\otimes 2}\|_{\tilde q} &\le& \|M_{n-1}(t)^{\otimes 2}\|_{\tilde q} \cdot ... \|M_1(t)^{\otimes 2}\|_{\tilde q} (y-y')^{\otimes 2}\\
    &=& \big(\|M_1(t)^{\otimes 2}\|_{\tilde q}\big)^{n-1}\cdot (y-y')^{\otimes 2}.
\end{eqnarray*}
By \reff{largest_eigenvalue_garch_bounded1}, Theorem 2 in \cite{wushao2004} yields existence and a.s. uniqueness of $\tilde Y_i(t) = H(t,\sF_i)$, $\sup_{t\in[0,1]}\|\tilde Y_0(t)\|_{\tilde q} < \infty$ and $\sup_{t\in[0,1]}\delta_{\tilde q}^{\tilde Y(t)}(k) = \|\tilde Y_i(t) - \tilde Y_i(t)^{*}\|_{\tilde q} = O(c^k)$ for some $0<c<1$. 
This shows Assumption \ref{ass3}\ref{ass3_dep}.

We now aim to show \ref{ass3}\ref{ass3_stat}.
\reff{eq:garchstat2} implies the explicit representation
\begin{equation}
    \tilde P_i(t) = \sum_{k=0}^{\infty}\Big(\prod_{j=0}^{k-1}M_{i-j}(t)\Big) a_{i-k}(t).\label{eq:garchexplicit}
\end{equation}
We therefore have for $t,t'\in [0,1]$:
\begin{eqnarray*}
    \tilde P_i(t) - \tilde P_i(t') &=& \sum_{k=0}^{\infty}\sum_{l=0}^{k-1}C_{k,l} + \sum_{k=0}^{\infty}C_{k},
\end{eqnarray*}
where
\begin{eqnarray*}
    C_{k,l} &:=& \Big\{\Big(\prod_{j=0}^{l-1}M_{i-j}(t)\Big) \cdot \{M_{i-l}(t) - M_{i-l}(t')\}\cdot \Big(\prod_{j=l+1}^{k-1}M_{i-j}(t')\Big)\Big\}a_{i-k}(t)\\
    C_{k} &:=& \Big(\prod_{j=0}^{k-1}M_{i-j}(t')\Big)\cdot (a_{i-k}(t) - a_{i-k}(t')).
\end{eqnarray*}
We obtain that
\begin{eqnarray*}
    (\tilde P_i(t) - \tilde P_i(t'))^{\otimes 2} &=& \sum_{k,k'=0}^{\infty}\sum_{l=0}^{k-1}\sum_{l'=0}^{k'-1}C_{k,l}\otimes C_{k',l'} + \sum_{k,k'=0}^{\infty}\sum_{l=0}^{k-1}\{C_{k,l}\otimes C_{k'} + C_{k'} \otimes C_{k,l}\}\\
    &&\quad\quad\quad\quad + \sum_{k,k'=0}^{\infty}C_k \otimes C_{k'},
\end{eqnarray*}
in particular, it holds component-wise that
\begin{eqnarray}
    \|(\tilde P_i(t) - \tilde P_i(t'))^{\otimes 2}\|_{\tilde q} &\le& \sum_{k,k'=0}^{\infty}\sum_{l=0}^{k-1}\sum_{l'=0}^{k'-1}\|C_{k,l}\otimes C_{k',l'}\|_{\tilde q} + \sum_{k,k'=0}^{\infty}\sum_{l=0}^{k-1}\{\|C_{k,l}\otimes C_{k'}\|_{\tilde q} + \|C_{k'} \otimes C_{k,l}\|_{\tilde q}\}\nonumber\\
    &&\quad\quad\quad\quad + \sum_{k,k'=0}^{\infty}\|C_k \otimes C_{k'}\|_{\tilde q}.\label{largest_eigenvalue_garch_bounded_rep}
\end{eqnarray}
For the first summand in \reff{largest_eigenvalue_garch_bounded_rep}, we only investigate the case $l < l' < k < k'$. All other cases can be similar dealt with since similar terms arise. We have with $(AB)\otimes(CD) = (A \otimes C)(B\otimes D)$:
\begin{eqnarray*}
    C_{k,l}\otimes C_{k',l'} &=& \Big(\prod_{j=0}^{l-1}M_{i-j}(t)^{\otimes 2}\Big)\cdot \big(\{M_{i-l}(t) - M_{i-l}(t')\}\otimes M_{i-l}(t)\big)\cdot \Big(\prod_{j=l+1}^{l'-1}M_{i-j}(t')\otimes M_{i-j}(t)\Big)\\
    &&\quad\quad\quad\quad\times \big(M_{i-l'}(t')\otimes \{M_{i-l'}(t) - M_{i-l'}(t')\}\big)\cdot \Big(\prod_{j=l'+1}^{k-1}M_{i-j}(t')^{\otimes 2}\Big)\cdot \Big(\prod_{j=k}^{k'-1}M_{i-j}(t')\otimes I\Big)\\
    &&\quad\quad\quad\quad\times (a_{i-k}(t)\otimes a_{i-k'}(t))
\end{eqnarray*}
By independence and \reff{product_rule_qnorm}, it holds componentwise that
\begin{eqnarray*}
    \big\|C_{k,l}\otimes C_{k',l'}\big\|_{\tilde q} &\le& \|M_0(t)^{\otimes 2}\|_{\tilde q}^{l}\cdot \|\{M_{0}(t) - M_{0}(t')\}\otimes M_{0}(t)\|_{\tilde q}\cdot \|M_0(t')\otimes M_0(t)\|_{\tilde q}^{l'-l-1}\\
    &&\quad\quad\quad\quad\times \|M_{0}(t')\otimes \{M_{0}(t) - M_{0}(t')\}\|_{\tilde q}\cdot \|M_0(t')^{\otimes 2}\|_{\tilde q}^{k-l'-1}\cdot \|M_{0}(t')\otimes I\|_{\tilde q}^{k'-k}\\
    &&\quad\quad\quad\quad\times \|a_{0}(t) \otimes a_0(t)\|_{\tilde q}.
\end{eqnarray*}
By Lemma \ref{lemma_matrixnorm_householder} and \reff{largest_eigenvalue_garch_bounded0}, \reff{largest_eigenvalue_garch_bounded1}, \reff{largest_eigenvalue_garch_bounded2}, we obtain that with some $\tilde \rho \in (0,1)$ and some constant $\tilde c > 0$,
\begin{eqnarray*}
    \big|\big\|C_{k,l}\otimes C_{k',l'}\big\|_{\tilde q}\big|_1 &\le& \tilde c\cdot \tilde \rho^l \cdot \big| \|\{M_{0}(t) - M_{0}(t')\}\otimes M_{0}(t)\|_{\tilde q}\big|_1 \cdot \tilde \rho^{l'-l-1}\\
    &&\quad\quad\quad\quad\times\big|\|M_{0}(t')\otimes \{M_{0}(t) - M_{0}(t')\}\|_{\tilde q}\big|_1\cdot \tilde \rho^{k-l'-1}\cdot \tilde \rho^{k'-k}\\
    &=& \tilde c\cdot \tilde \rho^{k'-2}\cdot \|\{M_{0}(t) - M_{0}(t')\}\otimes M_{0}(t)\|_{\tilde q}\big|_1\cdot \big|\|M_{0}(t')\otimes \{M_{0}(t) - M_{0}(t')\}\|_{\tilde q}\big|_1
\end{eqnarray*}
By Lipschitz-continuity of $\theta(\cdot)$ and $\IE[|\zeta_1|^{4+a}] < \infty$, the Cauchy-Schwarz inequality yields
\begin{equation}
    \big|\|\{M_{0}(t) - M_{0}(t')\}\otimes M_{0}(t)\|_{\tilde q}\big|_1 \le \big|\big\|M_{0}(t) - M_{0}(t')\big\|_{2 \tilde q}\big|_1\cdot \big| \|M_{0}(t)\|_{2\tilde q}\big|_1 \le C\cdot |t-t'|\label{largest_eigenvalue_garch_bounded_rep_sum1_eq1}
\end{equation}
with some constant $C > 0$. A similar result for $\big|\|M_{0}(t')\otimes \{M_{0}(t) - M_{0}(t')\}\|_{\tilde q}\big|_1$ implies
\[
    \big|\big\|C_{k,l}\otimes C_{k',l'}\big\|_{\tilde q}\big|_1 \le \tilde c\cdot \tilde \rho^{k'-2}\cdot C^2\cdot  |t-t'|^2.
\]
We therefore have
\begin{equation}
    \Big|\sum_{0 \le l < l' < k < k' < \infty}\|C_{k,l}\otimes C_{k',l'}\|_{\tilde q}\Big|_1 \le \tilde c \cdot C^2 |t-t'|^2 \cdot \sum_{k'=0}^{\infty}(k')^3\cdot \tilde \rho^{k'-2}.\label{largest_eigenvalue_garch_bounded_rep_sum1}
\end{equation}
For the third summand in \reff{largest_eigenvalue_garch_bounded_rep}, we will only investigate the case $k < k'$. Here we have
\[
    C_{k}\otimes C_{k'} = \Big(\prod_{j=0}^{k-1}M_{i-j}(t')^{\otimes 2}\Big)\cdot \Big(\prod_{j=k}^{k'-1}\{M_{i-j}(t')\otimes I\}\Big)\cdot \{(a_{i-k}(t) - a_{i-k}(t'))\otimes (a_{i-k'}(t) - a_{i-k'}(t'))\}.
\]
As before we conclude that with some $\tilde c > 0, \tilde \rho \in (0,1)$ it holds that 
\[
    \big| \big\|C_{k} \otimes C_{k'}\big\|_{\tilde q}\big|_1 \le \tilde c\cdot \tilde \rho^{k}\cdot \tilde \rho^{k'-k}\cdot \big| \big\|(a_{0}(t) - a_{0}(t'))\otimes (a_{0}(t) - a_{0}(t'))\big\|_{\tilde q}\big|_1.
\]
By Lipschitz continuity of $\theta(\cdot)$ and $\IE[|\zeta_1|^{4+a}] < \infty$, the Cauchy-Schwarz inequality yields with some constant $C > 0$ that
\[
    \big| \big\|(a_{0}(t) - a_{0}(t'))\otimes (a_{0}(t) - a_{0}(t'))\big\|_{\tilde q}\big|_1 \le \big| \|a_{0}(t) - a_{0}(t')\|_{2\tilde q}\big|_1^2 \le C\cdot |t-t'|.
\]
We obtain that
\[
    \big| \big\|C_{k} \otimes C_{k'}\big\|_{\tilde q}\big|_1 \le \tilde c\cdot \tilde \rho^{k'}\cdot C^2 \cdot |t-t'|^2,
\]
and thus
\begin{equation}
    \Big|\sum_{0 \le k < k' < \infty}\big\|C_{k} \otimes C_{k'}\big\|_{\tilde q}\Big|_1 \le \tilde c C^2 |t-t'|^2\cdot \sum_{k'=0}^{\infty}k' \tilde \rho^{k'}.\label{largest_eigenvalue_garch_bounded_rep_sum3}
\end{equation}
The second summand in \reff{largest_eigenvalue_garch_bounded_rep} can be similar dealt with as the first and the third summand. We obtain from \reff{largest_eigenvalue_garch_bounded_rep_sum1} and \reff{largest_eigenvalue_garch_bounded_rep_sum3} that there exists some constant $\tilde c' > 0$ such that
\[
    \big| \|(\tilde P_i(t) - \tilde P_i(t'))^{\otimes 2}\|_{\tilde q}\big|_1 \le \tilde c'\cdot |t-t'|^2.
\]
Since $\tilde P_i(t)$ contains $\tilde Y_i(t)$ as its first element, we obtain in particular that
\[
    \|\tilde Y_i(t) - \tilde Y_i(t')\|_{2\tilde q}^2 \le \tilde c'|t-t'|^2,
\]
that is, the second part of  \ref{ass3}\ref{ass3_stat}, equation \reff{eq:ass1_stat1}.

We now discuss the first part of \ref{ass3}\ref{ass3_stat}, equation \reff{eq:ass1_stat1}. Let
\[
    P_i := (Y_i,...,Y_{i-m+1},\sigma_i^2,...,\sigma_{i-l+1}^2)\tran,
\]
then it holds that
\[
    P_i = M_i(i/n)\cdot P_{i-1} + a_i(i/n), \quad i = 1,...,n,
\]
with $P_0 = \tilde P_0(0)$ by definition. Thus $P_i$ is well-defined and $\|Y_i\|_{2\tilde q} < \infty$ exists. We therefore have the representation
\[
    P_i = \sum_{k=0}^{\infty}\Big(\prod_{j=0}^{k-1}M_i(\frac{i-j}{n} \vee 0)\Big)\cdot a_{i-k}(\frac{i-k}{n} \vee 0),
\]
which leads to
\begin{eqnarray*}
    (P_i - \tilde P_i(i/n))^{\otimes 2} &=& \Big[\sum_{k=0}^{\infty}\Big\{\Big(\prod_{j=0}^{k-1}M_{i-j}(\frac{i-j}{n} \vee 0)\Big) - \Big(\prod_{j=0}^{k-1}M_{i-j}(\frac{i}{n})\Big)\Big\}a_{i-k}(\frac{i-k}{n} \vee 0)\\
    &&\quad\quad\quad\quad + \sum_{k=0}^{\infty}\Big(\prod_{j=0}^{k-1}M_{i-j}(\frac{i}{n})\Big)\{a_{i-k}(\frac{i-k}{n} \vee 0) - a_{i-k}(\frac{i}{n})\}\Big]^{\otimes 2}\\
    &=& \sum_{k,k'=0}^{\infty}\sum_{l=0}^{k-1}\sum_{l'=0}^{k'-1}D_{k,l}\otimes D_{k',l'} + \sum_{k,k'=0}^{\infty}\sum_{l=0}^{k-1}\{D_{k,l}\otimes D_{k'} + D_{k'} \otimes D_{k,l}\}\\
    &&\quad\quad\quad\quad + \sum_{k,k'=0}^{\infty}D_k \otimes D_{k'},
\end{eqnarray*}
where
\begin{eqnarray*}
    D_{k,l} &:=& \Big\{\Big(\prod_{j=0}^{l-1}M_{i-j}(\frac{i-j}{n}\vee 0)\Big) \cdot \{M_{i-l}(\frac{i-l}{n}\vee 0) - M_{i-l}(\frac{i}{n})\}\cdot \Big(\prod_{j=l+1}^{k-1}M_{i-j}(\frac{i}{n})\Big)\Big\}a_{i-k}(\frac{i-k}{n} \vee 0)\\
    D_{k} &:=& \Big(\prod_{j=0}^{k-1}M_{i-j}(\frac{i-j}{n})\Big)\cdot (a_{i-k}(\frac{i-k}{n} \vee 0) - a_{i-k}(\frac{i}{n})).
\end{eqnarray*}
As in \reff{largest_eigenvalue_garch_bounded_rep}, we have component-wise that 
\begin{eqnarray}
    \|(P_i - \tilde P_i(\frac{i}{n}))^{\otimes 2}\|_{\tilde q} &\le& \sum_{k,k'=0}^{\infty}\sum_{l=0}^{k-1}\sum_{l'=0}^{k'-1}\|D_{k,l}\otimes D_{k',l'}\|_{\tilde q} + \sum_{k,k'=0}^{\infty}\sum_{l=0}^{k-1}\{\|D_{k,l}\otimes D_{k'}\|_{\tilde q} + \|D_{k'} \otimes D_{k,l}\|_{\tilde q}\}\nonumber\\
    &&\quad\quad\quad\quad + \sum_{k,k'=0}^{\infty}\|D_k \otimes D_{k'}\|_{\tilde q}.\label{largest_eigenvalue_garch_bounded_rep_local}
\end{eqnarray}
Since \reff{largest_eigenvalue_garch_bounded_rep_local} is nearly of the same structure as \reff{largest_eigenvalue_garch_bounded_rep}, we only discuss the new aspect coming in on the first summand of \reff{largest_eigenvalue_garch_bounded_rep_local} for $l < l' < k < k'$. We have
\begin{eqnarray*}
    &&D_{k,l}\otimes D_{k',l'}\\
    &=& \Big(\prod_{j=0}^{l-1}M_{i-j}(\frac{i-j}{n}\vee 0)^{\otimes 2}\Big)\cdot \big(\{M_{i-l}(\frac{i-l}{n}\vee 0) - M_{i-l}(\frac{i}{n})\}\otimes M_{i-l}(\frac{i}{n})\big)\\
    &&\quad\quad\quad\quad\times\Big(\prod_{j=l+1}^{l'-1}M_{i-j}(\frac{i}{n})\otimes M_{i-j}(\frac{i-j}{n}\vee 0)\Big)\cdot \big(M_{i-l'}(\frac{i}{n})\otimes \{M_{i-l'}(\frac{i-l'}{n}\vee 0) - M_{i-l'}(\frac{i}{n})\}\big)\\
    &&\quad\quad\quad\quad\times\Big(\prod_{j=l'+1}^{k-1}M_{i-j}(\frac{i}{n})^{\otimes 2}\Big)\cdot \Big(\prod_{j=k}^{k'-1}M_{i-j}(\frac{i}{n})\otimes I\Big)\\
    &&\quad\quad\quad\quad\times (a_{i-k}(\frac{i-k}{n}\vee 0)\otimes a_{i-k'}(\frac{i-k'}{n}\vee 0)).
\end{eqnarray*}
Using \reff{product_rule_qnorm} and independence, we obtain component-wise that
\begin{eqnarray}
    \big\|D_{k,l}\otimes D_{k',l'}\big\|_{\tilde q} &\le& \Big(\prod_{j=0}^{l-1}\|M_0(\frac{i-j}{n}\vee 0)^{\otimes 2}\|_{\tilde q}\Big)\cdot \|\{M_{0}(\frac{i-l}{n}\vee 0) - M_{0}(\frac{i}{n})\}\otimes M_{0}(\frac{i}{n})\|_{\tilde q}\nonumber\\
    &&\quad\quad\quad\quad\times\Big(\prod_{j=l+1}^{l'-1}\|M_0(\frac{i}{n})\otimes M_0(\frac{i-j}{n}\vee 0)\|_{\tilde q}\Big)\nonumber\\
    &&\quad\quad\quad\quad\times \|M_{0}(\frac{i}{n})\otimes \{M_{0}(\frac{i-l'}{n}\vee 0) - M_{0}(\frac{i}{n})\}\|_{\tilde q}\cdot \|M_0(\frac{i}{n})^{\otimes 2}\|_{\tilde q}^{k-l'-1}\nonumber\\
    &&\quad\quad\quad\quad\times \|M_{0}(\frac{i}{n})\otimes I\|_{\tilde q}^{k'-k} \cdot \big\|a_{0}(\frac{i-k}{n}) \otimes a_0(\frac{i-k'}{n})\big\|_{\tilde q}.\label{largest_eigenvalue_garch_bounded_rep_local_eq1}
\end{eqnarray}
Define $A(t,t') := \|M_0(t)\otimes M_0(t')\|_{\tilde q}$. From \reff{largest_eigenvalue_garch_bounded0}, we have $\rho_2 := \sup_{t,t' \in [0,1]}\rho(A(t,t')) < 1$. We can apply a straightforward generalization of Lemma \ref{lemma_matrixnorm} (with $[0,1]^2$ as domain of definition for $A$ instead of $[0,1]$) to obtain that there exists a finite set of invertible matrices $M_1,...,M_L$ and a finite partition $[0,1]^2 = \bigcup_{k=1}^{L}I_k$ into rectangles $I_k \subset [0,1]^2$ such that for $(t,t') \in I_k$,
\[
    |A(t,t')|_{M_k} \le \frac{1+\rho_2}{2} < 1.
\]
Similar as in the proof of Lemma \ref{lemma_randomiterativemodel}, equation \reff{lemma_randomiterativemodel_eq1} therein, this shows that there exists a constant $C = C(L) > 0$ such that
\[
    \Big|\prod_{j=0}^{l-1}\|M_0(\frac{i-j}{n}\vee 0)^{\otimes 2}\|_{\tilde q}\Big|_1 \le C(L)\cdot \big(\frac{1+\rho_2}{2}\big)^{l},\quad\quad \Big|\|M_0(\frac{i}{n})^{\otimes 2}\|_{\tilde q}^{k-l'-1}\Big|_1 \le C(L)\cdot \big(\frac{1+\rho_2}{2}\big)^{k-l'-1},
\]
and
\[
    \Big|\prod_{j=l+1}^{l'-1}\|M_0(\frac{i}{n})\otimes M_0(\frac{i-j}{n}\vee 0)\|_{\tilde q}\Big|_1 \le C(L)\cdot \big(\frac{1+\rho_2}{2}\big)^{l'-l-1},
\]
and as an implication of  \reff{largest_eigenvalue_garch_bounded2}, 
\[
    \Big|\|M_{0}(\frac{i}{n})\otimes I\|_{\tilde q}^{k'-k}\Big|_1 \le C(L)\cdot \big(\frac{1+\rho_2}{2}\big)^{k'-k}.
\]
Putting these results into \reff{largest_eigenvalue_garch_bounded_rep_local_eq1} yields
\begin{eqnarray}
    \big|\big\|D_{k,l} \otimes D_{k',l'}\big\|_{\tilde q}\big|_1 &\le& C(L)^4 \big(\frac{1+\rho_2}{2}\big)^{k'-2}\cdot \big|\|\{M_{0}(\frac{i-l}{n}\vee 0) - M_{0}(\frac{i}{n})\}\otimes M_{0}(\frac{i}{n})\|_{\tilde q}\big|_1\nonumber\\
    &&\quad\quad\quad\quad\times \big| \|M_{0}(\frac{i}{n})\otimes \{M_{0}(\frac{i-l'}{n}\vee 0) - M_{0}(\frac{i}{n})\}\|_{\tilde q}\big|_1\nonumber\\
    &&\quad\quad\quad\quad\times \big|\big\|a_{0}(\frac{i-k}{n}) \otimes a_0(\frac{i-k'}{n})\big\|_{\tilde q} \big|_1.\label{largest_eigenvalue_garch_bounded_rep_local_eq2}
\end{eqnarray}
As in \reff{largest_eigenvalue_garch_bounded_rep_sum1_eq1}, we obtain
\begin{eqnarray*}
    \big|\|\{M_{0}(\frac{i-l}{n}\vee 0) - M_{0}(\frac{i}{n})\}\otimes M_{0}(\frac{i}{n})\|_{\tilde q}\big|_1 &\le& C\cdot \frac{l}{n},\\
    \big| \|M_{0}(\frac{i}{n})\otimes \{M_{0}(\frac{i-l'}{n}\vee 0) - M_{0}(\frac{i}{n})\}\|_{\tilde q}\big|_1 &\le& C\cdot \frac{l'}{n}
\end{eqnarray*}
with some constant $C > 0$ independent of $i,l',l,k,k',n$, and
\[
    \big|\big\|a_{0}(\frac{i-k}{n}) \otimes a_0(\frac{i-k'}{n})\big\|_{\tilde q} \big|_1 \le C.
\]
Insertion into \reff{largest_eigenvalue_garch_bounded_rep_local_eq2} and using $l <l' < k < k'$ yields
\[
    \big|\big\|D_{k,l} \otimes D_{k',l'}\big\|_{\tilde q}\big|_1 \le C(L)^4 C^3 \big(\frac{1+\rho_2}{2}\big)^{k'-2}\cdot \frac{(k')^2}{n^2},
\]
thus
\[
    \Big|\sum_{0 \le l < l' < k < k' < \infty}\big\|D_{k,l} \otimes D_{k',l'}\big\|_{\tilde q}\Big|_1 \le \frac{C(L)^4 C^3}{n^2}\cdot \sum_{k'=0}^{\infty}\big(\frac{1+\rho_2}{2}\big)^{k'-2}\cdot (k')^5.
\]
Similar calculations for all other possibilities and summands in \reff{largest_eigenvalue_garch_bounded_rep_local} show that
\[
    \big|\big\|(\tilde P_i - \tilde P_i(i/n))^{\otimes 2}\big\|_{\tilde q}\big|_1 \le \frac{D}{n}
\]
with some $D > 0$ independent of $i,n$. We conclude that
\[
    \big\|Y_i - \tilde Y_i(i/n)\big\|_{\tilde q} \le \frac{D}{n},
\]
which shows the first part of \ref{ass3}\ref{ass3_stat}, equation \reff{eq:ass1_stat1}.

Let $\Sigma(x,\theta) := (\sigma(x,\theta)^2,\ldots,\sigma(x_{(l-1)\rightarrow},\theta)^2)\tran$ and $A(x,\theta) := (\alpha_0 + \sum_{j=1}^{m}\alpha_j x_{j},\ldots,\alpha_0 + \sum_{j=1}^{m}\alpha_j x_{j+l-1})\tran$, and
\[
    B(\theta) = \begin{pmatrix}
        \beta_1 & \dots & \dots & \dots & \beta_l\\
        1 & 0 & \dots & \dots & 0\\
        0 & \ddots & \ddots &  & \vdots\\
        \vdots & \ddots & \ddots & \ddots & 0\\
        0 & \dots & 0 & 1 & 0
    \end{pmatrix}.
\]
As said in Theorem 2.1 in \cite{momentgarch}, $\rho(\IE M_0(\theta)^{\otimes 2}) < 1$ is a necessary and sufficient condition for the corresponding GARCH process with parameters $\theta$ to have 4-th moments. We conclude that $\rho(\IE M_0(\theta)) < 1$ which by Proposition 1 in \cite{garch2004} implies $\rho(B(\theta)) < 1$. We have the explicit representation
\begin{equation}
    \sigma(x,\theta)^2 = \sum_{k=0}^{\infty}\big(B(\theta)^k A(x_{k\rightarrow},\theta)\big)_1.\label{garch:sigma_explicit}
\end{equation}
Since $A(0,\theta) = (\alpha_0,0,...,0)\tran$, we have
\[
\sigma(0,\theta)^2 = \alpha_0\sum_{k=0}^{\infty}(B(\theta)^k)_{11}.
\]
From \reff{garch:sigma_explicit} we also obtain that
\begin{equation}
        \sigma(x,\theta)^2 = c_0(\theta) + \sum_{j=1}^{\infty}c_j(\theta)\cdot x_{j},\label{garch:sigma_explicit2}
    \end{equation}
    where $c_j(\theta) \ge 0$ satisfies
    \begin{equation}
        \sup_{\theta\in\Theta}|c_j(\theta)| \le C \cdot \rho^j\label{eq:garchgeometricdecaycoefficients}
    \end{equation}
    with some $\rho \in (0,1)$ and $c_0(\theta) \ge \sigma_{min}^2 > 0$ (due to $\alpha_0 \ge \alpha_{min} > 0$). Due to the explicit representation \reff{garch:sigma_explicit} with geometrically decaying summands, it is easy to see that $\sigma(x,\theta)^2$ is four times continuously differentiable w.r.t. $\theta$ with
    \begin{equation}
        \nabla_{\theta}^k(\sigma(x,\theta)^2) = \nabla_{\theta}^k c_0(\theta) + \sum_{j=1}^{\infty}\nabla_{\theta}^k c_j(\theta)\cdot x_{j},\quad\quad k \in \{0,1,2,3,4\},\label{eq:explicitsigmarepresentation}
    \end{equation}
    where $(\nabla_{\theta}^k c_j(\theta))_{j}$ is still geometrically decaying with $\sup_{\theta \in \Theta}|\nabla_{\theta}^k c_j(\theta)|_{\infty} \le C \cdot \rho^j$, say.\\
    
    From \reff{eq:explicitsigmarepresentation} we conclude that (component-wise) for $k = 0,1,2,3$:
    \begin{eqnarray}
        |\nabla_{\theta}^k(\sigma(x,\theta)^2) - \nabla_{\theta}^k(\sigma(x',\theta)^2)| &\le& C |x - x'|_{(\rho^j)_j,1},\nonumber\\
        \label{eq:sigmagarch_xdiff}\\
        |\nabla_{\theta}^k(\sigma(x,\theta)^2) - \nabla_{\theta}^k(\sigma(x,\theta')^2)| &\le& |\theta - \theta'|_1 \cdot \sup_{\theta \in \Theta}|\nabla_{\theta}^{k+1}(\sigma(x,\theta)^2)|_{\infty} \le C |\theta - \theta'|_1 \cdot |x|_{(\rho^j)_j,1}.\nonumber\\\label{eq:sigmagarch_thetadiff}
    \end{eqnarray}
    
    We obtain that $\ell(y,x,\theta)$ is four times continuously differentiable and
\begin{eqnarray*}
    \ell(y,x,\theta) &=& \frac{1}{2}\Big(\frac{y}{\sigma(x,\theta)^2} + \log(\sigma(x,\theta)^2)\Big),\\
    \nabla_{\theta}\ell(y,x,\theta) &=& \frac{\nabla_{\theta}(\sigma(x,\theta)^2)}{2\sigma(x,\theta)^2}\Big(1-\frac{y}{\sigma(x,\theta)^2}\Big),\\
    \nabla_{\theta}^2 \ell(y,x,\theta) &=& \Big[-\frac{\nabla_{\theta}(\sigma(x,\theta)^2) \nabla_{\theta}(\sigma(x,\theta)^2)\tran}{2 \sigma(x,\theta)^4} +\frac{\nabla_{\theta}^2(\sigma(x,\theta)^2)}{2\sigma(x,\theta)^2} \Big]\Big(1-\frac{y}{\sigma(x,\theta)^2}\Big)\\
    &&\quad\quad\quad + \frac{\nabla_{\theta}(\sigma(x,\theta)^2) \nabla_{\theta}(\sigma(x,\theta)^2)\tran}{2\sigma(x,\theta)^4}\cdot \frac{y}{\sigma(x,\theta)^2}.
\end{eqnarray*}
It was shown in the proof of Theorem 2.1 in \cite{garch2004}, that $\theta \mapsto L(t,\theta) = \IE \ell(\tilde Z_0(t),\theta)$ is uniquely minimized in $\theta = \theta(t)$, which shows Assumption \ref{ass3}\ref{ass3_model}. As in the proof of Proposition \ref{example:tvrec}, we obtain that
\begin{eqnarray*}
    V(t) = \IE\big[\frac{\nabla_{\theta} (\sigma(\tilde X_0(t),\theta(t))^2) \nabla_{\theta}(\sigma(\tilde X_0(t),\theta(t))^2)\tran}{2\sigma(\tilde X_0(t),\theta(t))^4}\big] = I(t) \frac{2}{\IE \zeta_0^4-1}.
\end{eqnarray*}
Furthermore,
\[
    \nabla_{\theta}\ell(\tilde Z_i(t),\theta(t)) = \frac{\nabla_{\theta}(\sigma(\tilde X_i(t),\theta(t))^2)}{2\sigma(\tilde X_i(t),\theta(t))^2}\{1 - \zeta_i^2\},
\]
which shows that $\nabla_{\theta}\ell(\tilde Z_i(t),\theta(t))$ is a martingale difference sequence w.r.t. $\sF_i$. Thus $\Lambda(t) = I(t)$. It was shown in the proof of Theorem 2.2 in \cite{garch2004} that $V(t)$ is positive definite for each $t \in [0,1]$. By continuity, we conclude that Assumption \ref{ass3}\ref{ass3_matrix} is fulfilled.

\textbf{Proof of Assumption \ref{ass3}\ref{ass3_smooth}:} It holds that
    \[
        2|\ell(y,x,\theta) - \ell(y',x',\theta)| \le |y-y'|\cdot \frac{1}{\sigma(x,\theta)^2} + |y'|\cdot \Big|\frac{1}{\sigma(x,\theta)^2} - \frac{1}{\sigma(x',\theta)^2}\Big| + |\log( \sigma(x,\theta)^2) - \log(\sigma(x',\theta)^2)|.
    \]
    Since $\sigma(x,\theta)^2 \ge \sigma_{min}^2 > 0$, Lipschitz continuity of $\log$ on $[\sigma_{min},\infty)$ and \reff{eq:garchgeometricdecaycoefficients}, there exists some constant $C' > 0$ such that
    \begin{equation}
       2|\ell(y,x,\theta) - \ell(y',x',\theta)| \le C'(|y-y'| + |x-x'|_{(\rho^j)_j,1}) + |y'|\cdot \Big|\frac{1}{\sigma(x,\theta)^2} - \frac{1}{\sigma(x',\theta)^2}\Big|.\label{eq:garchlikelihood_upperbound}
    \end{equation}
    Note that
    \begin{eqnarray*}
        && \Big|\frac{1}{\sigma(x,\theta)^2} - \frac{1}{\sigma(x',\theta)^2}\Big| \le \frac{\sum_{j=0}^{\infty}c_j(\theta)|x_j . x_j'|}{\sigma(x,\theta)^2 \sigma(x',\theta)^2} \le \sum_{j=1}^{\infty}\frac{c_j(\theta)|x_j - x_j'|}{(\sigma_{min}^2 + c_j(\theta) x_j)(\sigma_{min}^2 + c_j(\theta) x_j')}\\
        &\le& \frac{1}{\sigma_{min}^2}\sum_{j=1}^{\infty}\frac{c_j(\theta)|x_j - x_j'|}{\sigma_{min}^2 + c_j(\theta) |x_j - x_j'|.}
    \end{eqnarray*}
    The last step holds due to the following argument: It holds either $|x_j - x_j'| \le x_j$ or $|x_j - x_j'| \le x_j'$ since $x_j,x_j' \ge 0$. Therefore, one factor in the denominator can be lower bounded by $\sigma_{min}$ and the other one by $\sigma_{min} + c_j(\theta) |x_j - x_j'|$.
Following the ideas of \cite{garch2004}, for arbitrarily small $s > 0$ we use the inequality $\frac{x}{1+x} \le x^s$ to obtain
    \begin{eqnarray*}
        && \Big|\frac{1}{\sigma(x,\theta)^2} - \frac{1}{\sigma(x',\theta)^2}\Big| \le \frac{1}{\sigma_{min}^{4+2s}}\sum_{j=1}^{\infty}c_j(\theta)^s|x_j - x_j'|^s \le \frac{C}{\sigma_{min}^{4+2s}}|x-x|_{(\rho^{js})_j,s}^s
    \end{eqnarray*}
Together with \reff{eq:garchlikelihood_upperbound}, we obtain \reff{eq:additionalgarchlipschitz}.

Using directly \reff{eq:garchlikelihood_upperbound} and \reff{eq:sigmagarch_xdiff}, we obtain
\[
    \sup_{\theta \in \Theta}\sup_{z\not=z'}\frac{|\ell(z,\theta) - \ell(z',\theta)|}{|z-z'|_{(\rho^j)_j,1}\cdot (1 + |z|_{(\rho^j)_j}^{2M-1} + |z'|_{(\rho^j)_j}^{2M-1})} < \infty.
\]
Note that with some constant $C' > 0$,
\begin{eqnarray*}
    2|\ell(z,\theta) - \ell(z,\theta')| &\le& |y|\cdot \Big[\frac{1}{\sigma(x,\theta)^2} - \frac{1}{\sigma(x,\theta')^2}\Big] + | \log(\sigma(x,\theta)^2) - \log(\sigma(x,\theta')^2)|\\
    &\le& C' (1+|y|)\cdot |\sigma(x,\theta)^2 - \sigma(x,\theta')^2|.
\end{eqnarray*}
Together with \reff{eq:sigmagarch_thetadiff}, we obtain
\[
    \sup_{\theta\not=\theta'}\sup_{z}\frac{|\ell(z,\theta) - \ell(z,\theta')|}{|\theta-\theta'|_{1}\cdot (1 + |z|_{(\rho^j)_j}^{2M} + |z'|_{(\rho^j)_j}^{2M})} < \infty.
\]
This shows $\ell \in \sH(2M, (\rho^j)_j,\bar C)$ with some suitably chosen $\bar C > 0$.

Let $s > 0$  be arbitrary. It was shown in \cite{garch2004}, (4.25) therein that with some small $\iota > 0$ only depending on $s,\Theta$, it holds that
\begin{equation}
    \sup_{|\tilde \theta - \theta| < \iota}\frac{\sigma(x,\tilde \theta)^2}{\sigma(x,\theta)^2} \le \bar C(1+ |x|_{(\rho^{js})_j,s}^s).\label{eq:garchsigmaproperty1}
\end{equation}
Similarly, one can obtain for $k = 1,2,3$ that
\begin{equation}
    \sup_{|\tilde \theta - \theta| < \iota}\frac{\nabla_{\theta}^k \sigma(x,\tilde \theta)^2}{\sigma(x,\theta)^2} \le \bar C(1+ |x|_{(\rho^{js})_j,s}^s).\label{eq:garchsigmaproperty2}
\end{equation}
In the following we show that $\nabla_{\theta}\ell \in \sH(2M(1+s),(\rho^j)_j,\bar C)$ with some suitably chosen $\bar C > 0$. We have (component-wise):
\begin{eqnarray*}
    && 2|\nabla_{\theta}\ell(y,x,\theta) - \nabla_{\theta}\ell(y',x',\theta)|\\
    &\le& |y - y'| \cdot \frac{1}{\sigma_{min}^2}\frac{|\nabla_{\theta}(\sigma(x,\theta)^2)|}{\sigma(x,\theta)^2} + |y'|\cdot \Big|\frac{1}{\sigma(x,\theta)^2} - \frac{1}{\sigma(x',\theta)}\Big| \cdot \frac{|\nabla_{\theta}(\sigma(x,\theta)^2)|}{\sigma(x,\theta)^2}\\
    &&\quad + \Big(1 + \frac{|y'|}{\sigma_{min}^2}\Big)\cdot \Big(\frac{|\nabla_{\theta}(\sigma(x,\theta)^2) - \nabla_{\theta}(\sigma(x',\theta)^2)|}{\sigma_{min}^2} + \frac{|\nabla_{\theta}(\sigma(x',\theta)^2)|}{\sigma(x',\theta)^2\sigma_{min}^2}|\sigma(x,\theta)^2 - \sigma(x,\theta')^2|\Big).
\end{eqnarray*}
Using \reff{eq:sigmagarch_xdiff} and \reff{eq:garchsigmaproperty2}, we obtain (component-wise) with some suitably chosen $\bar C > 0$:
\begin{equation}
    2|\nabla_{\theta}\ell(y,x,\theta) - \nabla_{\theta}\ell(y',x',\theta)| \le \bar C |z-z'|_{(\rho^j)_j,1}\cdot (1 + |z|_{(\rho^j)_j}^{2M-1} + |z'|_{(\rho^j)_j}^{2M-1})^{1+s}.\label{eq:garchprooflikelihood1}
\end{equation}
We have (component-wise):
\begin{eqnarray*}
    && 2|\nabla_{\theta}\ell(z,\theta) - \nabla_{\theta}\ell(z,\theta')|\\
    &\le& |y|\cdot \Big|\frac{1}{\sigma(x,\theta)^2} - \frac{1}{\sigma(x,\theta')}\Big| \cdot \frac{|\nabla_{\theta}(\sigma(x,\theta)^2)|}{\sigma(x,\theta)^2}\\
    &&\quad + \Big(1 + \frac{|y|}{\sigma_{min}^2}\Big)\cdot \Big(\frac{|\nabla_{\theta}(\sigma(x,\theta)^2) - \nabla_{\theta}(\sigma(x,\theta')^2)|}{\sigma_{min}^2} + \frac{|\nabla_{\theta}(\sigma(x,\theta')^2)|}{\sigma(x,\theta')^2\sigma_{min}^2}|\sigma(x,\theta)^2 - \sigma(x,\theta')^2|\Big).
\end{eqnarray*}
Using \reff{eq:sigmagarch_xdiff} and \reff{eq:garchsigmaproperty2}, we obtain (component-wise) with some suitably chosen $\bar C > 0$:
\begin{equation}
    2|\nabla_{\theta}\ell(z,\theta) - \nabla_{\theta}\ell(z,\theta')| \le \bar C |\theta - \theta'|_1 \cdot (1 + |z|_{(\rho^j)_j}^{2M} + |z'|_{(\rho^j)_j}^{2M})^{1+s}.\label{eq:garchprooflikelihood2}
\end{equation}
We conclude from \reff{eq:garchprooflikelihood1} and \reff{eq:garchprooflikelihood2} that $\nabla_{\theta}\ell \in \sH(2M(1+s),(\rho^j)_j,\bar C)$.
The proof for $\nabla_{\theta}^2\ell$ is similar in view of \reff{eq:sigmagarch_xdiff}, \reff{eq:sigmagarch_thetadiff} and  \reff{eq:garchsigmaproperty2} and therefore omitted.

Let $s > 0$ be arbitrary and $\iota > 0$ such that \reff{eq:garchsigmaproperty1} and \reff{eq:garchsigmaproperty2} hold. In the following we show that $\nabla_{\theta}\tilde \ell \in \sH_{\iota}^{mult}(M(1+s),(\rho^j)_j,\bar C)$ with some suitable chosen $\bar C > 0$. It holds that
\[
    \nabla_{\theta}\tilde \ell_{\tilde \theta}(y,x,\theta) = \frac{\nabla_{\theta}(\sigma(x,\theta)^2)}{2\sigma(x,\theta)^2}\Big(1-y\frac{\sigma(x,\tilde \theta)^2}{\sigma(x,\theta)^2}\Big).
\]
We have for $|\theta - \tilde \theta|_1 < \iota$:
\begin{eqnarray*}
    && 2|\nabla_{\theta}\tilde \ell_{\tilde \theta}(y,x,\theta) - \nabla_{\theta}\tilde \ell_{\tilde \theta}(y,x',\theta)|\\
    &\le& |y|\cdot \Big[\frac{|\sigma(x,\tilde \theta)^2 - \sigma(x',\tilde \theta)^2|}{\sigma_{min}^2} + \frac{\sigma(x',\tilde \theta)^2}{\sigma(x',\theta)^2 \sigma_{min}^2}|\sigma(x,\theta)^2 - \sigma(x',\theta)^2|\Big] \cdot \frac{|\nabla_{\theta}(\sigma(x,\theta)^2)|}{\sigma(x,\theta)^2}\\
    &&\quad + \Big(1 + |y|\cdot \frac{\sigma(x',\tilde \theta)^2}{\sigma(x',\theta)^2}\Big)\cdot \Big(\frac{|\nabla_{\theta}(\sigma(x,\theta)^2) - \nabla_{\theta}(\sigma(x',\theta)^2)|}{\sigma_{min}^2} + \frac{|\nabla_{\theta}(\sigma(x',\theta)^2)|}{\sigma(x',\theta)^2\sigma_{min}^2}|\sigma(x,\theta)^2 - \sigma(x,\theta')^2|\Big).
\end{eqnarray*}
Using \reff{eq:sigmagarch_xdiff} and \reff{eq:garchsigmaproperty2}, we obtain (component-wise) with some suitably chosen $\bar C > 0$:
\begin{equation}
    2|\nabla_{\theta}\tilde \ell_{\tilde \theta}(y,x,\theta) - \nabla_{\theta}\tilde \ell_{\tilde \theta}(y,x',\theta)| \le \bar C (1+|y|)\cdot |x-x'|_{(\rho^j)_j,1} \cdot |x|_{(\rho^{js})_j,s}^s.\label{eq:garchprooflikelihood3}
\end{equation}
We have for $|\theta - \tilde \theta|_1, |\theta' - \tilde \theta|_1 < \iota$:
\begin{eqnarray*}
    && 2|\nabla_{\theta}\tilde \ell_{\tilde \theta}(y,x,\theta) - \nabla_{\theta}\tilde \ell_{\tilde \theta}(y,x,\theta')|\\
    &\le& |y|\cdot \frac{\sigma(x,\tilde \theta)^2}{\sigma(x,\theta)^2 \sigma_{min}^2}|\sigma(x,\theta)^2 - \sigma(x,\theta')^2| \cdot \frac{|\nabla_{\theta}(\sigma(x,\theta)^2)|}{\sigma(x,\theta)^2}\\
    &&\quad + \Big(1 + |y|\cdot \frac{\sigma(x,\tilde \theta)^2}{\sigma(x,\theta')^2}\Big)\cdot \Big(\frac{|\nabla_{\theta}(\sigma(x,\theta)^2) - \nabla_{\theta}(\sigma(x,\theta')^2)|}{\sigma_{min}^2} + \frac{|\nabla_{\theta}(\sigma(x,\theta')^2)|}{\sigma(x,\theta')^2\sigma_{min}^2}|\sigma(x,\theta)^2 - \sigma(x,\theta')^2|\Big).
\end{eqnarray*}
Using \reff{eq:sigmagarch_thetadiff} and \reff{eq:garchsigmaproperty2}, we obtain (component-wise) with some suitably chosen $\bar C > 0$:
\begin{equation}
    2|\nabla_{\theta}\tilde \ell_{\tilde \theta}(y,x,\theta) - \nabla_{\theta}\tilde \ell_{\tilde \theta}(y,x,\theta')| \le \bar C (1+|y|)\cdot |\theta-\theta'|_{1} \cdot (1+|x|_{(\rho^j)_j}^{M}) \cdot |x|_{(\rho^{js})_j,s}^s.\label{eq:garchprooflikelihood4}
\end{equation}
We conclude from \reff{eq:garchprooflikelihood3} and \reff{eq:garchprooflikelihood4} that $\nabla_{\theta}\tilde\ell \in \sH_{\iota}^{mult}(M(1+s),(\rho^j)_j,\bar C)$.
The proof for $\nabla_{\theta}^2\tilde\ell$ is similar in view of \reff{eq:sigmagarch_xdiff}, \reff{eq:sigmagarch_thetadiff} and  \reff{eq:garchsigmaproperty2} and therefore omitted.

\end{proof}

\end{document}